%% file: Main.tex
\pdfoutput=1

\RequirePackage{fix-cm}

\documentclass[
	,12pt%
	,pagesize%
	,headings=big%
	,paper=a4%
	,parskip=false%
	,headsepline=true%
	,abstracton
	,toc=bibliography
]{scrartcl}
\addtokomafont{sectioning}{\normalfont\bfseries}
\setkomafont{title}{\normalfont\LARGE}
\setkomafont{subtitle}{\normalfont\Large}
\addtokomafont{pageheadfoot}{\scshape\small}

\usepackage[utf8]{inputenc}
\usepackage[T1]{fontenc}
\usepackage{csquotes}
\usepackage[english]{babel}
\usepackage{graphicx}
\graphicspath{{Pictures/}{/}}
\usepackage[usenames,dvipsnames]{xcolor}
\usepackage{tikz-cd}
\usetikzlibrary{backgrounds}

\usepackage[draft]{fixme}
\usepackage[nointlimits]{amsmath}
\usepackage{amssymb,mathrsfs,mathtools}
\usepackage{txfonts}
\usepackage{upgreek}
\usepackage{braket}
\usepackage{cancel}
\usepackage{siunitx}
\usepackage{aliascnt} 

\usepackage[amsmath,thmmarks,thref]{ntheorem}
\usepackage[expansion=true,protrusion=true]{microtype}
\usepackage{xkeyval}
\usepackage{multirow}

\usepackage[
	backend=bibtex%
	,isbn=false%
	,doi=false%
	,url=false%
	,giveninits=true%
	]{biblatex}
\usepackage[final,%
	bookmarks,%
	bookmarksdepth=3,%
	breaklinks=true,%
	colorlinks=true,%
	urlcolor=RoyalBlue,%
	linkcolor=RoyalBlue,%
	citecolor=ForestGreen,%
]{hyperref}%

\usepackage[all]{hypcap}

\addto\extrasenglish{%
}

\input{macros}
\input{ntheoremenglish}

\begin{document}
\title{Variational Convergence\\ of Discrete Elasticae}
\author{%
Sebastian Scholtes\thanks{
\href{mailto:sebastian.scholtes@rwth-aachen.de}{sebastian.scholtes@rwth-aachen.de}
}
\and
Henrik Schumacher\thanks{
\href{schumacher@instmath.rwth-aachen.de}{schumacher@instmath.rwth-aachen.de}
}
\and
Max Wardetzky\thanks{
\href{mailto:wardetzky@math.uni-goettingen.de}{wardetzky@math.uni-goettingen.de}
}
}

\maketitle

\begin{abstract}
We discuss a discretization by polygonal lines of the Euler-Bernoulli bending energy
and of Euler elasticae under clamped boundary conditions.
We show Hausdorff convergence 
of the set of almost minimizers of the discrete bending energy
to the set of smooth Euler elasticae under mesh refinement in (i) the $W^{1,\infty}$-topology 
for piecewise-linear interpolation and in (ii) the $W^{2,p}$-topology, $p \in \intervalco{2,\infty}$, using a suitable smoothing operator to create $\Sobo[2,p]$-curves from polygons.
\end{abstract}

\input{Introduction}

\clearpage
\input{SmoothSetting}

\clearpage
\input{DiscreteSetting}

\clearpage
\input{Reconstruction}

\clearpage
\input{Sampling}

\clearpage
\appendix

\input{NormEstimates}

\input{EulerBernoulliLipschitz}

\input{ThetaLemma}

\input{DDPhi}

\input{Kantorovich}

\input{SamplingEstimates}

\newpage

\printbibliography

\end{document}

%% file: macros.tex
\newcommand{\mymathcal}[1]{\mathcal{#1}}
\DeclareMathAlphabet{\mathcal}{OMS}{cmsy}{m}{n}

\newcommand{\Angle}[2]{\measuredangle(#1,#2)}
\newcommand{\bigAngle}[2]{\measuredangle \big( #1,#2 \big)}

\newcommand{\LogStrain}{\sigma\!}
\newcommand{\StrainBound}{\varLambda}

\newcommand{\apriori}{a\,priori\xspace}

\newcommand{\DiscD}{\cD_\Polygon}
\newcommand{\DiscDnor}{\cD_\Polygon^\perp}

\newcommand{\DiscLaplacian}{\Laplacian_\Polygon}

\newcommand{\prnor}{\pr_\Curve^\perp}

\newcommand{\Discprnor}{\pr_\Polygon^\perp}

\newcommand{\LeftShift}{\downarrow}
\newcommand{\RightShift}{\uparrow}

\newcommand{\shiftl}[1]{#1^{\LeftShift}}
\newcommand{\shiftll}[1]{#1^{\LeftShift\!\LeftShift}}

\newcommand{\shiftr}[1]{#1^{\RightShift}}
\newcommand{\shiftrr}[1]{#1^{\RightShift\!\RightShift}}

\newcommand{\VthisV}[1]{#1^{\phantom{}}}
\newcommand{\VnextV}[1]{#1^{\RightShift\!\RightShift}}
\newcommand{\VprevV}[1]{#1^{\LeftShift\!\LeftShift}}

\newcommand{\EnextE}[1]{#1^{\RightShift\!\RightShift}}
\newcommand{\EprevE}[1]{#1^{\LeftShift\!\LeftShift}}

\newcommand{\EnextV}[1]{#1^{\RightShift}}
\newcommand{\EprevV}[1]{#1^{\LeftShift}}
\newcommand{\VnextE}[1]{#1^{\RightShift}}
\newcommand{\VprevE}[1]{#1^{\LeftShift}}

\DeclareDocumentCommand{\sobo}{ O{} O{} o O{\R}}{
	\IfValueTF{#3}{
	  w^{#1}_{#2}(#3;#4)
	}{
	  w^{#1}_{#2}
	}
}

\DeclareDocumentCommand{\imm}{ O{} O{} o O{\R}}{
	\IfValueTF{#3}{
	  \operatorname{imm}^{#1}_{#2}(#3;#4)
	}{
	  \operatorname{imm}^{#1}_{#2}
	}
}

\DeclareDocumentCommand{\bv}{ O{} O{} o O{\R}}{
	\IfValueTF{#3}{
	  bv^{#1}_{#2}(#3;#4)
	}{
	  bv^{#1}_{#2}
	}
}

\DeclareDocumentCommand{\tv}{ O{} O{} o O{\R}}{
	\IfValueTF{#3}{
	  tv^{#1}_{#2}(#3;#4)
	}{
	  tv^{#1}_{#2}
	}
}

\DeclareDocumentCommand{\Sobo}{ O{} O{} o O{\R}}{
	\IfValueTF{#3}{
	  W^{#1}_{#2}(#3;#4)
	}{
	  W^{#1}_{#2}
	}
}
\DeclareDocumentCommand{\SoboC}{ O{} O{} }{\Sobo[#1][#2][\Interval][\AmbSpace]}

\DeclareDocumentCommand{\Holder}{ O{} O{} o O{\R}}{
	\IfValueTF{#3}{
	  C^{#1}_{#2}(#3;#4)
	}{
	  C^{#1}_{#2}
	}
}
\DeclareDocumentCommand{\HolderC}{ O{} O{} }{\Holder[#1][#2][\Interval][\AmbSpace]}

\DeclareDocumentCommand{\BV}{ O{} O{} o O{\R}}{
	\IfValueTF{#3}{
	  BV^{#1}_{#2}(#3;#4)
	}{
	  BV^{#1}_{#2}
	}
}
\DeclareDocumentCommand{\BVC}{ O{} O{} }{\BV[#1][#2][\Interval][\AmbSpace]}

\DeclareDocumentCommand{\TV}{ O{} O{} o O{\R}}{
	\IfValueTF{#3}{
	  TV^{#1}_{#2}(#3;#4)
	}{
	  TV^{#1}_{#2}
	}
}

\DeclareDocumentCommand{\Imm}{ O{} O{} o O{\AmbSpace}}{
	\IfValueTF{#3}{
	  \operatorname{Imm}^{#1}_{#2}(#3;#4)
	}{
	  \operatorname{Imm}^{#1}_{#2}
	}
}
\DeclareDocumentCommand{\ImmC}{ O{} O{} }{\Imm[#1][#2][\Interval][\AmbSpace]}

\DeclareDocumentCommand{\Diff}{ O{} O{} }{\operatorname{Diff}^{#1}_{#2}}

\newcommand{\Triangulation}{T}
\newcommand{\textstyleswitcha}{\textstyle}
\newcommand{\OpenBall}[2]{B(#1;#2)}
\newcommand{\ClosedBall}[2]{\bar B(#1;#2)}
\newcommand{\bigClosedBall}[2]{\bar B \,\bigparen{#1;#2}}

\newcommand{\Dtot}{\cD_\Curve}
\newcommand{\Dtan}{\Dtot^\top}
\newcommand{\Dnor}{\Dtot^\perp}
\newcommand{\Tangent}{\tau}
\newcommand{\TangentC}{\Tangent_\Curve}
\newcommand{\Curvature}{\kappa}
\newcommand{\CurvatureC}{\Curvature_\Curve}
\newcommand{\SignedDistC}[2]{s_{\Curve}(#1,#2)}
\DeclareDocumentCommand{\SignedDistCp}{ o O{} O{}}{
	\IfValueTF{#1}{
	  s_{\Curve}(#2,#3)^{#1}
	}{
	  s_{\Curve}(#2,#3)
	}
}

\newcommand{\UnsignedDistC}[2]{d_\Curve(#1,#2)}
\DeclareDocumentCommand{\UnsignedDistCp}{ o O{} O{}}{
	\IfValueTF{#1}{
	  d_{\!\Curve}(#2,#3)^{#1}
	}{
	  d_{\!\Curve}(#2,#3)
	}
}

\DeclareDocumentCommand{\LineElement}{ o o }{
	\IfValueTF{#1}{
	  \omega_{#1}\IfValueTF{#2}{(#2)}{}
	}{
	  \omega\IfValueTF{#2}{(#2)}{}
	}
}

\DeclareDocumentCommand{\LineElementC}{ o }{
	\IfValueTF{#1}{
	  \omega_\Curve(#1)
	}{
	  \omega_\Curve
	}
}

\newcommand{\EnergyBound}{K}

\newcommand{\SubscriptTriangulation}{{\scriptscriptstyle\Triangulation}}

\newcommand{\Feasible}{\cC}
\newcommand{\DiscFeasible}{\Feasible_\SubscriptTriangulation}
\newcommand{\ConfSpace}{\tilde{\Feasible}}
\newcommand{\DiscConfSpace}{\ConfSpace_\SubscriptTriangulation}

\newcommand{\hide}[1]{}
\newcommand{\Sampling}[1][\SubscriptTriangulation]{{\mymathcal{S}_{#1}}}
\newcommand{\Reconstruction}[1][\SubscriptTriangulation]{{\mymathcal{R}_{#1}}}
\newcommand{\SamplingApprox}[1][\SubscriptTriangulation]{{\tilde{\mymathcal{S}}_{#1}}}
\newcommand{\ReconstructionApprox}[1][\SubscriptTriangulation]{{\tilde{\mymathcal{R}}_{#1}}}
\newcommand{\DiscTestMap}[1][\SubscriptTriangulation]{{\varPsi_{#1}}}
\newcommand{\Adjust}{\mymathcal{P}}
\newcommand{\DiscAdjust}{\mymathcal{P}_{\SubscriptTriangulation}}

\newcommand{\EulerBernoulli}{\cE}
\newcommand{\DiscEulerBernoulli}{\EulerBernoulli_\SubscriptTriangulation}

\newcommand{\Length}{\cL}

\newcommand{\MaxRadius}{h}

\newcommand{\AmbDim}{m} 
\newcommand{\AmbSpace}{{\R^\AmbDim}}

\newcommand{\Sphere}{\mathbb{S}}

\newcommand{\Priors}{\cA}

\newcommand{\DiscRestorationPriors}{\tilde\Priors_\SubscriptTriangulation(\StrainBound,\EnergyBound,\eta)}
\newcommand{\RestorationPriors}{\tilde\Priors(\StrainBound,\EnergyBound,\eta)}

\newcommand{\DiscPriors}{\Priors_\SubscriptTriangulation}

\newcommand{\Minimizers}{\cM}
\newcommand{\DiscMinimizers}{\Minimizers_\SubscriptTriangulation}

\newcommand{\DCond}{q} 
\newcommand{\NCond}{N} 

\newcommand{\EdgeMidpoints}{m}
\newcommand{\TurningAngles}{\alpha}
\newcommand{\EdgeVectors}{\tau}
\newcommand{\EdgeLengths}{\ell}

\newcommand{\ReferenceEdgeLengths}{\ell_0}
\newcommand{\DualEdgeLengths}{\bar \ell}
\newcommand{\DualReferenceEdgeLengths}{\bar \ell_0}

\newcommand{\Vertex}{i}
\newcommand{\OtherVertex}{j}
\newcommand{\Vertices}{V}

\newcommand{\IVertices}{V_{\on{int}}}

\newcommand{\Edge}{I}
\newcommand{\DualEdge}{\bar I}
\newcommand{\InteriorEdge}{I}

\newcommand{\Edges}{E}
\newcommand{\BEdges}{E_{\on{bnd}}}
\newcommand{\IEdges}{E_{\on{int}}}

\newcommand{\Interval}{\varSigma}

\newcommand{\Polygon}{P}
\newcommand{\OtherPolygon}{Q}
\newcommand{\Curve}{\gamma}

\newcommand{\ConstraintMap}{\varPhi}
\newcommand{\DiscConstraintMap}{\ConstraintMap_\SubscriptTriangulation}

\newcommand{\Tame}{\cB}
\newcommand{\DiscTame}{\Tame_\SubscriptTriangulation}

\newcommand{\Compactum}{\cK}

\newcommand{\TargetSpace}{\cY}
\newcommand{\DiscTargetSpace}{\TargetSpace_\SubscriptTriangulation}

\newcommand{\myincludegraphics}[2]{\begin{tikzpicture}
    \node[inner sep=0pt] (fig) at (0,0) {\includegraphics{#1}};
	\node[above right= 1ex] at (fig.south west) {\footnotesize\textbf{(#2)}};    
\end{tikzpicture}
}
\newcommand{\llincludegraphics}[2]{\begin{tikzpicture}
    \node[inner sep=0pt] (fig) at (0,0) {\includegraphics{#1}};
	\node[above right= 1ex] at (fig.south west) {#2};    
\end{tikzpicture}
}

\DeclareDocumentCommand{\Hess}{ O{} }{\operatorname{Hess}_{#1}}

\usepackage{xparse}

\DeclareDocumentCommand{\RX}{ O{\,} O{} O{} }{X_{\!#2}#3#1}
\DeclareDocumentCommand{\RH}{ O{\,} O{} O{} }{H_{\!#2}#3#1}
\DeclareDocumentCommand{\RY}{ O{\,} O{} O{} }{Y_{\!#2}#3#1}
\DeclareDocumentCommand{\RT}{ O{\,} O{} O{} }{T_{\!#2}#3#1}
\DeclareDocumentCommand{\RZ}{ O{\,} O{} O{} }{Z_{\!#2}#3#1}
\DeclareDocumentCommand{\T}{ O{} O{} }{T_{\!#1}#2}

\DeclareDocumentCommand{\Ri}{ O{} }{i_{#1}}
\DeclareDocumentCommand{\Rj}{ O{} }{j_{#1}}
\DeclareDocumentCommand{\RI}{ O{} }{I_{#1}}
\DeclareDocumentCommand{\RJ}{ O{} }{J_{#1}}
\DeclareDocumentCommand{\RK}{ O{} }{K_{#1}}
\DeclareDocumentCommand{\Rg}{ O{} }{g_{#1}}

\DeclareDocumentCommand{\RP}{ O{} }{P_{#1}}
\DeclareDocumentCommand{\RQ}{ O{} }{Q_{#1}}
\DeclareDocumentCommand{\RPsi}{ O{} }{\varPsi_{#1}}

\usepackage{xspace}

\newcommand{\lcold}{Black}

\newcommand{\invisible}[1]{}

\newcommand{\qand}{\quad \text{and} \quad}

\newcommand{\bnd}{\partial}

\DeclareMathOperator*{\esssup}{ess\,sup}

\DeclareMathOperator{\epi}{epi}
\DeclareMathOperator{\argmin}{arg\,min}

\newcommand{\cA}{{\mymathcal{A}}}
\newcommand{\cB}{{\mymathcal{B}}}
\newcommand{\cC}{{\mymathcal{C}}}
\newcommand{\cD}{{\mymathcal{D}}}
\newcommand{\cE}{{\mymathcal{E}}}
\newcommand{\cF}{{\mymathcal{F}}}
\newcommand{\cG}{{\mymathcal{G}}}

\newcommand{\cK}{{\mymathcal{K}}}
\newcommand{\cL}{{\mymathcal{L}}}
\newcommand{\cM}{{\mymathcal{M}}}

\newcommand{\cY}{{\mymathcal{Y}}}

\DeclareMathOperator{\dom}{dom}

\DeclareMathOperator{\Laplacian}{\Delta}
\newcommand{\dd}{{\operatorname{d}}}

\newcommand{\Map}{{\on{Map}}}

\newcommand{\real}{{\on{r}}}

\newcommand{\ee}{{\on{e}}}
\newcommand{\lamda}{{\lambda}}

\newcommand{\ceq}{\coloneqq}

\newcommand{\R}{{\mathbb{R}}}

\newcommand{\N}{\mathbb{N}}

\DeclareMathOperator{\id}{id}

\newcommand{\abs}[1]{\left\lvert#1\right\rvert} 
\newcommand{\nabs}[1]{\lvert{#1}\rvert} 
\newcommand{\bigabs}[1]{\big\lvert{#1}\big\rvert} 

\newcommand{\norm}[1]{\left\lVert#1\right\rVert}
\newcommand{\nnorm}[1]{\lVert{#1}\rVert}
\newcommand{\bignorm}[1]{\big\lVert{#1}\big\rVert}

\newcommand{\seminorm}[1]{\left[#1\right]}
\newcommand{\nseminorm}[1]{[{#1}]}

\newcommand{\ninnerprod}[1]{\langle{#1}\rangle}

\newcommand{\Biginnerprod}[1]{\Big\langle{#1}\Big\rangle}

\newcommand{\paren}[1]{\left(#1\right)}
\newcommand{\nparen}[1]{(#1)}
\newcommand{\bigparen}[1]{\big(#1\big)}

\newcommand{\Bigparen}[1]{\Big(#1\Big)}

\newcommand{\intervaloo}[1]{\left]#1\right[}
\newcommand{\intervalco}[1]{\left[#1\right[}
\newcommand{\intervalcc}[1]{\left[#1\right]}

\newcommand{\nintervaloo}[1]{{]#1[}}

\newcommand{\nintervalcc}[1]{{[#1]}}

\newcommand{\on}[1]{\operatorname{#1}}

\DeclareMathOperator{\Ima}{im}

\DeclareMathOperator{\dist}{dist}  
\DeclareMathOperator{\Hom}{Hom}    

\DeclareMathOperator{\pr}{pr}

%% file: ntheoremenglish.tex
\newcommand{\mynewtheorem}[4] 
{
\newaliascnt{#1}{#2}
\newtheorem{#1}[#1]{#3}
\aliascntresetthe{#1}
\expandafter\def\csname #1autorefname\endcsname{%
#4%
}%
}

\newtheorem{theorem}{Theorem}[section]
\mynewtheorem{lemma}{theorem}{Lemma}{Lemma} 
\mynewtheorem{proposition}{theorem}{Proposition}{Proposition} 
\mynewtheorem{corollary}{theorem}{Corollary}{Corollary}
\mynewtheorem{quest}{theorem}{Question}{Question}

\mynewtheorem{problem}{theorem}{Problem}{Problem}

\theoremstyle{break}
\mynewtheorem{btheorem}{theorem}{Theorem}{Theorem} 
\mynewtheorem{blemma}{theorem}{Lemma}{Lemma}
\mynewtheorem{bproposition}{theorem}{Proposition}{Proposition}
\mynewtheorem{bcorollary}{theorem}{Corollary}{Corollary}
\mynewtheorem{bproblem}{theorem}{Problem}{Problem}

\theoremstyle{plain}
\theorembodyfont{\normalfont}
\mynewtheorem{definition}{theorem}{Definition}{Definition}
\mynewtheorem{example}{theorem}{Example}{Example}
\mynewtheorem{remark}{theorem}{Remark}{Remark}

\theoremstyle{break}
\mynewtheorem{bdfn}{theorem}{Definition}{Definition}
\mynewtheorem{bex}{theorem}{Example}{Example}
\mynewtheorem{brem}{theorem}{Remark}{Remark}
\theoremheaderfont{\itshape}
\theorembodyfont{\upshape}
\theoremstyle{nonumberplain}
\theoremseparator{.}
\theoremsymbol{\ensuremath{\Box}}
\newtheorem{proof}{\textsc{Proof}}

%% file: Introduction.tex

\section{Introduction}\label{sec:intro}

The \emph{Euler-Bernoulli bending energy} is frequently used as a model for the bending part of the stored elastic energy of a thin, flexible but inextensible piece of material  that has a straight cylindrical rest state. 
For a compact interval $\varSigma \subset \R$ and a curve $\Curve \in \SoboC[2,2]$, the Euler-Bernoulli bending energy is defined as the integral of squared curvature with respect to the curve's line element $\LineElementC$, i.e., 
\begin{align}
	\EulerBernoulli(\Curve) \ceq \frac{1}{2} \int_\varSigma \,\nabs{\CurvatureC}^2 \, \LineElementC.
	\label{eq:EulerBernoulliEnergy}
\end{align}
The classical \emph{Euler elastica problem} is to find minimizers of $\EulerBernoulli$ in a feasible set $\Feasible$ of all curves of given \emph{fixed} curve length $L$ subject to \emph{fixed} first order boundary conditions that pin down positions and tangent directions at both ends of the curve. 
Together, these two constraints---fixed curve length and fixed boundary conditions---constitute the main difficulty of the problem. Indeed, in dimension two, dropping the positional constraints (while keeping the tangent and length constraints) would yield rather trivial minimizers in the form of circular arcs. Likewise, dropping the length constraint (while keeping endpoint and end tangents constraints)  would prevent existence of solutions: In dimension two, consider two straight line segments that respectively meet the two boundary conditions and connect these line segments by a circular arc at their free ends. The energy of such a curve is reciprocal to the length of the circular arc and thus arbitrarily small, yielding a minimizing sequence that does not converge.

We call prescribed constraints \emph{commensurable} if the distance between the end points is less than the curve's length $L$.
By $\Minimizers$ we denote the set of \emph{Euler elasticae}, i.e., the minimizers for fixed curve length and fixed boundary data. Notice that this set is nonempty for commensurable prescribed constraints.

\begin{figure}[t]
\capstart
\newcommand{\settrimming}{%
    \setkeys{Gin}{%
        trim = 48 115 35 150, 
        clip=true, 
        width=0.325\textwidth
    }
    \presetkeys{Gin}{clip}{}
}
\settrimming
\begin{center}
\llincludegraphics{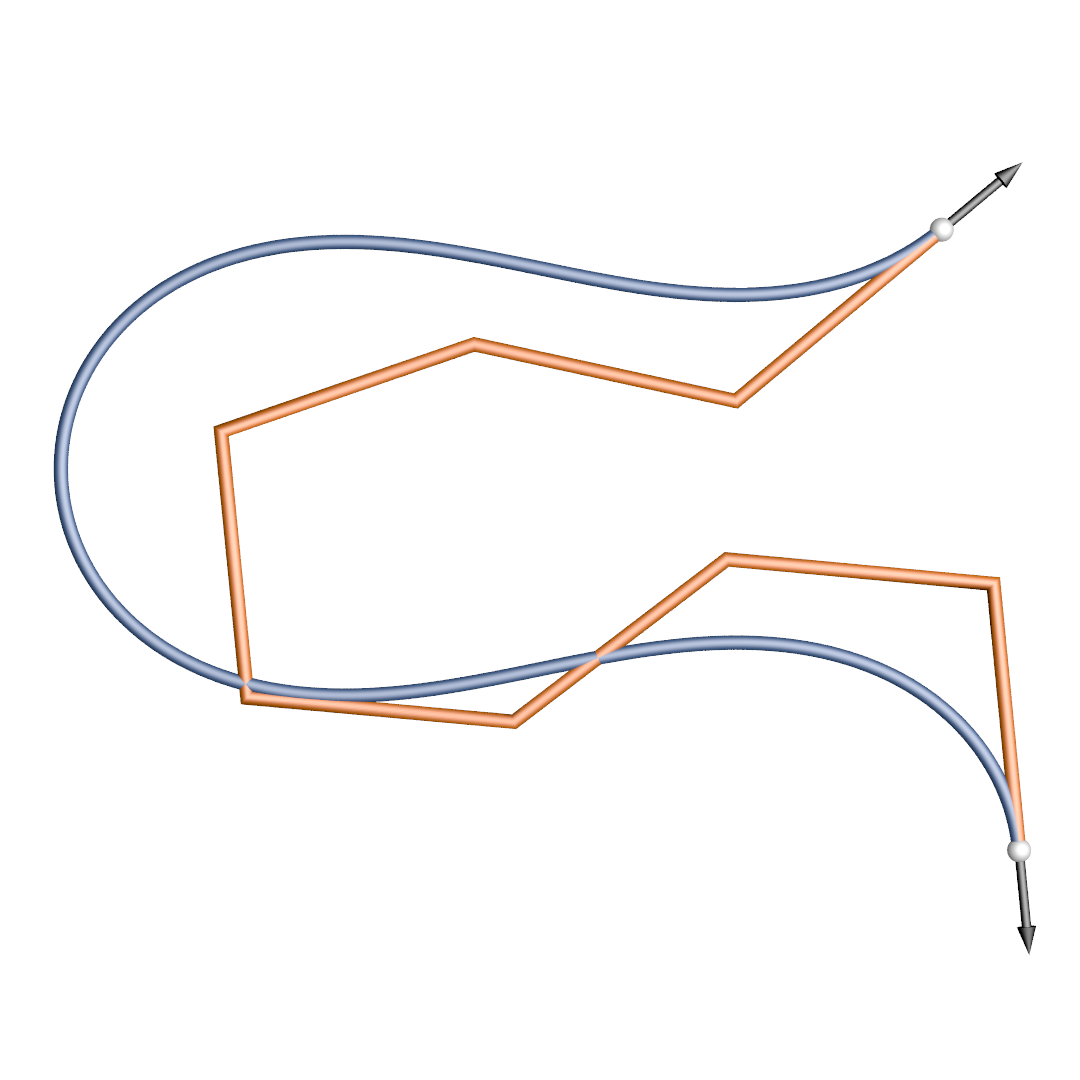}{\scriptsize$\nabs{\Edges(\Triangulation)} = 8$}%
\llincludegraphics{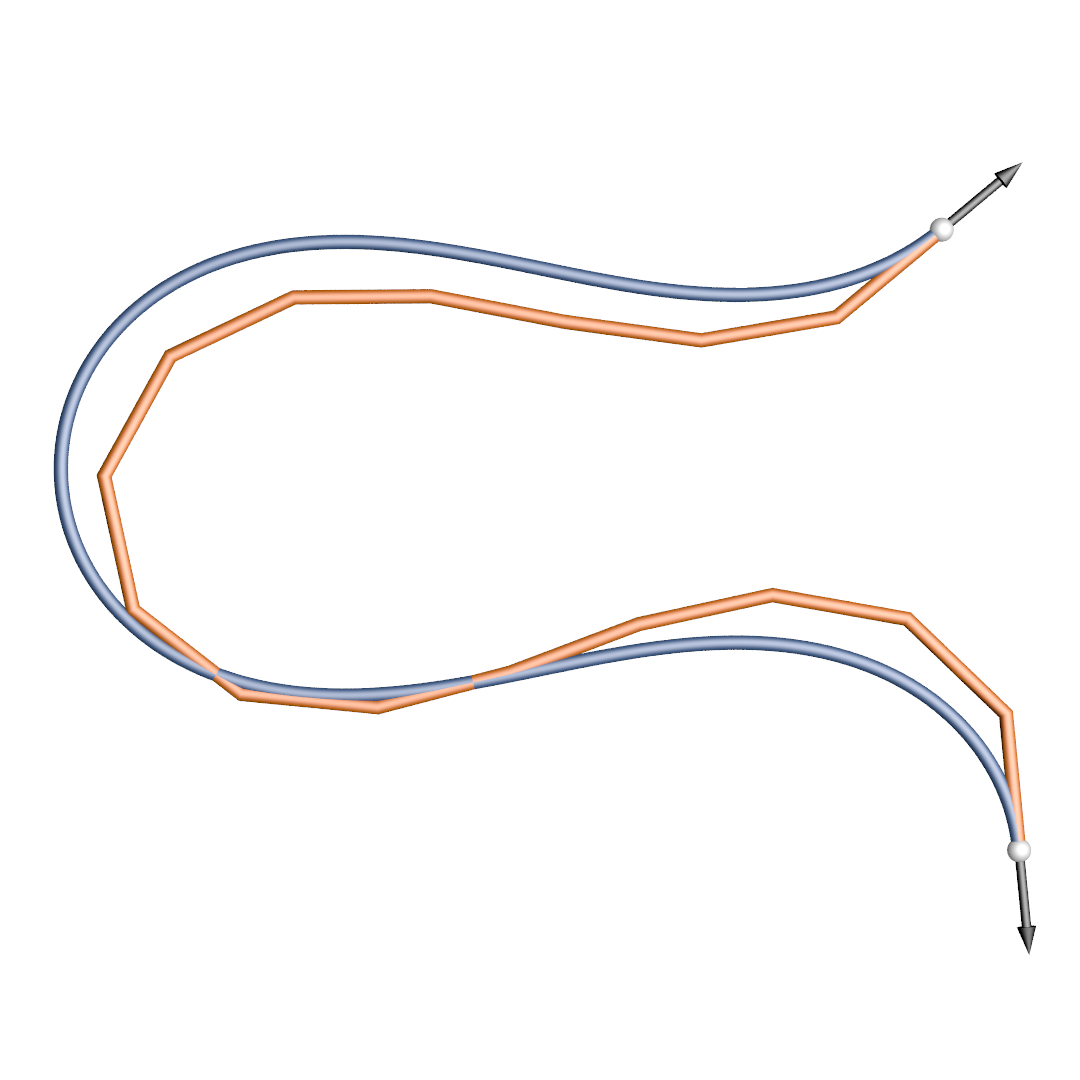}{\scriptsize$\nabs{\Edges(\Triangulation)} = 16$}%
\llincludegraphics{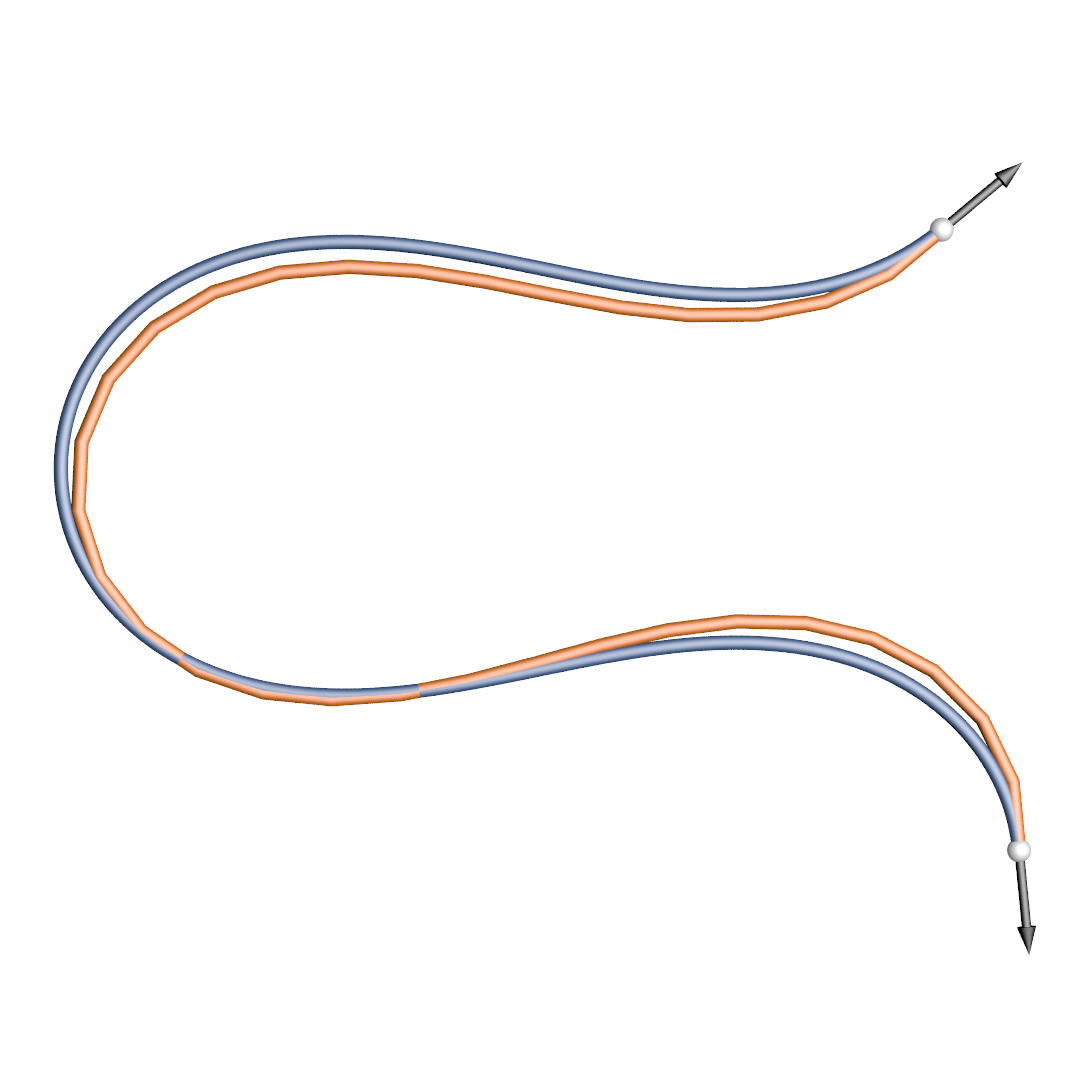}{\scriptsize$\nabs{\Edges(\Triangulation)} = 32$}%
\\
\llincludegraphics{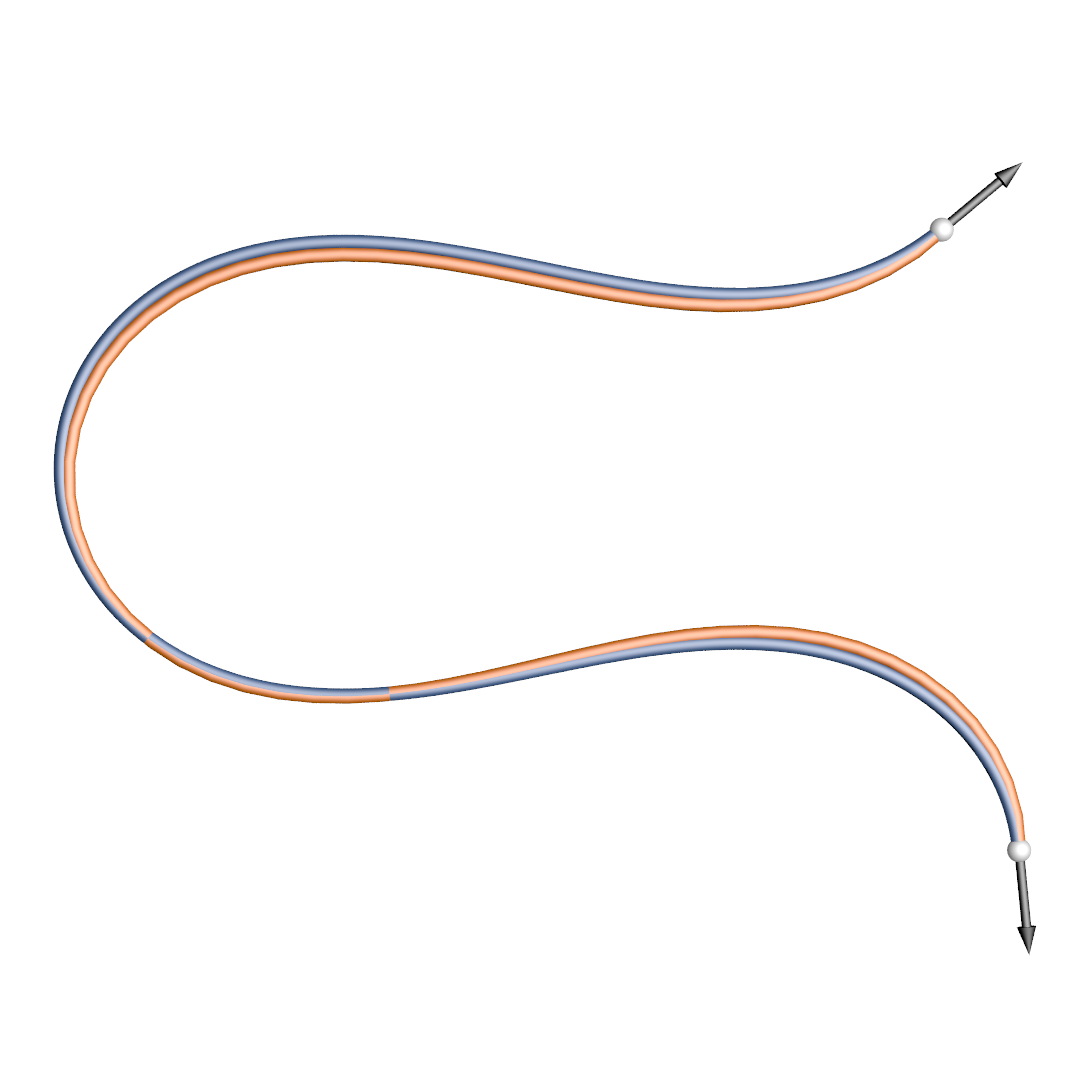}{\scriptsize$\nabs{\Edges(\Triangulation)} = 64$}%
\llincludegraphics{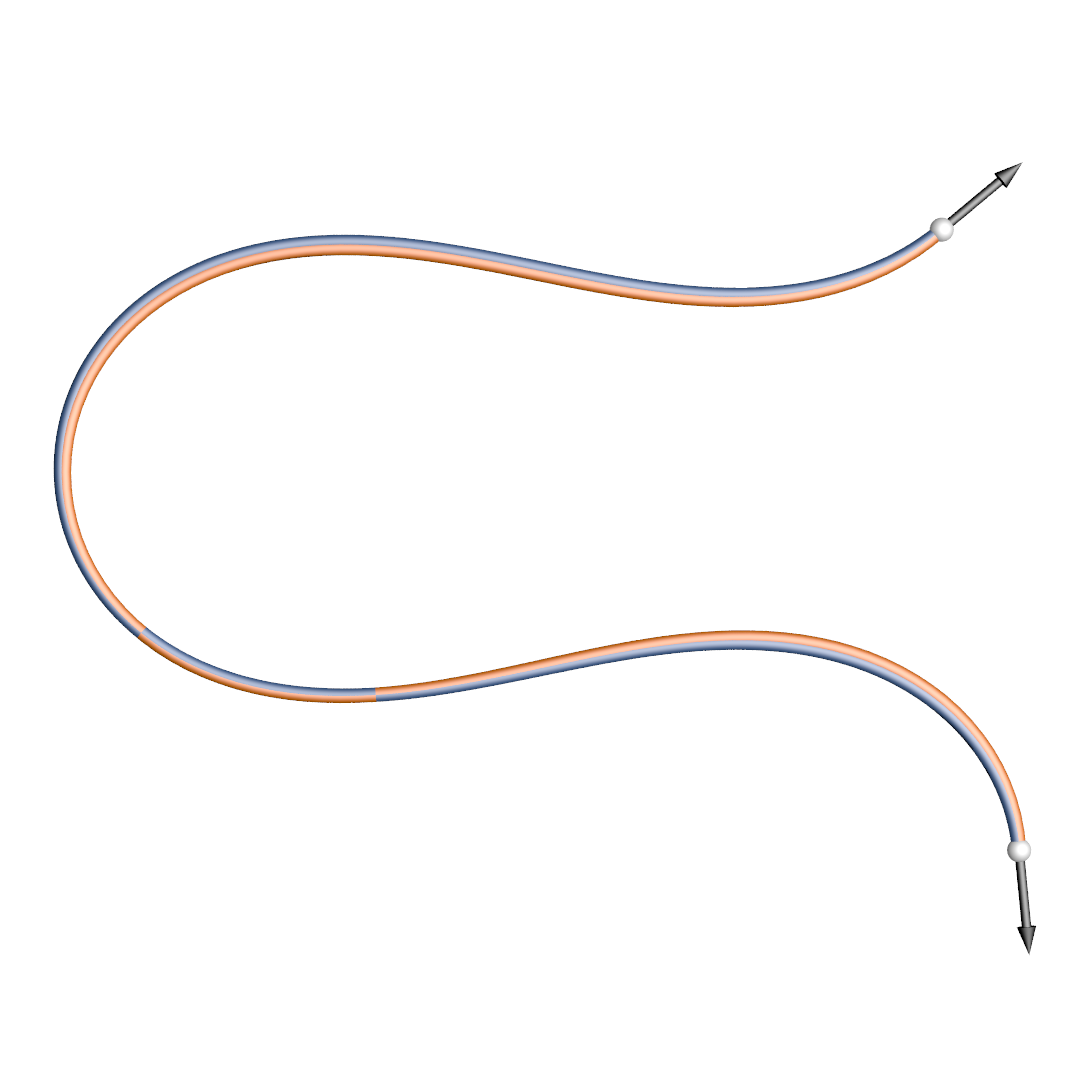}{\scriptsize$\nabs{\Edges(\Triangulation)} = 128$}%
\llincludegraphics{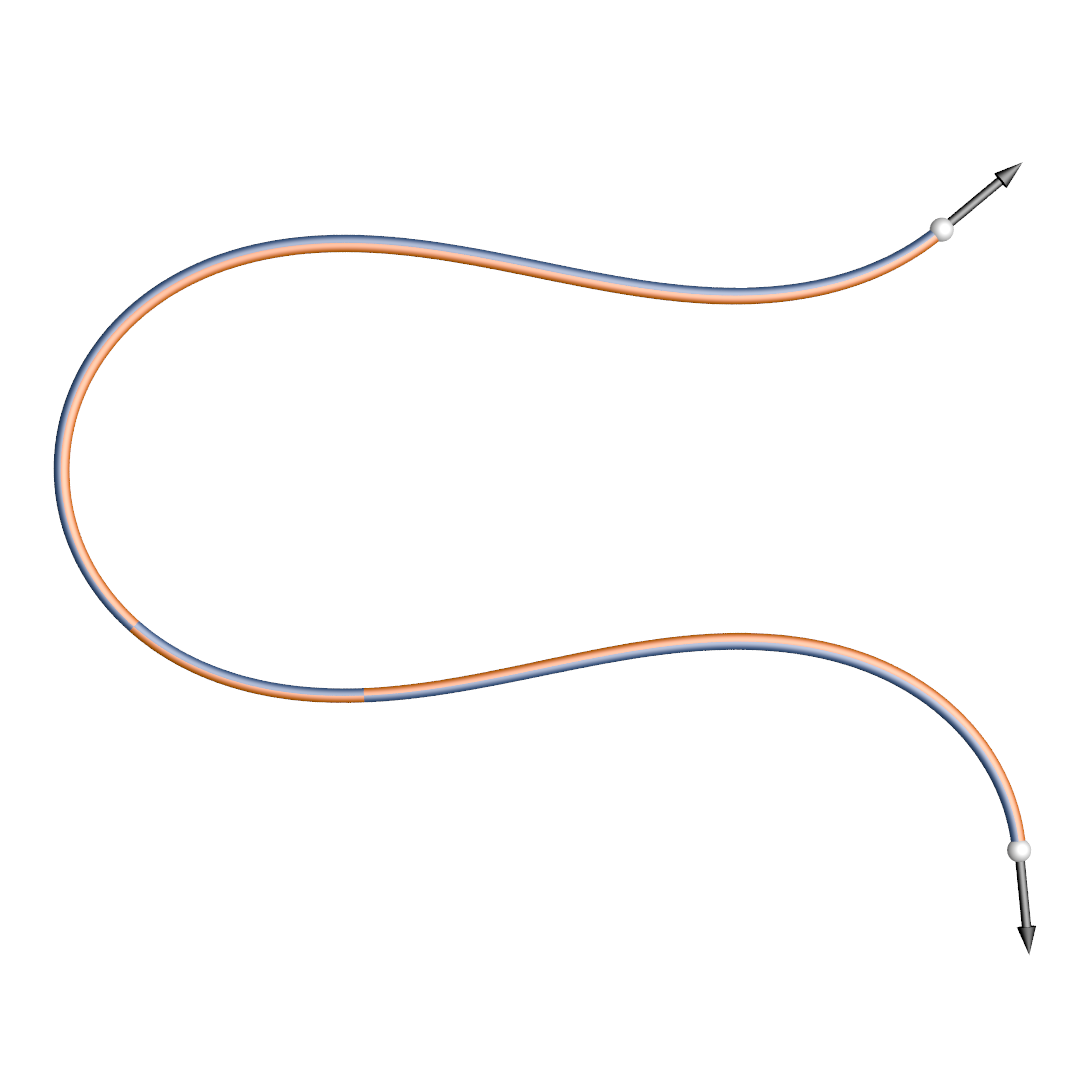}{\scriptsize$\nabs{\Edges(\Triangulation)} = 256$}%
\end{center}
\caption{An elastica in $\R^2$ with clamped ends (blue, obtained from Jacobi elliptic functions) compared to discrete energy minimizers (orange) for various 
resolutions.
}
\label{fig:SmoothvsDiscreteElastica2}
\end{figure}

\bigskip

Analytic representations of Euler elasticae in $\R^2$ and $\R^3$ can be expressed in terms of elliptic integrals 
(see, e.g., \cite{Levien08theelastica:}). Sometimes, however, it is more desirable to consider approximate solutions. 
This is even more true when elastic curves are coupled to external forces, in which case analytic solutions are no longer available. A possible finite dimensional ansatz space for approximate solutions is the space of cubic B-splines (piecewise polynomials of third order, fitted together with $C^1$-continuity), see, e.g., \cite{MR3119710}. While this space is a subset of the energy space $\SoboC[2,2]$, this formulation poses difficulties when enforcing the length constraint.

Therefore, polygonal models are often preferred due to their conceptual simplicity and ease of formulation:
On a finite partition $\Triangulation$ of the interval $\varSigma$ with vertex set $\Vertices(\Triangulation)$, consider the set of \emph{discrete immersions}.
This space contains all
polygons $\Polygon \colon V(\Triangulation) \to \AmbSpace$ whose successive vertices are mapped to distinct points.
On this set, define the \emph{discrete Euler-Bernoulli energy}  by
\begin{align}
	\textstyle
	\DiscEulerBernoulli(\Polygon) \ceq 
	\frac{1}{2}
	\sum_{\Vertex \in \IVertices(\Triangulation)} 
	\left(
		\frac{\TurningAngles_\Polygon(\Vertex)}{\DualEdgeLengths_\Polygon(\Vertex)}
	\right)^2 \, \DualEdgeLengths_\Polygon(\Vertex)
	=	
	\frac{1}{2}
	\sum_{\Vertex \in \IVertices(\Triangulation)} 
\frac{\TurningAngles^2_\Polygon(\Vertex)}{\DualEdgeLengths_\Polygon(\Vertex)},
	\label{eq:DiscreteEulerBernoulliEnergy}
\end{align}
where $\TurningAngles_\Polygon(\Vertex)$ is the \emph{turning angle} at an interior vertex $\Vertex$ and $\DualEdgeLengths_\Polygon(\Vertex)$ is the \emph{dual edge length}, 
i.e., the arithmetic mean of the lengths of the two adjacent (embedded) edges. 
This energy is motivated by the observation that turning angles are in many ways a reasonable surrogate for integrated absolute curvature on dual edges (see, e.g., \cite{MR2405664}, \cite{Crane_Wardetzky_2017}).\footnote{Notice that if one considers dual edge lengths $\DualEdgeLengths$ based on the \emph{circular arc} through three consecutive points in the formulation of \eqref{eq:DiscreteEulerBernoulliEnergy}, then the value of the resulting discrete energy is \emph{exact} for circles, independent of the discretization.}
If one replaces the absolute curvature density in \eqref{eq:EulerBernoulliEnergy} by the averages of the total absolute curvatures over their respective dual edges, one is immediately led to \eqref{eq:DiscreteEulerBernoulliEnergy}. It seems that this model has first been considered by Hencky in his 1921 PhD thesis \cite{hencky1921angenaherte}.

\begin{table}
\begin{center}
\begin{tabular}{c|c|l|c|c|c|c|c}
	& functional & convergence
	& clamped	
	& $\Length$		
	& $\ell_i \neq \ell_j$
	& $\R^m$
	\\\hline
	\cite{MR1880530}
	& $\EulerBernoulli$, $\EulerBernoulli^p$ &$\Gamma$ (Fréchet) 
	&\checkmark 
	& ${-}^{\text{a}}$	
	&\checkmark 
	&\checkmark
	\\\hline
	\cite{MR3318319}
	& $\EulerBernoulli + \beta \, \Length$
	& \shortstack[l]{$\Gamma$ (Fréchet and $\BV[2]$),
		clustering in $\Sobo[2,2]$
	}
	&\checkmark 
	& ${-}^{\text{a}}$	
	& $-$ 
	& $-$
	\\\hline
	\cite{MR3619438}	
	&$\EulerBernoulli$ &$\Gamma$ (weak-$*$ in $(C^0)'$)
	& ${-}^{\text{b}}$ 
	&\checkmark
    	& $-$ & \checkmark
	\\\hline
	\cite{MR3645030}	
	& $\cF$ 
	&$\Gamma$ (weak-$*$ in $(C^0)'$)
	& ${-}^{\text{b}}$  	
	&\checkmark
	& $-$ & \checkmark
	\\\hline
	\cite{doi:10.1177/1081286517707997}
	& $\EulerBernoulli$ & $\Gamma$ (weak-$\Sobo[2,2]$)
	& ${-}^{\text{c}}$
	&\checkmark
	& $-$ 
	& $-$
	\\\hline
	\text{ours}
	&$\EulerBernoulli$ 
	&\shortstack[l]{Hausdorff ($\Sobo[2,p]$, $p\in \intervalco{2,\infty}$)
	}
	& \checkmark 
	&\checkmark 
	&\checkmark 
	&\checkmark
\end{tabular}
\end{center}
\caption{
Brief overview on the literature related to the convergence of discrete elastica.\newline
$\EulerBernoulli^p(\Curve) = \int \nabs{\kappa}^p \, \nabs{\gamma'}\, \dd t$\newline
$\cF(\Curve) = \int (f(\nabs{\Curve'})\!+\!g(\nabs{\Curve''}) \, \dd t$ where $f$, $g$ convex\newline
${}^{\text{a}}$: recovery sequence does not satisfy this constraint\newline
${}^{\text{b}}$: clamped boundary conditions only at one end\newline
${}^{\text{c}}$: only closed curves are discussed
}
\label{tab:CitationonGammaConvergence}
\end{table}

The discrete problem is to find the set $\DiscMinimizers$ of minimizers of $\DiscEulerBernoulli$ restricted to the feasible set $\DiscFeasible$ of all polygons in discrete arc length parameterization subject to
the \emph{same} fixed first order boundary conditions as in the smooth setting above.  
Here, \emph{discrete arc length parameterization} means that an edge in the image of the piecewise linear parameterization has the same length as the corresponding edge of $\Triangulation$.

Convergence of discrete Euler elastica towards their smooth counterparts has previously been considered within the context of  $\Gamma$-convergence (also denoted $\epi$-con\-ver\-gence by some authors).
However, each of the treatments we know of relaxes at least one of the constraints in a significant manner, and several approaches show $\Gamma$-convergence with respect to rather coarse topologies. Also, all but one of these approaches require equilateral polygons in order to enable symmetric finite differencing. Among previous approaches, \cite{MR1880530} comes close to our goals, while our result is stronger by showing Hausdorff-convergence in $W^{2,p}$ for $p \in \intervalco{2,\infty}$. \autoref{tab:CitationonGammaConvergence} summarizes the situation. 

\begin{figure}[t]
\capstart
\setkeys{Gin}{%
        trim = 20 20 20 180 , 
        clip=true, 
        width=0.49\textwidth
    }
    \presetkeys{Gin}{clip}{}
\begin{center}
\includegraphics{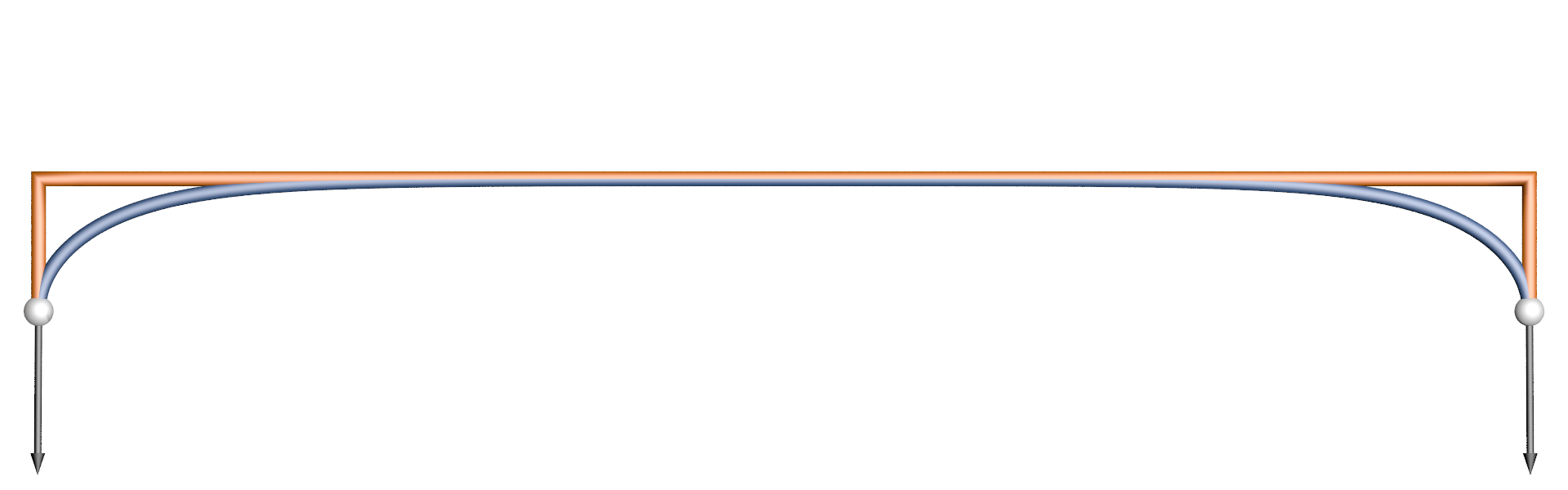}
\includegraphics{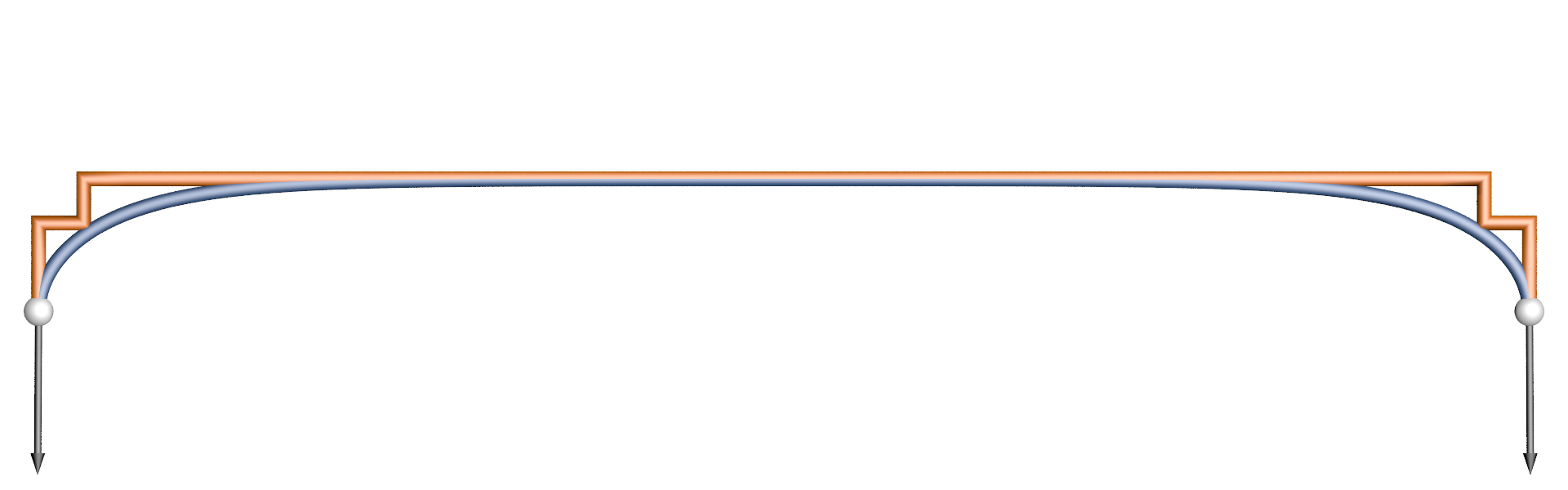}
\end{center}
\caption{Two polygonal curves (orange) that have exactly the same length and approximate a smooth elastic curve (blue) for vertically clamped edges. The energies of the discrete curves, however, differ by a factor of three, independent of their resolution. Small perturbations of vertex positions may thus lead to drastic changes in elastic energy.
}
\label{fig:Weierstrass}
\end{figure}

One of the challenges of $\Gamma$-convergence for discrete elasticae arises from the length constraint, which might be violated when approximating a smooth curve by a sequence of polygons. Repairing this constraint violation requires to change a polygon's vertex positions. Changing vertex positions, however, might drastically affect the energy. To illustrate this phenomenon, consider the example of an elastica in 2D with \emph{vertically} clamped boundaries at points $(-a, 0)$ and $(0,a)$ and with prescribed length close to $2a$, see \autoref{fig:Weierstrass}. Consider furthermore two polygonal approximations: (i) a rectangle respecting the boundary conditions whose edges are uniformly subdivided and (ii) the same shape as in (i) but now with the two corners flipped inwards, see \autoref{fig:Weierstrass}. Both of these polygonal shapes somewhat approximate the smooth solution and both polygons have exactly the same length, but their energies differ by a factor of three, while both energies are arbitrarily large depending on the amount of subdivision. This illustrates that small changes of vertex positions result in small (or even no) change in total length but might yield comparatively large changes of elastic energy. 

While $\Gamma$-convergence is a very satisfactory, qualitative result from the perspective of homogenization of discrete mechanical systems, 
it remains somewhat unsatisfactory due to its nonquantitative nature. In a nutshell,
$\Gamma$-convergence of $\cF_n \colon X \to \R$ to $\cF \colon X \to \R$ for $n \to \infty$ in a topological space $X$ implies that cluster points of $\cF_n$-minimizers are minimizers of $\cF$, i.e., that
$
	\textstyle
	\limsup_{n \to \infty} \argmin(\cF_n)
	\ceq 
	\bigcap_{n \to \infty} \overline{\paren{\bigcup_{k \geq n} \argmin(\cF_k)}}
	\subset \argmin(\cF).
$
While this on its own does not imply \emph{clustering}, i.e., that cluster points exist ($\limsup_{n \to \infty} \argmin(\cF_n) \neq \emptyset$), the latter can often be shown by utilizing uniform compactness properties of lower level sets of the functions $\set{\cF_k | k \geq n}$.
However, $\Gamma$-convergence does not guarantee that all minimizers of $\cF$ can be obtained as cluster points
($\limsup_{n \to \infty} \argmin(\cF_k) = \argmin(\cF)$) or that
$\argmin(\cF_k)$ converges to $\argmin(\cF)$ in the sense of Kuratowski or even Hausdorff.
Indeed, as is well know, this cannot be expected in general. 
For an illustration consider the example in \autoref{fig:TiltedMexicanHatAlmostMinimizers}. Here, sequences of minimizers of  $\cF_n$ converge to two distinct points, while the minimizing set of $\cF$ is a circle. The lack of convergence of minimizers of $\cF_n$  to the actual minimizer of $\cF$ is caused by symmetry breaking of $\cF_n$ vs. $\cF$---very similar to what happens when discretizing parameterization invariant optimization problems for immersed curves and surfaces---as, e.g., the Euler elastica problem.

\begin{figure}
\begin{center}
\includegraphics[width=	\textwidth]{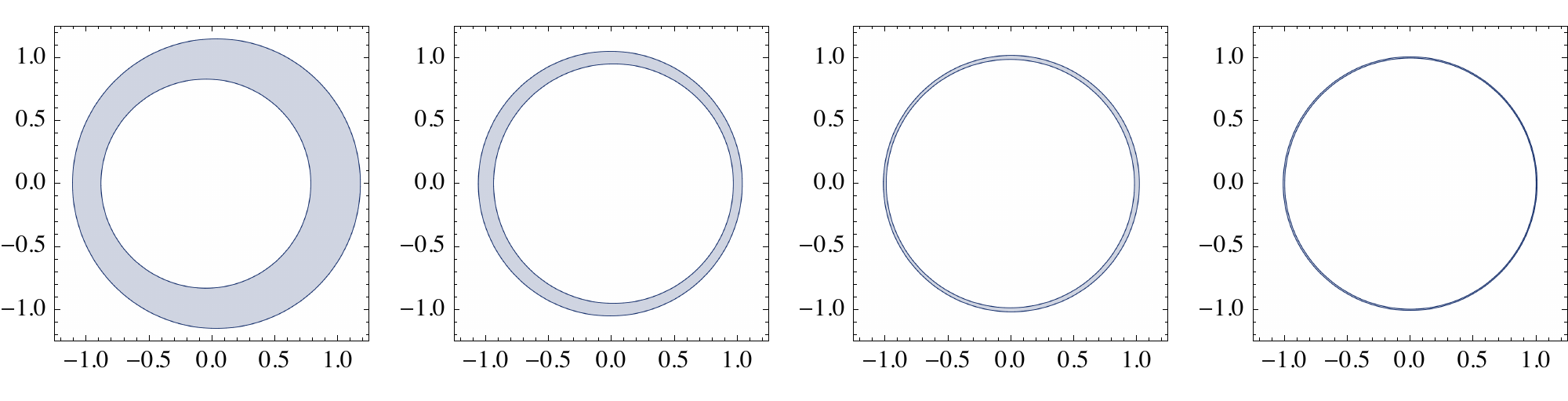}
\label{fig:TiltedMexicanHatAlmostMinimizers}
\caption{The \emph{almost minimizing sets} $\argmin^\frac{3}{n}(\cF_n)$ of the tilted potentials $\cF_n(x) = (1-\nabs{x}^2)^2 - (-1)^n \frac{1}{n} \frac{x_1}{1+\nabs{x}^2}$ converge (as  sets) to the minimizer of the potential $\cF(x) = (1-\nabs{x}^2)^2$ (the unit circle). In contrast, while $\cF_n$ $\Gamma$-converges to $\cF$, the respective minimizers ($\set{(1,0)}$ for even $n$ and $\set{(-1,0)}$ for odd $n$) \emph{do not} converge to the minimizers of $\cF$. 
}
\end{center}
\end{figure}

The problem illustrated in this example can be overcome by shifting attention from sets of minimizers to sets of \emph{almost-minimizers} or \emph{$\delta$-minimizers}
\begin{align*}
	\argmin^{\delta}(\cF_n) \ceq \set{x | \text{for all $y$: $\cF_n(x) \leq \cF_n(y) + \delta$}}.
\end{align*}
In this case, significantly more information about $\argmin(\cF)$ can be drawn from $\cF_n$. E.g., in the above example, the sets $\argmin^\frac{3}{n}(\cF_n) = \set{x| \cF(x) \leq \frac{1}{n}}$ converge uniformly to $\argmin(\cF)$ (see \autoref{fig:TiltedMexicanHatAlmostMinimizers}). This motivates our use of almost minimizers throughout this article.
Yet, focusing solely on $\delta$-minimizers 
\begin{align}
	\DiscMinimizers^\delta 
	&\ceq \argmin^\delta(\DiscEulerBernoulli)
	= \set{\Polygon \in \DiscFeasible | \text{for all $\OtherPolygon \in \DiscFeasible \colon \DiscEulerBernoulli (\Polygon) \leq \DiscEulerBernoulli (\OtherPolygon) + \delta$}}
	\label{eq:DiscreteAlmostMinimizers}
\end{align}
would require the extraction of information about rather \emph{generic} elements of the search space.
In contrast, minimizers are ``tamer'' than generic elements of the search space, usually formulated as \emph{\apriori estimates} on certain Sobolev norms. We therefore restrict our attention to the subset of almost minimizers that satisfy the same regularity properties as the actual minimizers.

We encode these \apriori assumptions into the two sets
\begin{align}
	\Priors &\ceq 
	\textstyle	
	\set{\Curve \in \Feasible
	| 
			\seminorm{\Curve}_{\Sobo[2,\infty]},\;
			\seminorm{\Curve}_{\Sobo[3,\infty]}
			\leq \EnergyBound_1
	},
	\label{eq:SmoothPriors}
	\\
	\DiscPriors
	&\ceq
	\textstyle
	\set{\Polygon \in \DiscFeasible 
		| 
		\nseminorm{\Polygon}_{\sobo[2,\infty][\Polygon]} ,
		\,
		\nseminorm{\Polygon}_{\tv[3][\Polygon]} \leq \EnergyBound_2
	},
	\label{eq:DiscretePriors}	
\end{align}
where $\EnergyBound_1$ and $\EnergyBound_2 \geq 0$ are suitable constants. The respective norms are specified in \autoref{sec:SmoothSetting} and \autoref{sec:DiscreteSetting}. These \apriori assumptions are justified by regularity properties of smooth and discrete minimizers. In the smooth setting, regularity of minimizers (see \autoref{lem:SmoothRegularity}) can be verified in various ways, e.g., by invoking elliptic integrals. Here we present a functional analytic approach since this approach can be closely mimicked in the discrete case. Our discrete regularity result (\autoref{lem:DiscreteRegularity}) appears to be new and might be of interest in its own right.

\subsection{Main result}

Let $\Feasible$ and $\DiscFeasible$ denote the spaces of smooth and discrete feasible configurations (i.e., those satisfying the boundary conditions). Let  $\Priors$ and $\DiscPriors$ denote the sets encoding our smooth and discrete \apriori assumptions (see \eqref{eq:SmoothPriors} and \eqref{eq:DiscretePriors}), and let $\Minimizers$ and $\DiscMinimizers$ denote the sets of smooth and discrete minimizers, respectively. These  sets satisfy the following inclusions:
\begin{align*}
\Minimizers \subset\Priors\subset \Feasible 
\quad\text{and} \quad
\DiscMinimizers \subset\DiscPriors\subset\DiscFeasible . 
\end{align*}
We rely  on a \emph{reconstruction operator} 
$\Reconstruction \colon 
\DiscPriors \to \Feasible$ and a \emph{sampling operator} 
$\Sampling \colon 
\Priors \to \DiscFeasible$, taking feasible polygons to smooth feasible curves and vice-versa. We provide the requisite reconstruction and sampling operators in \autoref{sec:Reconstruction} and \autoref{sec:Sampling}. We first provide \emph{approximate} reconstruction and sampling operators 
$\ReconstructionApprox$ and $\SamplingApprox$
that map 
$\DiscPriors$ and $\Priors$ into sufficiently small vicinities of
$\Feasible$ and $\DiscFeasible$, respectively. The main idea for these \emph{approximate} operators is that they only satisfy the boundary conditions and the length constraint \emph{approximately} but not necessarily exactly. 
We analyze  properties of these operators in \autoref{prop:ApproximateReconstructionTheorem} and \autoref{prop:ApproximateSamplingTheorem}. We then
apply a Newton-Kantorovich-type theorem
in order to show that \emph{exact} reconstruction and sampling operators $\Reconstruction$ and $\Sampling$ (i.e., those that satisfy the requisite constraints exactly) can be obtained from $\ReconstructionApprox$ and $\SamplingApprox$ by sufficiently small perturbations (see \autoref{prop:RestorationOperator} and \autoref{prop:DiscreteRestorationOperator}). In view of the example given in \autoref{fig:Weierstrass}, it is by no means obvious that this is possible while keeping the length constraint \emph{and} maintaining a useful energy bound. These results then lead to \autoref{theo:ReconstructionTheorem}, \autoref{theo:SamplingTheorem}, and \autoref{cor:Loop}, which are required for the proof of our main result:

\begin{theorem}[Convergence Theorem]\label{theo:ConvergenceTheorem}
Fix a prescribed curve length $L$ and commensurable boundary conditions. 
Denote by $\MaxRadius(\Triangulation)$ the maximum edge length of a partition $\Triangulation$ of the domain $\varSigma$, and let $\DiscTestMap \colon \DiscFeasible \to \SoboC[1,\infty]$  denote the interpolation of vertices of discrete curves by continuous, piecewise affine curves. 
Then for each $p \in \intervalco{2,\infty}$  there is a constant $C \geq 0$ such that one has the following convergence in Hausdorff distances:
\begin{align*}
	&\lim_{\MaxRadius(\Triangulation) \to 0} \dist_{\Sobo[2,p]} \bigparen{\Minimizers , \Reconstruction \bigparen{\DiscPriors \cap \DiscMinimizers^{C \MaxRadius(\Triangulation)}}} = 0 \qand
	\\
	&\lim_{\MaxRadius(\Triangulation) \to 0} \dist_{\Sobo[1,\infty]} \bigparen{\Minimizers, \DiscTestMap \bigparen{\DiscPriors \cap \DiscMinimizers^{C \MaxRadius(\Triangulation)}}} = 0.	
\end{align*}
\end{theorem}
Notice that although sampling operators do not appear explicitly in this result, they play a prominent role in the proof since they guarantee existence of discrete almost minimizers in the vicinity of every smooth minimizer.

By relying on \apriori assumptions, our result is different from $\Gamma$-convergence. Indeed, by restricting to the sets  $\Priors$ and $\DiscPriors$, we avoid the need for recovery sequences for every element in configuration space. We thus obtain a stronger convergence result in the sense that \emph{all} discrete minimizers are uniformly close to the set of smooth minimizers $\Minimizers$ with respect to 
$\Sobo[2,p]$-norm, 
i.e., there exists a function $f \colon \intervalcc{0,\infty} \to \intervalcc{0,\infty}$, continuous at $0$ and with $f(0) = 0$, such that
\begin{align*}
	\sup_{\Polygon \in \DiscMinimizers} 
	\inf_{\Curve \in \Minimizers}
	\nnorm{\Curve - \Reconstruction(\Polygon)}_{\Sobo[2,p]} \leq f(\MaxRadius(\Triangulation)).
\end{align*}
Since $p>2$ is allowed, \emph{we obtain convergence in a topology finer than the one of the energy space}. 

\begin{proof}
We here provide the proof of our main result since it is fairly short once our results for reconstruction and sampling operators from \autoref{theo:ReconstructionTheorem} and \autoref{theo:SamplingTheorem} are established. Indeed, as will become evident, these two theorems together with \autoref{cor:Loop} suffice in order to complete the proof. Establishing these theorems is the technically involved part of our exposition, and the following sections are devoted to proving these results.
We only consider those partitions $\Triangulation$ with maximum edge length $\MaxRadius = \MaxRadius(\Triangulation)$ for which $\MaxRadius$ is small enough such that \autoref{theo:ReconstructionTheorem} and \autoref{theo:SamplingTheorem} can be applied. 

For a subset $A$ of a metric space with metric $d$ and for $r>0$, let the \emph{$r$-thickening} of $A$ be
\begin{align*}
	\textstyle
	\ClosedBall{A}{r}
	\ceq 
	\set{x | \text{there is } a \in A \colon d(x,a) \leq r}
	=
	\bigcup_{a \in A} \ClosedBall{a}{r}.
\end{align*}
We fix $p\in \intervalco{2,\infty}$ and unless stated otherwise, all balls considered in this subsection are with respect to the $\Sobo[2,p]$-norm. In order to show convergence in the $\Sobo[2,p]$-Hausdorff distance, we show that 
\begin{align}
	\Minimizers 
	\subset
	\bigClosedBall{\Reconstruction (\DiscPriors \cap \DiscMinimizers^{C \MaxRadius})}{C \, \MaxRadius}
	\quad \text{and}\quad
	\Reconstruction( \DiscPriors \cap \DiscMinimizers^{C \MaxRadius})
	\subset \bigClosedBall{\Minimizers}{f(h)}
	\label{eq:Hsdff}
\end{align}
for a  monotonically increasing function $f \colon \intervalco{0,\infty} \to \R$ with $0=f(0) = \lim_{\delta \to 0 } f(\delta)$. Since the $\Sobo[2,p]$-norm dominates the $\Sobo[1,\infty]$-norm, Hausdorff convergence of $\DiscTestMap \bigparen{
			\DiscPriors
			\cap
			\DiscMinimizers^{\delta}
		}$
to $\Minimizers$ with respect to the $\Sobo[1,\infty]$-distance then follows from the above and
Statement~\ref{item:RProximity} in \autoref{theo:ReconstructionTheorem}.

In order to prove \eqref{eq:Hsdff}, we first note that 
the consistency estimates from \autoref{theo:ReconstructionTheorem} and \autoref{theo:SamplingTheorem} 
can be summarized as follows:
\begin{align*}
	\EulerBernoulli \circ \Reconstruction(\Polygon) 
	&\leq 
	\DiscEulerBernoulli(\Polygon) + C \, \MaxRadius 
	\quad
	\text{for all $\Polygon \in \DiscPriors$ and}
	\\
	\DiscEulerBernoulli \circ \Sampling(\Curve) 
	&\leq 
	\EulerBernoulli(\Curve) + C \, \MaxRadius
	\quad 
	\text{for all $\Curve \in \Priors$.}
\end{align*}
From these inequalities we deduce two consequences. First, by choosing minimizing sequences $(\Curve_k)_{k\in\N}$ of $\EulerBernoulli$ in $\Priors$ and $(\Polygon_{k})_{k\in\N}$ of $\DiscEulerBernoulli$ in $\DiscPriors$, we obtain that 
\begin{align*}
	\setlength{\arraycolsep}{0.1em}
	\begin{array}{ccccccccc}
	0 
	&\leq 
	&\inf (\EulerBernoulli)
	&\leq 
	&\EulerBernoulli \circ \Reconstruction(\Polygon_k) 
	&\leq 
	&\DiscEulerBernoulli(\Polygon_k) + C \, \MaxRadius
	&\stackrel{k \to \infty}{\longrightarrow} 
	&\inf(\DiscEulerBernoulli) + C \, \MaxRadius,
	\\
	0 
	&\leq 
	&\inf (\DiscEulerBernoulli) 
	&\leq 
	&\DiscEulerBernoulli \circ \Sampling(\Curve_k) 
	&\leq 
	&\EulerBernoulli(\Curve_k) + C \, \MaxRadius
	&\stackrel{k \to \infty}{\longrightarrow} 
	&\inf(\EulerBernoulli) + C \, \MaxRadius.	
	\end{array}
\end{align*}
This shows that the minimal values of the smooth and discrete problems satisfy 
\begin{align*}
	\abs{\inf(\EulerBernoulli) - \inf(\DiscEulerBernoulli)} \leq C \, \MaxRadius.
\end{align*}
Secondly, we deduce that the sets of $\delta$-minimizer are related to each other as follows:
\begin{align}
	\Reconstruction(\DiscPriors \cap \DiscMinimizers^\delta) \subset \Minimizers^{\delta+C \MaxRadius}
	\qand
	\Sampling(\Priors \cap \Minimizers^\delta) \subset \DiscMinimizers^{\delta+C \MaxRadius}.
	\label{eq:SublevelConsistency}
\end{align}
Moreover, \autoref{cor:Loop} and \autoref{theo:SamplingTheorem}  along with $\Minimizers \subset \Priors$ guarantee the following:
\begin{itemize}
\item[(i)] $\sup_{\Curve \in \Minimizers} \nnorm{\Curve - (\Reconstruction \circ \Sampling)(\Curve)}_{\Sobo[2,p]}  \leq C \, \MaxRadius$ and
\item[(ii)] $\Sampling(\Priors) \subset \DiscPriors$.
\end{itemize}
Together (i), (ii), and \eqref{eq:SublevelConsistency} imply that
\begin{align*}
	\Minimizers 
	\subset 
	\bigClosedBall{ (\Reconstruction \circ \Sampling)(\Minimizers)}{C\,\MaxRadius}
	\subset
	\bigClosedBall{\Reconstruction (\DiscPriors \cap \DiscMinimizers^{C \MaxRadius})}{C\,\MaxRadius},
\end{align*}
which is the first part of \eqref{eq:Hsdff}. In order to show the second part of  \eqref{eq:Hsdff}, notice that 
 \autoref{theo:ReconstructionTheorem} implies that $\Reconstruction(\DiscPriors)$ is bounded in $\BVC[3] $. Therefore, since the embedding $\BVC[3] \hookrightarrow \SoboC[2,p]$ is compact for $p< \infty$, we have 
\begin{itemize}
\item[(iii)] $\Reconstruction(\DiscPriors) \subset \Compactum$ with a $\Sobo[2,p]$-compact set $\Compactum$ containing $\Minimizers$.\footnote{The set $\Compactum$ can be chosen as a $\TV[3]$-ball of sufficiently large radius (independent of the partition $\Triangulation$) so that both the $\TV[3]$-bounded sets $\Reconstruction(\DiscPriors)$ and $\Minimizers \subset \Priors$ are contained in it.}
\end{itemize}
Combining this with \eqref{eq:SublevelConsistency}, we obtain the desired inclusion
\begin{align*}
	\Reconstruction(\DiscPriors \cap \DiscMinimizers^{C \MaxRadius}) 
	\subset \Compactum \cap \Minimizers^{2 C \MaxRadius}
	\subset \ClosedBall{\Minimizers}{f(\MaxRadius)},
\end{align*}
where the function $f$ is defined by
\begin{align*}
	f(\delta) \ceq \sup \set{
		d(\Curve,\Minimizers) | \Curve \in \Compactum \cap \Minimizers^{2C\delta}
	}.
\end{align*}
Indeed, because of $\emptyset \neq \Minimizers \subset \Compactum \cap \Minimizers^{2C\delta}$, we have $f(\delta) \geq 0$. 
Since $\EulerBernoulli$ is continuous with respect to the $\Sobo[2,p]$-norm, the lower level sets $\Minimizers^{\delta}$ are closed so that the sets $\Compactum \cap \Minimizers^{2C\delta}$ are compact. This shows that $f$ is well-defined and fulfills $f(\delta)<\infty$ for each $\delta \geq 0$.
Since the sets $\Minimizers^{\delta}$ shrink monotonically to $\Minimizers$ as $\delta \searrow 0$, the function $f(\delta)$ is monotonically decreasing for $\delta \searrow 0$;
hence, the limit $r \ceq \lim_{\delta \searrow 0} f(\delta) \geq 0$ exists.
We are left to show that $r = 0$.
\emph{Assume} that $r > 0$. 
We choose a minimizing sequence $(\Curve_n)_{n \in \N}$ of~$\EulerBernoulli$ as follows:
For each $n \in \N$, we pick $\Curve_n \in \Compactum \cap \Minimizers^{2C/n}$ with
$
	0 < \tfrac{1}{2} r \leq \tfrac{1}{2} f( \tfrac{1}{n}) \leq d(\Curve_n, \Minimizers) \leq f(\tfrac{1}{n}).
$
Since $\Compactum$ is compact and $\EulerBernoulli$ is continuous, this minimizing sequence $(\Curve_n)_{n \in \N}$ has a cluster point $\Curve_\infty \in \Minimizers = \Compactum\cap \Minimizers$. This leads to the \emph{contradiction}
$
	0 = \liminf_{n \to \infty} d(\Curve_n, \Minimizers) \geq \tfrac{1}{2} r > 0,
$
thus establishing the second part of \eqref{eq:Hsdff}, which completes the proof of \autoref{theo:ConvergenceTheorem}.
\end{proof}

%% file: SmoothSetting.tex

\section{Smooth Setting}
\label{sec:SmoothSetting}

In the following, $\Interval \ceq \intervalcc{0,L}$  will denote a compact interval of length $L >0$.
We denote the standard \emph{Sobolev spaces} of functions with $k$ weak derivatives in $L^p(\varSigma;\AmbSpace)$ by $\SoboC[k,p]$ 
and the space of \emph{Lipschitz immersions} into Euclidean space by
\begin{align*}
	\ImmC[1,\infty]
	\ceq \set{
		\Curve \in \SoboC[1,\infty] | 
		\nnorm{\LogStrain_\Curve}_{L^\infty} < \infty
	}.
\end{align*}
Here, $\LogStrain_\Curve$ denotes the \emph{logarithmic strain} of the curve $\Curve$, given by
\begin{align*}
\LogStrain_\Curve \ceq \log( \nabs{\Curve'}).
\end{align*}
This definition of Lipschitz immersions follows the one given in \cite{MR3008339}.\footnote{We would like to point out that our definition is significantly more general than the one used in \cite{MR2118960}, which requires that the constant rank theorem holds true for each Lipschitz immersion.}
Notice that the set of Lipschitz immersions $\ImmC[1,\infty]$ is an open subset of $\SoboC[1,\infty]$.
We denote the \emph{unit tangent vector field} by $\TangentC$.

For $k \geq 1$ and $p \geq1$ with $k-\frac{1}{p} \geq 1$, we define 
\begin{align*}
	\ImmC[k,p]	\ceq \ImmC[1,\infty] \cap \SoboC[k,p].
\end{align*}
Notice that $\ImmC[k,p]$ is an open subset of $\SoboC[k,p]$ because of the Morrey embedding $\SoboC[k,p]\hookrightarrow \SoboC[1,\infty]$. 

For $\Curve \in \ImmC[1,\infty]$, we denote by $\LineElementC (t) \ceq \nabs{\Curve'(t)} \, \dd t$ the line element of $\Curve$ on $\Interval$. For a sufficiently regular mapping $u \colon \Interval \to \AmbSpace$, we define the \emph{differential with respect to unit speed} $\Dtot u$ 
and the \emph{Laplace-Beltrami} operator $\Laplacian_\Curve u$ by
\begin{align*}
	\Dtot u \ceq \nabs{\Curve'}^{-1} \, u'
	\qand
	\Laplacian_\Curve u \ceq \Dtot^2 u. 
\end{align*}
If $\TangentC$ is weakly differentiable, then the \emph{curvature vector} $\CurvatureC$ of $\Curve$ is given by 
\begin{align*}
	\CurvatureC = \Dtot \TangentC = \Laplacian_\Curve \Curve.
\end{align*}

\subsection{Sobolev and TV Norms}
For a weakly differentiable function
$u \in \SoboC[k,p]$, $k \in \N \cup \set{0}$ (and also for scalar-valued functions), we denote the usual Sobolev seminorms by
\begin{align*}
	\textstyle
	\nseminorm{u}_{\Sobo[k,p]} \ceq \Bigparen{ \int_\varSigma \bigabs{\bigparen{\tfrac{\dd}{\dd t}}^k u(t)}^p \, \dd t }^\frac{1}{p}
	\;\text{for $p \in \intervalco{1,\infty}$}
	\qand
	\nseminorm{u}_{\Sobo[k,\infty]} \ceq \esssup_{t \in \Interval} \bigabs{\big(\tfrac{\dd}{\dd t}\big)^k u(t)}.
\end{align*}
We also consider slightly different seminorms that take into account the line element of a curve $\Curve$. Concretely, for $\Curve \in \ImmC[k,\infty]$, we define 
\begin{align*}
	\textstyle
	\nseminorm{u}_{\Sobo[k,p][\Curve]} \ceq \Big( \int_\varSigma \nabs{ \Dtot^k u(t)}^p \, \LineElementC(t)\Big)^\frac{1}{p},
	\;\text{for $p \in \intervalco{1,\infty}$}
	\qand
	\nseminorm{u}_{\Sobo[k,\infty][\Curve]} \ceq \esssup_{t \in \Interval} \nabs{\Dtot^k u(t)}.
\end{align*} 
Both seminorms, $\nseminorm{u}_{\Sobo[k,p]}$ and $\nseminorm{u}_{\Sobo[k,p][\Curve]}$, give rise to respective Sobolev norms in the usual manner.

Analogously, we introduce two versions of total variation seminorms for $k \in \N \cup \set{0}$ by
\begin{align*}
\nseminorm{u}_{\TV[k]} \ceq \nnorm{\bigparen{\tfrac{\dd}{\dd t}}^k u}_{\TV[0]} \qquad \text{and} \qquad \nseminorm{u}_{\TV[k][\Curve]} \ceq \nnorm{\Dtot^k u}_{\TV[0]}, 
\end{align*}
where $\nnorm{\cdot}_{\TV[0]} \ceq \nnorm{\cdot}_{(C^0)'}$ denotes the dual norm of $\nnorm{\cdot}_{C^0}$ on the space $\BVC[0] \ceq (C^0(\Interval;\AmbSpace))'$. The according Banach spaces for $k \geq 1$ are 
$\BVC[k] \ceq \set{\SoboC[k-1,1] | \nseminorm{u}_{\TV[k]}<\infty}$ with norms 
$\nnorm{u}_{\TV[k]} \ceq \nnorm{u}_{\Sobo[k-1,1]} + \nseminorm{u}_{\TV[k]}$
and
$\nnorm{u}_{\TV[k][\Curve]} \ceq \nnorm{u}_{\Sobo[k-1,1][\Curve]} + \nseminorm{u}_{\TV[k][
\Curve]}$ for $k \geq 1$.

The seminorms $\nseminorm{u}_{\Sobo[k,p][\Curve]}$ and $\nseminorm{u}_{\TV[k][\Curve]}$ depend on $\Curve$; they are in some sense more natural and often easier to handle since they employ metric information induced by the immersion.
They coincide with the usual ones whenever $\Curve$ is parameterized by arc length, i.e., when $\nabs{\Curve'(t)} = 1$ for almost all $t \in \Interval$. Of course it might seem preferable to work with arc-length parameterized curves throughout. However, this would lead to cumbersome calculations later on when we will have to compute derivatives in the space of curves. 
In the sequel, we will interchangeably work with both variants ($\Curve$-dependent and $\Curve$-independent), depending on which choice is more convenient. The following lemma ensures that the two variants are equivalent whenever the logarithmic strain $\LogStrain_\Curve$ is sufficiently regular. Additionally, this lemma shows that one can extend (i) the definition  of $\nseminorm{u}_{\Sobo[k,p][\Curve]}$ from $\Curve \in \ImmC[k,\infty]$ to $\Curve \in \ImmC[k,p]$
and (ii) the definition of $\nseminorm{u}_{\TV[k][\Curve]}$ from $\Curve \in \ImmC[k,\infty]$ to $\Curve \in \ImmC[1,\infty] \cap \BVC[k]$. 

\begin{lemma}[Norm Equivalences]\label{lem:NormEquivalences}
Let $k \in \N \cup \{0\}$, $p \in \intervalcc{1,\infty}$.  Then for each $\varLambda \geq 0$ and $\EnergyBound \geq 0$ there is a $C \geq 0$, such that for each $\Curve \in \ImmC[1,\infty] \cap \SoboC[k,p]$ with $\nnorm{\LogStrain_\Curve}_{L^\infty} \leq \varLambda$, the following two statements hold true for all $u \in \SoboC[k,p]$:
\begin{enumerate}
	\item 
	If 
	$\nnorm{\LogStrain_\Curve}_{\Sobo[\max(0,k-1),p][\Curve]} \leq \StrainBound$,
	then
$
	\nnorm{u}_{\Sobo[k,p]} \leq C \, \nnorm{u}_{\Sobo[k,p][\Curve]}.
$
	\item 
	If $\nnorm{\LogStrain_\Curve}_{\Sobo[\max(0,k-1),p]} \leq \StrainBound$,
	then
$
	\nnorm{u}_{\Sobo[k,p][\Curve]} \leq C \, \nnorm{u}_{\Sobo[k,p]}
	.
$
\end{enumerate}
Moreover, if $\Curve \in \ImmC[1,\infty] \cap \BVC[k+1]$ and $u \in  \BVC[k+1]$, then one even has:
\begin{enumerate}
	\setcounter{enumi}{2}
	\item 
	If 
	$\nnorm{\LogStrain_\Curve}_{\TV[k][\Curve]} \leq \StrainBound$,
	then
$
	\nnorm{u}_{\TV[k+1]} \leq C \, \nnorm{u}_{\TV[k+1][\Curve]}.
$
	\item 
	If $\nnorm{\LogStrain_\Curve}_{\TV[k]} \leq \StrainBound$,
	then
$
	\nnorm{u}_{\TV[k+1][\Curve]} \leq C \, \nnorm{u}_{\TV[k+1]}
	.
$
\end{enumerate}
\end{lemma}
The proof is deferred to \autoref{sec:NormEquivalences}.

\subsection{Smooth Optimization Problem}
\label{sec:SmoothOptimizationProblem}

For the well-definedness of the bending energy \eqref{eq:EulerBernoulliEnergy}, we require $\Curve \in \ImmC[2,2]$. Due to our convergence results in $\Sobo[2,p]$, we work within the slightly more general setting  $\Curve \in \ImmC[2,p]$, $p \geq 2$. Accordingly, for $p \in \intervalcc{2,\infty}$, we define the \emph{configuration space} $\ConfSpace \ceq \ImmC[2,p]$.
We encode the equality constraints into the smooth mapping
\begin{gather}
	\ConstraintMap \colon \ConfSpace \to \TargetSpace \ceq \AmbSpace \times \AmbSpace \times \Sphere \times \Sphere \times \Sobo[1,p][][\Interval],
	\notag
	\\
	\ConstraintMap(\Curve) = \bigparen{ 
	    	\Curve(0) - \DCond(0),
		\Curve(L) - \DCond(L),
		\Tangent_\Curve(0) - \NCond(0),
		\Tangent_\Curve(L) - \NCond(L),
		\LogStrain_\Curve
	}.
	\label{eq:DefTargetSpace}
\end{gather}
Here,  $\Sphere \subset \AmbSpace$ denotes the unit sphere and
$\DCond \colon \partial \Interval \to \AmbSpace$ and $\NCond \colon \partial \Interval \to \Sphere$ represent the prescribed boundary conditions.
We define the \emph{feasible set} $\Feasible$ by
\begin{align*}
	\Feasible \ceq \set{ \Curve \in \ConfSpace | \ConstraintMap(\Curve) = 0}
\end{align*}
and consider the following optimization problem:
\begin{problem}\label{prob:SmoothEulerBernoulli}
Minimize the functional $\EulerBernoulli$ on $\Feasible$, i.e., among all $\Curve \in \ConfSpace$ subject to
$\ConstraintMap(\Curve) = 0$.
\end{problem}

Solutions to this problem are automatically parameterized by arc length (since $\LogStrain_\Curve=0$). 
The existence of solutions $\Curve \in \ImmC[2,2]$ for commensurable constraints can be shown with the direct method of calculus of variations. 
Up to reparameterization, the solutions of \autoref{prob:SmoothEulerBernoulli} are precisely
the solutions to the classical Euler elastica problem:
\begin{problem}\label{prob:SmoothEulerBernoulli2}
Minimize the functional $\EulerBernoulli$ among all $\Curve \in \ConfSpace$ subject to
$\Curve|_{\partial \Interval} = \DCond$,
$\Tangent_\Curve|_{\partial \Interval} = \NCond$,
and $\Length(\Curve) = L$.
\end{problem}
We now turn to the regularity of solutions of these problems.

\subsection{Smooth Regularity}\label{sec:SmoothRegularity}

We show that our \apriori assumptions \eqref{eq:SmoothPriors} on minimizers of the smooth elastica problem are indeed valid. The validity of these assumptions could be verified by various means, e.g., by invoking elliptic integrals. Here we present a functional analytic approach since this approach can be mimicked in the discrete case (which appears to be new). Our proof of the validity of the \apriori assumptions hinges on what could be called ``elliptic bootstrapping''. For this purpose we require bounds on the Lagrange multipliers that arise in our constraint optimization problem. In order to be useful for our purpose, these bounds must depend only on the energy of the minimizer and the geometric constraints. Before going into detail, we briefly outline the general strategy of our approach. 

Let $\Curve$ be a critical point of \autoref{prob:SmoothEulerBernoulli}. Then the Euler-Lagrange equations read
\begin{align*}
	D\EulerBernoulli(\Curve) 
	+ D \ConstraintMap(\Curve)' \lambda = 0
	\qand
	\ConstraintMap(\Curve) = 0,
\end{align*}
where $\lambda$ is the Lagrange multiplier and $D \ConstraintMap(\Curve)'$ denotes the dual of the constraint differential. Suppose that $B_\Curve$ is a bounded right inverse of the differential $D \ConstraintMap(\Curve)$. Then $B_\Curve'$ is a left inverse of the dual  constraint differential $D \ConstraintMap(\Curve)'$. Hence, left-multiplying the previous equation by $B_\Curve' $ leads to
\begin{align*}
	B_\Curve' \, D\EulerBernoulli(\Curve) 
	+  \lambda = 0	\ .
\end{align*}
Therefore, since $B_\Curve$ is bounded by assumption, the norm of $\lambda$ is controlled by the dual norm $\norm{D\EulerBernoulli(\Curve)}_{(\Sobo[2,2][\Curve])'}$. 
A straight-forward computation (see \autoref{lem:EulerBernoulliisLipschitz} in \autoref{sec:EulerBernoulliisLipschitz}) shows that this dual norm is bounded by $C \, \EulerBernoulli(\Curve)$, where $C$ depends only on the length $L$ of the curve. 
Below we verify the assumption that there exists a right inverse $B_\Curve$ that is bounded in terms of the energy of the critical point and constants that depend only on the geometric constraints. 
This implies that also $\lambda$ can be bounded in terms of these quantities, which in turn allows for extracting quantitative \apriori information for critical points in terms of bounds  in various norms (see \autoref{lem:SmoothRegularity} below). 

For the purpose of constructing a right inverse $B_\Curve$, we define the set of sufficiently ``tame immersions''  by
\begin{align}
	\Tame^{k,p}(\StrainBound,\EnergyBound,\eta)
	\ceq 
	\set{
		\Curve \in \ConfSpace |
			\nnorm{\LogStrain_\Curve}_{\Sobo[k-1,p][\Curve]} \!\leq\! \StrainBound,
			\,
			\nnorm{\Tangent_\Curve}_{\Sobo[k-1,p][\Curve]} \!\leq\! \EnergyBound,
			\,
			\Length(\Curve) \!\geq\! (1\!+\!\eta)\,\nabs{\Curve(L) \!-\! \Curve(0)}
	}	\label{eq:ThetaConditions}
\end{align}
with $p \in \intervalcc{2,\infty}$, $k \in \N$, $k \geq 2$, $\StrainBound \geq 0$, $\EnergyBound \geq 0$, and $\eta>0$. Let the target space $\TargetSpace$  of the constraint mapping (see \eqref{eq:DefTargetSpace}) be equipped with the usual Euclidean norm for the boundary conditions and with the $\Sobo[1,p]$-norm for the logarithmic strain $\LogStrain_\Curve$. By a slight abuse of notation we simply refer to the resulting norm on the product spaces as the $\Sobo[1,p]$-norm.

\begin{blemma}[Right Inverse of $D\ConstraintMap$]\label{lem:SmoothRightinverseDPhi}
There is a constant $C > 0$ and a right inverse $B_\Curve$ of $D\ConstraintMap(\Curve)$ such that the mapping 
$\Curve \mapsto B_\Curve$ is
uniformly bounded 
as a map from the space $\Tame^{2,p}(\StrainBound,\EnergyBound,\eta)$ to the space of bounded linear mappings $L(T\TargetSpace;\SoboC[2,p])$, where the latter is equipped with the operator norm; i.e., $\nnorm{B_\Curve}_{\Sobo[1,p] \to \Sobo[2,p]} \leq C$.
\end{blemma}
\begin{proof}
As before, let $\Tangent_\Curve$ denote the unit tangent of $\Curve$. Furthermore, let $\prnor(r)$ denote the orthoprojector onto the orthogonal complement of $\TangentC(r)$. Recall the explicit form of the constraint mapping from \eqref{eq:DefTargetSpace}:
\begin{gather}
		\ConstraintMap(\Curve)
	= \bigparen{
		\Curve(0) - \DCond(0),
		\Curve(L) - \DCond(L),
		\cD_\Curve \Curve(0) - \NCond(0),
		\cD_\Curve \Curve(L) - \NCond(L),
		\log(\nabs{\Curve'})
	}.
	\notag
\end{gather}
Its differential is given by
\begin{align*}
	D\ConstraintMap(\Curve)\,u
	= \bigparen{
		u(0),
		u(L),
		\Dnor u(0),
		\Dnor u(L),
		\ninnerprod{\TangentC,\Dtot u}
	}.	
\end{align*}
Now let $\Curve \in \Tame^{2,p}(\StrainBound,\EnergyBound,\eta)$ and
$(U_0,U_1,V_0,V_1,\lambda)$ be an element of the tangent space $T_{\ConstraintMap(\Curve)} \TargetSpace$ and let $\varphi_\Curve(r) \ceq \Length(\Curve|_{\intervalcc{0,r}})/\Length(\Curve)$ be the normalized arc length parameter of $\Curve$. We construct
a right inverse $B_\Curve (U_0,U_1,V_0,V_1,\lambda) \ceq u$ as follows:
\begin{align}
	u(t)
	&\ceq 
	\textstyle
	U_0 + 
	\int_{\intervalcc{0,t}} 
		\lambda(r) \, \TangentC(r)
	\, \LineElementC(r)
	\notag\\
	&\qquad 
	\textstyle
	+\int_{\intervalcc{0,t}} 
		\prnor(r) \,  \Bigparen{ 
			(1-\varphi_\Curve(r))^2 \, V_0
			+ (1-\varphi_\Curve(r))\,\varphi_\Curve(r) \,V
			+ \varphi_\Curve(r)^2 \, V_1
		}
	\, \LineElementC(r),
	\label{eq:DefinitionofRightinverse}
\end{align}
where $V$ will be specified below. By construction, this vector field $u$ along $\Curve$ satisfies 
\begin{align*}
	u(0) = U_0, 
	\quad 
	\Dnor u(0) = V_0,
	\quad 
	\Dnor u(L) = V_1,
	\qand
	\ninnerprod{\TangentC(t) , \Dtot u(t)} = \lambda(t).
\end{align*}
It remains to choose $V$ such that $u(L) = U_1$. This requires us to solve the linear equation $\varTheta_\Curve\,V = b_\Curve$,  where
\begin{align}
	\label{eq:DefTheta}
	\varTheta_\Curve
	 &\ceq  \textstyle
	 \int_\Interval (1-\varphi_\Curve(r))\, \varphi_\Curve(r)\, \prnor(r) \, \LineElementC(r)
	 \in \Hom(\AmbSpace;\AmbSpace)
	\qand	
	\\
	\notag
	b_\Curve 
	&\ceq  \textstyle
	U_1 - U_0 - \!\int_\Interval
	\Bigparen{
		\lambda(r) \, \TangentC(r)
		+ \prnor(r) \,  \bigparen{
			(1-\varphi_\Curve(r))^2 \, V_0
			+ \varphi_\Curve(r)^2 \, V_1
		}
	}\, \LineElementC(r).
\end{align}
First, we prove that $\varTheta_\Curve$  is invertible for all $\Curve \in \Tame^{2,p}(\StrainBound,\EnergyBound,\eta)$.
\emph{Assume the contrary.} 
Then there is a unit vector $V \in \Sphere$ such that $\varTheta_\Curve\, V = 0$. 
This implies
\begin{align*}
    \textstyleswitcha
	0 = \ninnerprod{V, \varTheta_\Curve\, V}
	= \int_\Interval (1-\varphi_\Curve(r))\, \varphi_\Curve(r)\, \nabs{V - \TangentC(r) \, \ninnerprod{\TangentC(r),V}}^2 \, \LineElementC(r).
\end{align*}
Since $(1-\varphi_\Curve(r))\, \varphi_\Curve(r)$ is strictly positive almost everywhere and since $r \mapsto \nabs{V - \TangentC(r) \, \ninnerprod{\TangentC(r),V}}^2$ is continuous and nonnegative, we have
$V - \TangentC(r) \, \ninnerprod{\TangentC(r),V} = 0$ for all $r \in \Interval$. 
Hence $\TangentC$ is parallel to $V$. Continuity implies that $\Curve$ is a straight line. 
But this is ruled out by the third condition in~\eqref{eq:ThetaConditions}. Thus, $\varTheta_\Curve$ must be invertible. 
This shows that the function $\Curve \mapsto \nnorm{\varTheta_\Curve^{-1}}$ is well-defined and continuous with respect to the $\Sobo[1,\infty]$\hbox{-}topology  on $\Tame^{2,p}(\StrainBound,\EnergyBound,\eta)$. 
Thus it has a finite maximum on the compact subset 
$\set{\Curve \in \Tame^{2,p}(\StrainBound,\EnergyBound,\eta) | \Curve(0) = 0}$ with respect to the $\Sobo[1,\infty]$\hbox{-}norm. Indeed, this subset is bounded in the $\Sobo[2,p]$-norm by \autoref{lem:NormEquivalences} and compactness follows from the compact embedding $\Sobo[2,p] \hookrightarrow \Sobo[1,\infty]$.
As $\Curve \mapsto \nnorm{\varTheta_\Curve^{-1}}$ is invariant under translations of $\Curve$ in $\AmbSpace$, this proves the claim.

Thus, $\varTheta_\Curve$ is always invertible and the inverse satisfies an estimate of the form
$\nnorm{\varTheta_\Curve^{-1}} \leq C$.
Therefore, $B_\Curve$ is indeed a bounded right inverse of $D\ConstraintMap(\Curve)$.
By substituting $V = \varTheta_\Curve^{-1} \, b_\Curve$ into~\eqref{eq:DefinitionofRightinverse}
and by applying standard Hölder estimates and the Sobolev embedding $\Sobo[2,p] \hookrightarrow \Sobo[1,\infty]$, a~straight-forward calculation shows that $\nnorm{B_\Curve}_{\Sobo[1,p] \to \Sobo[2,p]}$ is uniformly bounded.
\end{proof}

We now study regularity of minimizers. In particular, the following result justifies our \apriori assumptions. Notice that this result requires that the distance between the prescribed end points $\DCond=\Curve|_{\partial \Interval}$ of the curve is strictly less than the curve length $L$. This assumption is necessary in order to rule out the situation where the distance between the end points is equal to $L$ but the first order boundary conditions $\NCond = \Tangent_\Curve|_{\partial \Interval}$ are not those of a straight line segment, in which case in which case the feasible set $\Feasible$ is empty.

\begin{btheorem}[Regularity of Minimizers] \label{lem:SmoothRegularity}
For each $L = \nabs{\Interval} >0$, $\eta \ceq L / \nabs{\DCond(L) - \DCond(0)}-1 >0$, $\EnergyBound \geq 0$, and $k\geq 1$
there is a $C\geq 0$ such that 
each critical point $\Curve \in \Feasible$ of \autoref{prob:SmoothEulerBernoulli} 
with $\EulerBernoulli(\Curve) \leq \EnergyBound$ 
is contained in $\SoboC[k,\infty]$ and
also satisfies
\begin{align*}
	\seminorm{\Curve}_{\Sobo[k,\infty][\Curve]} 
	\leq C.
\end{align*}
In particular, this means that the minimizers $\Minimizers$ of \autoref{prob:SmoothEulerBernoulli} 
satisfy the \apriori assumptions of the form \eqref{eq:SmoothPriors}, i.e., $\Minimizers \subset \Priors$.
\end{btheorem}
\begin{proof}
Using the direct method of the calculus of variations, it follows that minimizers must exist in $\SoboC[2,2]$. So let $\Curve$ be a critical point of \autoref{prob:SmoothEulerBernoulli}. 
Because $D\ConstraintMap(\Curve)$ admits a right-inverse $B_\Curve$
(see \autoref{lem:SmoothRightinverseDPhi} for $p = 2$ ), there are Lagrange multipliers $\lambda \in (\SoboC[1,2])'$, $\mu_{00}  \in \AmbSpace$, $\mu_{10} \in \AmbSpace$, 
$\mu_{01} \in T_{N(0)} \Sphere^{m-1}$, 
and
$\mu_{11} \in T_{N(L)} \Sphere^{m-1}$,
such that
\begin{align}
	D\EulerBernoulli(\Curve) \, u 
	+ \ninnerprod{\lambda, \ninnerprod{\Curve',u'}_{\AmbSpace}}
	+ \ninnerprod{\mu_{00}, u(0)}	
	+ \ninnerprod{\mu_{10}, u(L)}		
	+ \ninnerprod{\mu_{01}, u'(0)}	
	+ \ninnerprod{\mu_{11}, u'(L)}	
	= 0	
	\label{eq:KKTCondition}
\end{align}
holds for all $u \in \SoboC[2,2]$. Moreover, $B_\Curve$ and $\norm{D \EulerBernoulli(\Curve)}_{(\Sobo[2,2])'}$ are uniformly bounded for all $\Curve \in \Tame^{2,2}(0,\EnergyBound,\eta)$. Hence, the norms of these Lagrange multipliers are bounded by a constant $C$ that depends only on the geometric quantities $L$, $\eta$ and on the energy bound $\EnergyBound$. 

Writing $\Curve(t) = \Curve(0) + \int_{0}^t \TangentC(r)  \, \dd r$ 
(notice that $\Curve$ is parameterized by arc length) and testing only against
$u$ of the form
$
	u(t) = \int_{0}^t v(r)  \, \dd r
$
with $v(r) \perp \TangentC(r)$, $v(0) =0$, and $v(L) = 0$ leads to
\begin{align}
	\textstyle
	\int_\varSigma \ninnerprod{\TangentC', v'}_\AmbSpace \, \dd t
	+ \int_\varSigma \ninnerprod{\mu_{10}, v}_\AmbSpace \,\dd t
	= 0.
	\label{eq:IndicatrixBootstrappingWeakEquation}
\end{align}
On the one hand, we may integrate by parts and obtain
\begin{align}\label{eq:SmoothRegularity1}
	\prnor \TangentC'' = \prnor \, \mu_{10},
\end{align} 
where $\prnor(r)$ denotes the projection onto the orthogonal complement of $\Tangent_\Curve(r)$.
On the other hand, we have
$
	\ninnerprod{\TangentC,\TangentC''}
	= \tfrac{\dd}{\dd t} \ninnerprod{\TangentC,\TangentC'}
	 - \ninnerprod{\TangentC',\TangentC'}
	= - \ninnerprod{\TangentC',\TangentC'}
$.
Combining these two equations, we obtain the following second order ODE for $\TangentC$
\begin{align}
	\TangentC'' 
	= \prnor \, \TangentC'' + \TangentC \, \ninnerprod{\TangentC,\TangentC''}	
	= \prnor \, \mu_{10}  - \TangentC \, \ninnerprod{\TangentC',\TangentC'}
	=\mu_{10} - \TangentC \, \ninnerprod{\TangentC,\mu_{10}} - \TangentC \, \ninnerprod{\TangentC',\TangentC'}.
	\label{eq:IndicatrixBootstrappingEquation}
\end{align}
Notice that the right hand side is a member of $L^1(\Interval;\AmbSpace)$. ``Elliptic bootstrapping''\footnote{We point out that elliptic bootstrapping in Sobolev spaces $\SoboC[k,p]$ in dimension $\dim(\varSigma) = 1$ is -- contrary to higher dimensional domains -- indeed possible for all $p\in [1,\infty]$.}
yields $\TangentC \in \SoboC[2,1] \subset \SoboC[1,\infty]$ and
$\nseminorm{\Curve''}_{L^\infty} = \nseminorm{\TangentC}_{\Sobo[1,\infty]} \leq C$ for some $C$ that again depends only on the boundary conditions and the energy bound.
Now, the right hand side of \eqref{eq:IndicatrixBootstrappingEquation} and thus $\TangentC''$ lies in $L^\infty$, hence $\TangentC \in \SoboC[2,\infty]$.
Thus, the right hand side of \eqref{eq:IndicatrixBootstrappingEquation} lies in $\SoboC[1,\infty]$ so that a further bootstrapping step leads to $\TangentC \in \SoboC[3,\infty]$
and $\nseminorm{\Curve}_{\Sobo[4,\infty]} = \nseminorm{\TangentC}_{\Sobo[3,\infty]} \leq C$ for some constant $C$, and so forth.
\end{proof}

\begin{remark}
One can show in the same way that
the maximal regularity 
for a clamped elastic curve $\Curve$ subject to point loads is $\Curve \in \BVC[4] \subset \SoboC[3,\infty]$. So the \apriori assumptions~$\Priors$ from \eqref{eq:SmoothPriors} are still valid. Indeed, also \autoref{theo:ConvergenceTheorem} holds true for more general optimization problems involving objective functions of the form $\cF = \EulerBernoulli + \cG$ with a sufficiently well-behaved energy $\cG$ of ``lower order''. For the sake of brevity, we focus here only on the classical elastica problem.
\end{remark}

%% file: DiscreteSetting.tex

\section{Discrete Setting}
\label{sec:DiscreteSetting}

In our exposition of the discrete setting, we aim at mimicking the smooth setting as closely as possible. We first introduce some basic notation that we require throughout.

\paragraph*{Notation} Let $\Triangulation$ be a finite partition of $\Interval=[0,L]$, i.e., a finite decomposition into compact intervals. We denote by $\Vertices(\Triangulation)\subset\varSigma$ the set of vertices and by $\IVertices(\Triangulation)\ceq V(\Triangulation)\setminus \set{0,L}$ the set of interior vertices. Moreover, we denote by $\Edges(\Triangulation)$
the set of edges, by $\BEdges(\Triangulation)$ the set of those edges that contain a boundary vertex ($0$ or $L$), and by $\IEdges(\Triangulation)$ the set of edges that do not touch the boundary of $\Interval$.

For a vertex $\Vertex \in \IVertices(\Triangulation)$ and an edge $\Edge \in \Edges(\Triangulation)$, we introduce the following shift notation:
\begin{itemize}
	\item $\VprevE{\Vertex}$ and $\VnextE{\Vertex} \in \Edges(\Triangulation)$ 
	for the preceding and following edge of vertex $\Vertex$ and
	\item $\EprevV{\Edge}$ and $\EnextV{\Edge} \in \Vertices(\Triangulation)$ for the left and right boundary vertices of edge $\Edge$.
\end{itemize}
With reference to \autoref{fig:Notation_Discrete}, we extend this notation transitively, i.e., $\VprevV{\Vertex}$ and $\VnextV{\Vertex}$ stands for the vertices before and after vertex $\Vertex$ and $\EprevE{\Edge}$ and $\EnextE{\Edge}$ stands for the preceding and following edges of edge $\Edge$, respectively. We also apply this shift notation to functions
$\varphi \colon \Vertices(\Triangulation) \to \R^m$ and
$\psi\colon \Edges(\Triangulation) \to \R^m$ via pullback, i.e., we
define
$\EprevV{\varphi}(\Edge) \ceq \varphi(\EprevV{\Edge})$,
$\VnextV{\varphi}(\Vertex) \ceq \varphi(\VnextV{\Vertex})$,
$\VnextE{\psi}(\Vertex) \ceq \psi(\VnextE{\Vertex})$,
$\EprevE{\psi}(\Edge) \ceq \psi(\EprevE{\Edge})$, etc.

\begin{figure}
\capstart
\begin{center}
	\def\svgwidth{0.7 \textwidth}
	\input{./Pictures/Notation_Discrete_pdf.tex}
\end{center}
\caption{Sketch of the notation used for polygonal lines.}
\label{fig:Notation_Discrete}
\end{figure}
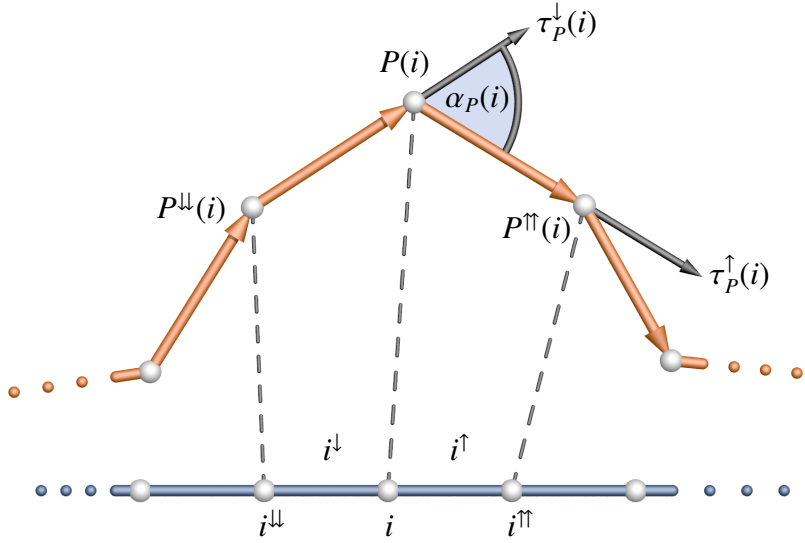%

We define the \emph{edge midpoint} $\EdgeMidpoints(\Edge) $ of $\Edge \in \Edges(\Triangulation)$ and the \emph{dual edge} $\DualEdge(\Vertex)$ of $\Vertex \in \IVertices(\Triangulation)$ as follows:
\begin{align*}
	\EdgeMidpoints(\Edge) \ceq \tfrac{1}{2}(\EprevV{\Edge}+\EnextV{\Edge})
	\qand
	\DualEdge(\Vertex) \ceq 
	\nintervalcc{\VprevE{\EdgeMidpoints}(\Vertex),\VnextE{\EdgeMidpoints}(\Vertex)}
	.
\end{align*}
The \emph{reference edge lengths} $\ReferenceEdgeLengths \colon \Edges(\Triangulation) \to \intervalco{0,\infty}$
of the partition $\Triangulation$ and its \emph{dual reference edge lengths} 
$\DualReferenceEdgeLengths \colon \IVertices(\Triangulation) \to \intervalco{0,\infty}$
are defined by
\begin{align*}
	\ReferenceEdgeLengths(\Edge) \ceq \nabs{\Edge} = \shiftr{\Edge} - \shiftl{\Edge}
	\qand
	\DualReferenceEdgeLengths(\Vertex) \ceq 
	\nabs{\DualEdge(\Vertex)} = 
	\tfrac{1}{2} \nparen{
		 \shiftl{\ReferenceEdgeLengths}(\Vertex) 
		 + 
		 \shiftr{\ReferenceEdgeLengths}(\Vertex)
		 }.
\end{align*}
We measure the \emph{resolution} of the partition $\Triangulation$ by 
\begin{align*}
	\MaxRadius(\Triangulation) \ceq \sup_{\Edge \in \Edges(\Triangulation)}\ReferenceEdgeLengths(\Edge).
\end{align*}
Each polygon $\Polygon \colon \Vertices(\Triangulation) \to \AmbSpace$ induces its own \emph{edge lengths} $\EdgeLengths_\Polygon \colon \Edges(\Triangulation) \to \intervalco{0,\infty}$
and its own
\emph{dual edge lengths} $\DualEdgeLengths_\Polygon \colon \IVertices(\Triangulation)  \to \intervalco{0,\infty}$
by
\begin{align*}
	\EdgeLengths_\Polygon(\Edge) \ceq \nabs{\Polygon(\EnextV{\Edge}) - \Polygon(\EprevV{\Edge}) }
	\qand
	\DualEdgeLengths_\Polygon(\Vertex) \ceq 
	\tfrac{1}{2} \bigparen{ 
			\shiftl{\EdgeLengths_\Polygon}(\Vertex) + \shiftr{\EdgeLengths_\Polygon}(\Vertex) 
	}.
\end{align*}

\paragraph*{Discrete differential operators and configuration spaces}
Analogously to the smooth setting, we define the configuration space $\DiscConfSpace$ of \emph{discrete immersions} by
\begin{align*}
	\DiscConfSpace & \ceq \set{\Polygon \colon V(\Triangulation) \to \AmbSpace | \forall \Edge \in \Edges(\Triangulation) \colon \; \nabs{\LogStrain_\Polygon(\Edge)}  < \infty},
\end{align*}
where we denote the \emph{discrete logarithmic strain} per edge by
\begin{align*}
	\LogStrain_\Polygon(\Edge) \ceq \log ( \EdgeLengths_\Polygon(\Edge)/\ReferenceEdgeLengths(\Edge)).	
\end{align*}
For functions
$\varphi \colon \Vertices(\Triangulation) \to \R^m$ 
and
$\psi\colon \Edges(\Triangulation) \to \R^m$, we introduce the difference operators
\begin{equation*}
	\begin{aligned}
		\cD_0 \, \varphi (\Edge)\ceq \tfrac{\EnextV{\varphi}(\Edge)-\EprevV{\varphi}(\Edge)}{\ReferenceEdgeLengths(\Edge)}
	\\
		\cD_0\, \psi(\Vertex) \ceq \tfrac{\VnextE{\psi}(\Vertex) -\VprevE{\psi}(\Vertex) }{\DualReferenceEdgeLengths(\Vertex) }
	\end{aligned}
	\qquad\text{and}\qquad
	\begin{aligned}
		\DiscD \, \varphi (\Edge)\ceq \tfrac{\EnextV{\varphi}(\Edge)-\EprevV{\varphi}(\Edge)}{\EdgeLengths_\Polygon(\Edge)}\\
		\DiscD \, \psi(\Vertex) \ceq \tfrac{\VnextE{\psi}(\Vertex) -\VprevE{\psi}(\Vertex) }{\DualEdgeLengths_\Polygon(\Vertex) }
	\end{aligned}.
\end{equation*}
We drop the dependence on $\Vertex$ and $\Edge$ in the sequel. The \emph{discrete Laplace-Beltrami operators} are defined as
\begin{align*}
	\textstyle
	\DiscLaplacian \varphi
	= \DiscD \, \DiscD \, \varphi = \frac{1}{\DualEdgeLengths_\Polygon} \Bigparen{
		\frac{\VnextV{\varphi} - \varphi}{\VnextE{\EdgeLengths_\Polygon}} 
		- 
		\frac{\varphi - \VprevV{\varphi}}{\VprevE{\EdgeLengths_\Polygon}}
	}
	\qand
	\DiscLaplacian \psi
	= \DiscD \, \DiscD \, \psi = 
	\frac{1}{\EdgeLengths_\Polygon} \Bigparen{
		\frac{\EnextE{\psi} - \psi}{\EnextV{\DualEdgeLengths_\Polygon}} 
		- 
		\frac{\psi - \EprevE{\psi}}{\EprevV{\DualEdgeLengths_\Polygon}}
	}.
\end{align*}
For $\Polygon \in \DiscConfSpace$, the \emph{unit edge vectors} are given by
\begin{align*}
	\EdgeVectors_\Polygon \colon \Edges(\Triangulation) \to \mathbb{S}, 
	\quad \EdgeVectors_\Polygon(\Edge) 
	= \DiscD(\Polygon)(\Edge).
\end{align*}
The \emph{discrete curvature vectors} $\Curvature_\Polygon \colon \IVertices(\Triangulation) \to \AmbSpace$ can be written as
\begin{align*}
	\Curvature_\Polygon(\Vertex) 
	= \DiscD \Tangent_\Polygon
	= 	\DiscLaplacian \Polygon(\Vertex).
\end{align*}

For each $\Vertex \in \Vertices(\Triangulation)$ let $\lambda_\Vertex \in \Sobo[1,\infty](\varSigma;\AmbSpace)$ be the unique piecewise linear and continuous function satisfying $\lambda_\Vertex(\Vertex)=1$ and $\lambda_\Vertex(\OtherVertex) = 0$ for all $\OtherVertex \in \Vertices(\Triangulation) \setminus \{\Vertex\}$. We define
the interpolation operator
\begin{align}
	\textstyle
	\DiscTestMap \colon \Map(\Vertices(\Triangulation);\R^{m}) \to \Sobo[1,\infty](\varSigma;\AmbSpace),
	\quad
	\Polygon \mapsto \sum_{\Vertex \in \Vertices(\Triangulation)} \lambda_\Vertex \, \Polygon(\Vertex).
	\label{eq:piecewiselinear}
\end{align}
Notice that for each $\Polygon \in \Map(\Vertices(\Triangulation);\R^{m})$, the image of $\DiscTestMap(\Polygon)$ is a polygonal line interpolating the points $\Polygon(\Vertices(\Triangulation))$.

\subsection{Discrete Sobolev and TV Norms}\label{sec:DiscreteNorms}
For $p \in \intervalco{1,\infty}$ (with the typical extensions for $p = \infty$), we denote the discrete $\ell^p$-norms by
\begin{align*}
	\nnorm{\varphi}_{\ell^p}
	&= \bigparen{\textstyle \sum_{\Vertex \in \Vertices(\Triangulation)} \nabs{\varphi(\Vertex)}^p \, \DualReferenceEdgeLengths(\Vertex) }^\frac{1}{p}, \\	
	\nnorm{\psi}_{\ell^p}
	&= \bigparen{\textstyle \sum_{\Edge \in \Edges(\Triangulation)} \nabs{\psi(\Edge)}^p \, \ReferenceEdgeLengths(\Edge) }^\frac{1}{p}.
\end{align*}
As in the smooth setting, we additionally consider slightly different norms that take into account the line element of a discrete curve:
\begin{align*}
	\nnorm{\varphi}_{\ell^p_\Polygon}
	&= \bigparen{\textstyle \sum_{\Vertex \in \Vertices(\Triangulation)} \nabs{\varphi(\Vertex)}^p \, \DualEdgeLengths_\Polygon(\Vertex) }^\frac{1}{p}, \\	
	\nnorm{\psi}_{\ell^p_\Polygon}
	&= \bigparen{\textstyle \sum_{\Edge \in \Edges(\Triangulation)} \nabs{\psi(\Edge)}^p \, \EdgeLengths_\Polygon(\Edge) }^\frac{1}{p}.
\end{align*}
Likewise, we define \emph{discrete Sobolev seminorms} by
\begin{equation*}
	\begin{aligned}
	\nseminorm{\varphi}_{\sobo[k,p]} 
	&= \nnorm{\cD_0^k \varphi}_{\ell^p}\\
	\nseminorm{\psi}_{\sobo[k,p]} 
	&= \nnorm{\cD_0^k\psi}_{\ell^p}
	\end{aligned}
	\qquad\text{and}\qquad
	\begin{aligned}
	\nseminorm{\varphi}_{\sobo[k,p][\Polygon]}
	&= \nnorm{\DiscD \varphi}_{\ell^p_\Polygon} \\
	\nseminorm{\psi}_{\sobo[k,p][\Polygon]}
	&= \nnorm{\DiscD \psi}_{\ell^p_\Polygon}
	\end{aligned}
	\ .
\end{equation*}
Finally, the analogues of the smooth total variation norms $\nseminorm{\cdot}_{\TV[k]}$ and $\nseminorm{\cdot}_{\TV[k][\Curve]}$ are given by
\begin{align*}
\nseminorm{\cdot}_{\tv[k]}:=\nseminorm{\cdot}_{\sobo[k,1]} \qquad\text{and}\qquad
\nseminorm{\cdot}_{\tv[k][\Polygon]}:=\nseminorm{\cdot}_{\sobo[k,1][\Polygon]},
\end{align*}
which we refer to as \emph{discrete TV seminorms}.

\subsection{Discrete Optimization Problem}

In analogy to the constraints in the smooth setting, we define
\begin{gather}
	\DiscConstraintMap \colon \DiscConfSpace 
	\to
	\DiscTargetSpace \ceq
	\AmbSpace \times \AmbSpace
	\times 
	\Sphere \times \Sphere
	\times
	\Map(\Edges(\Triangulation);\R),
	\notag
	\\
	\DiscConstraintMap(\Polygon)
	\ceq ( 
		\Polygon(0) - \DCond(0),
		\Polygon(L) - \DCond(L),
		\EdgeVectors_\Polygon(\shiftr{0}) - \NCond(0),
		\EdgeVectors_\Polygon(\shiftl{L}) - \NCond(L),
		\LogStrain_\Polygon
		)\label{eq:DefDiscTargetSpace}
\end{gather}
for prescribed $\DCond \colon \partial \Interval \to \AmbSpace$ and
$\NCond \colon \partial \Interval \to \Sphere$.
We denote the \emph{discrete feasible set} by
\begin{align*}
	\DiscFeasible \ceq \set{\Polygon \in \DiscConfSpace | \DiscConstraintMap(\Polygon) = 0}.
\end{align*}
With the \emph{turning angles}
\begin{align*}
	\TurningAngles_\Polygon \colon \IVertices(\Triangulation) \to \intervalcc	{0, \uppi},
	\quad
	\TurningAngles_\Polygon(\Vertex) \ceq \bigAngle{\VprevE{\EdgeVectors_\Polygon}(\Vertex)}{ \VnextE{\EdgeVectors_\Polygon}(\Vertex)},
\end{align*}
we define the \emph{discrete Euler-Bernoulli energy}
\begin{align*}
	\DiscEulerBernoulli\colon \DiscConfSpace \to \R,
	\quad
	\DiscEulerBernoulli(\Polygon) 
	\ceq 
	\textstyle	
	\frac{1}{2}
	\nnorm{\TurningAngles_\Polygon/\DualEdgeLengths_\Polygon}_{\ell^2_\Polygon}^2
	=
	\frac{1}{2}
	\sum_{\Vertex \in \IVertices(\Triangulation)} 
\frac{\TurningAngles^2_\Polygon(\Vertex)}{\DualEdgeLengths_\Polygon(\Vertex)}.
\end{align*}

We consider the following discrete version of \autoref{prob:SmoothEulerBernoulli}:
\begin{problem}\label{prob:DiscreteEulerBernoullis}
For boundary conditions $\DCond \colon \partial \Interval \to \AmbSpace$ and $\NCond \colon  \partial \Interval \to \Sphere$,
minimize the function $\DiscEulerBernoulli$ on $\DiscFeasible$, i.e., among all $\Polygon \in \DiscConfSpace$ subject to
$\DiscConstraintMap(\Polygon) = 0$.
\end{problem}

\subsection{Discrete Regularity}
\label{sec:DiscreteRegularity}

Mimicking the smooth setting, we show that our \emph{discrete} \apriori assumptions \eqref{eq:DiscretePriors} on minimizers of the discrete elastica problem are indeed valid. As in the smooth case, our proof of the validity of the \apriori assumptions hinges on ``discrete elliptic bootstrapping''. Again, we require bounds on the Lagrange multipliers that arise in our constraint optimization problem. In order to be useful for our purpose, these bounds must only depend on the energy of the minimizer and the geometric constraints. Therefore, we first prove that the differential of the constraint mapping $\DiscConstraintMap$ has a right inverse that is uniformly bounded. The results in this section and the respective proofs are similar to the smooth setting above. 

Define the following set of discrete ``tame immersions'':
\begin{align}
	\DiscTame^{k,p}(\StrainBound,\EnergyBound,\eta)
	\ceq 
	\set{
		\Polygon\in \DiscConfSpace |
			\nnorm{ \LogStrain_\Polygon }_{\sobo[k-1,p][\Polygon]} \leq \StrainBound,
			\,
			\nnorm{\Tangent_\Polygon}_{\sobo[k-1,p][\Polygon]} \leq \EnergyBound,
			\,			
			\Length(\Polygon) \geq (1+\eta)
			\,
			\nabs{\textstyle \Polygon(L) - \Polygon(0) }
	}	\label{eq:DiscreteThetaConditions}
\end{align}	
with $p \in \intervalcc{2,\infty}$, $k \in \N$, $k \geq 2$ and where $\StrainBound$ and $\EnergyBound$ are nonnegative, and $\eta>0$. Let the target space $\DiscTargetSpace$  of the constraint mapping (see \eqref{eq:DefDiscTargetSpace}) be equipped with the usual Euclidean norm for the boundary conditions and with the $\sobo[1,p]$-norm for the logarithmic strain $\LogStrain_\Polygon$. By a slight abuse of notation we simply denote the resulting product space norm as the $\sobo[1,p]$-norm.

\begin{lemma}[Right Inverse of $D\DiscConstraintMap$]\label{lem:DiscreteRightinverseDPhi}
There are constants $\MaxRadius_0> 0$, $C > 0$, and a right inverse $B_{\Polygon}$ of $D\DiscConstraintMap$ such that for all partitions $\Triangulation$ with $\MaxRadius(\Triangulation)\leq \MaxRadius_0$ the mapping
$\Polygon \mapsto B_{\Polygon}$ is 
uniformly bounded 
as a map from $\DiscTame^{2,p}(\StrainBound,\EnergyBound,\eta)$ to the space
of bounded linear mappings $L(\DiscTargetSpace, \Map(\Vertices(\Triangulation); \AmbSpace))$ equipped with the operator norm; i.e., $\nnorm{B_{\Polygon}}_{\sobo[1,p] \to \sobo[2,p]} \leq C$. 
\end{lemma}
\begin{proof}
The proof follows the proof of \autoref{lem:SmoothRightinverseDPhi}. Let $\Discprnor(\Edge)$ denote the projection onto the orthogonal complement of the unit edge vector $\Tangent_\Polygon(\Edge)$.
For a tangent vector $u \in T_\Polygon\DiscConfSpace$, the derivative of $\DiscConstraintMap$ is given by
\begin{align*}
	D\DiscConstraintMap(\Polygon)\,u
	= \textstyle
	\bigparen{
		u(0),
		u(L),
		\DiscDnor u(\shiftr{0}),
		\DiscDnor u(\shiftl{L}),
		\ninnerprod{\EdgeVectors_\Polygon, \DiscD u}
	}.	
\end{align*}
Let
$\Polygon \in \DiscTame^{2,p}(\StrainBound, \EnergyBound,\eta)$ and
$
	(U_0,U_1,V_0,V_1,\lambda) \in T_{\DiscConstraintMap(\Polygon)} \TargetSpace
$
be given. 
Analogously as in \autoref{lem:SmoothRightinverseDPhi}, 
we define the \emph{normalized} distance to $\shiftr{0}$ as
\begin{align*}
	\varphi_\Polygon \colon \Edges(\Triangulation) \to \intervalcc{0,1},
	\quad 
	\varphi_\Polygon(\Edge)
	\ceq 
	\tfrac{
		s_\Polygon(\Edge) - s_\Polygon(\shiftr{0})
	}{
		s_\Polygon(\shiftl{L})-s_\Polygon(\shiftr{0})
	},
\end{align*}
where $s_\Polygon(\Edge) \ceq 
\EdgeMidpoints(\Edge) - \EdgeMidpoints(\shiftr{0})
= \sum_{\Vertex \in \IVertices(\Triangulation), \Vertex< \EdgeMidpoints(\Edge)} \DualEdgeLengths_\Polygon(\Vertex)$.
The function
$
u \colon \Vertices(\Triangulation) \to \AmbSpace
$
defined by
\begin{align}
	u(\Vertex) &= 
		\textstyle
	U_0 + \sum_{\substack{\Edge \in \Edges(\Triangulation)\\\EdgeMidpoints(\Edge) \leq \Vertex}}\;
		\lambda(\Edge) \, \EdgeVectors_\Polygon(\Edge) \, \EdgeLengths_\Polygon(\Edge)
	\label{eq:DiscreteRightInverse}		
	\\
	&\qquad
	\textstyle
	+ \sum_{\substack{\Edge \in \Edges(\Triangulation)\\\EdgeMidpoints(\Edge) \leq \Vertex}}\;
		\Discprnor(\Edge) \, \bigparen{
			(1-\varphi_\Polygon(\Edge))^2 \, V_0
			+ 
			(1-\varphi_\Polygon(\Edge)) \, \varphi_\Polygon(\Edge) \, V
		 	+ 
		 	\varphi_\Polygon(\Edge)^2 \, V_1
		} \,\EdgeLengths_\Polygon (\Edge)
	\notag
\end{align}
constitutes a right inverse $B_{\Polygon}(U_0,U_1,V_0,V_1,\lamda) \ceq u$ 
once we choose $V \in \AmbSpace$ to be the solution of the linear equation
$\varTheta_{\Polygon}  \, V = b_{\Polygon}$,
where
\begin{align}
	\varTheta_{\Polygon} 
	&\textstyle
	\ceq \sum_{\Edge \in \Edges(\Triangulation)}
		(1-\varphi_\Polygon(\Edge)) \, \varphi_\Polygon(\Edge)
			\, \EdgeLengths_\Polygon(\Edge) \, \Discprnor(\Edge)
	\qand
	\label{eq:DefvarThetaPolygon}
	\\
	b_{\Polygon} 
	&\textstyle
	\ceq U_1 - U_0 - \sum_{\Edge \in \Edges(\Triangulation)} 
	\bigparen{
		\lambda(\Edge) \, \EdgeVectors_\Polygon(\Edge)\
		+ 
		\Discprnor(\Edge) \, \bigparen{
			(1-\varphi_\Polygon(\Edge))^2 \, V_0+ \varphi_\Polygon^2(\Edge) \, V_1
		}
	}\,\EdgeLengths_\Polygon(\Edge) .
	\notag
\end{align}
Showing that the matrix $\varTheta_{\Polygon} $ is invertible with uniformly bounded inverse for all $\Polygon \in \DiscTame^{2,p}(\StrainBound, \EnergyBound,\eta)$
is a bit more involved than in the smooth setting, since $\Tangent_\Polygon$ is not continuous.
We defer this detail to \autoref{lem:DiscreteBoundonTheta} in the appendix.
Given this result, the uniform bound for the operator norm $\nnorm{B_{\Polygon}}_{\sobo[1,p]\to\sobo[2,p]}$ 
is obtained by elementary (but lengthy) calculations (which we skip here deliberately).
\end{proof}

We now turn to analyzing the regularity of discrete minimizers. In particular, the following result justifies our discrete \apriori assumptions. As in the smooth setting, our result requires that the distance between the prescribed end points $\DCond=\Curve|_{\partial \Interval}$ of the curve is strictly less than the curve length $L$. This assumption is necessary in order to rule out the situation where the distance between the end points is equal to $L$ but the first order boundary conditions are not those of a straight line segment, in which case the feasible set $\DiscFeasible$ is empty.

\begin{btheorem}[Regularity of Discrete Minimizers]\label{lem:DiscreteRegularity}
Let $L = \abs{\Interval} > 0$, $\eta\ceq L/ \nabs{\DCond(L) - \DCond(0)}-1 >0$, and $\EnergyBound \geq 0$ be given. Then there are constants $\MaxRadius_0 >0$ and $C \geq 0$ such that the following holds true
for each partition $\Triangulation$ with $\MaxRadius(\Triangulation)\leq \MaxRadius_0$: 
Every critical point $\Polygon \in \DiscFeasible$ of 
\autoref{prob:DiscreteEulerBernoullis} with $\DiscEulerBernoulli(\Polygon) \leq \EnergyBound$
satisfies
\begin{align*}
	\seminorm{\Polygon}_{\sobo[2,\infty][\Polygon]} \leq C
	\qand
	\seminorm{\Polygon}_{\tv[3][\Polygon]} \leq C.
\end{align*}
In particular, the minimizers $\DiscMinimizers$ of \autoref{prob:DiscreteEulerBernoullis}
satisfy the \apriori assumptions specified in \eqref{eq:DiscretePriors}, i.e., $\DiscMinimizers \subset \DiscPriors$.
Under the additional assumption of \emph{almost uniform partitions}, i.e. partitions $\Triangulation$ satisfying
$
	\nabs{\log(\VnextE{\ReferenceEdgeLengths}/\VprevE{\ReferenceEdgeLengths})}
	\leq C \, \min ( \VprevE{\ReferenceEdgeLengths},\VnextE{\ReferenceEdgeLengths})
$,
we even have
$
	\seminorm{\Polygon}_{\sobo[3,\infty][\Polygon]} \leq C
$.
\end{btheorem}
\begin{proof}
We closely follow the proof of the smooth setting, which immediately leads to a discrete analogue of the critical point equation
\eqref{eq:KKTCondition}. Testing this equation 
against those infinitesimal displacements $u \colon \Vertices(\Triangulation) \to \AmbSpace$ satisfying $u(0) = 0$, $\DiscDnor u(\shiftr{0})=0$, $\DiscDnor u(\shiftl{L})=0$, and $\ninnerprod{\EdgeVectors_\Polygon(\Edge) , \DiscD u(\Edge)} = 0$ for all $\Edge \in \Edges(\Triangulation)$, we obtain
\begin{align*}
	\textstyle
	\Discprnor \, \Bigparen{
		\frac{\EprevV{\TurningAngles_\Polygon}}{\VprevE{\DualEdgeLengths_\Polygon}\, \sin(\EprevV{\TurningAngles_\Polygon})}  
		\, 
	    \EprevE{\EdgeVectors_\Polygon}
		+
		\frac{\EnextV{\TurningAngles_\Polygon}}{\VnextE{\DualEdgeLengths_\Polygon}\, \sin(\EnextV{\TurningAngles_\Polygon})}  
		\, 
	    \EnextE{\EdgeVectors_\Polygon}
    }
	=
	\EdgeLengths_\Polygon \, \Discprnor \, \mu_{10}.
\end{align*}
Notice that this is the discrete analogue of \eqref{eq:SmoothRegularity1}, i.e., $\prnor \TangentC'' = \prnor \, \mu_{10}$.
Following the discussion on bounds on Lagrange multipliers in the beginning of \autoref{sec:SmoothRegularity},
the norm of $\mu_{10}$ can be uniformly bounded using \autoref{lem:DiscreteRightinverseDPhi} for $p = 2$.

Adding the identity
\begin{align*}
	\textstyle
	\EdgeVectors_\Polygon \, \Biginnerprod{
	\EdgeVectors_\Polygon,
		\tfrac{\EprevV{\TurningAngles_\Polygon}}{\VprevE{\DualEdgeLengths_\Polygon}\, \sin(\EprevV{\TurningAngles_\Polygon})}  
		\, 
	    \EprevE{\EdgeVectors_\Polygon}
		+
		\tfrac{\EnextV{\TurningAngles_\Polygon}}{\VnextE{\DualEdgeLengths_\Polygon} \, \sin(\EnextV{\TurningAngles_\Polygon})}  
		\, 
	    \EnextE{\EdgeVectors_\Polygon}
	}
	-
		\tfrac{\EprevV{\TurningAngles_\Polygon}}{\VprevE{\DualEdgeLengths_\Polygon}\, \sin(\EprevV{\TurningAngles_\Polygon})}
		\cos(\EprevV{\TurningAngles_\Polygon}) \, \EdgeVectors_\Polygon
	-	\tfrac{\EnextV{\TurningAngles_\Polygon}}{\VnextE{\DualEdgeLengths_\Polygon} \, \sin(\EnextV{\TurningAngles_\Polygon})}
		\cos(\EnextV{\TurningAngles_\Polygon}) \, \EdgeVectors_\Polygon
	= 0
\end{align*}
to the preceding equation and dividing by $\EdgeLengths_\Polygon$
leads to 
\begin{align*}
	\textstyle
	\frac{\EprevV{\TurningAngles_\Polygon}}{\sin(\EprevV{\TurningAngles_\Polygon})}  
	\, 
	\frac{\EprevE{\EdgeVectors_\Polygon} - \EdgeVectors_\Polygon \, \cos(\EprevV{\TurningAngles_\Polygon})}{\VprevE{\DualEdgeLengths_\Polygon}\,\EdgeLengths_\Polygon }
	+
	\frac{\EnextV{\TurningAngles_\Polygon}}{\sin(\EnextV{\TurningAngles_\Polygon})}  
	\, 
	\frac{\EnextE{\EdgeVectors_\Polygon} - \EdgeVectors_\Polygon \, \cos(\EnextV{\TurningAngles_\Polygon})}{\VnextE{\DualEdgeLengths_\Polygon}\,\EdgeLengths_\Polygon }
	=
	\Discprnor \, \mu_{10}
	.
\end{align*}
Notice that in the limit $\EprevV{\TurningAngles_\Polygon} \to 0$ and $\EnextV{\TurningAngles_\Polygon} \to 0$, this equation leads to a second order finite difference equation. More precisely, we have the following discrete analogue of \eqref{eq:IndicatrixBootstrappingEquation}:
\begin{align}
	\DiscLaplacian \EdgeVectors_\Polygon
	&=
	\Discprnor \, \mu_{10}
	+
	\tfrac{
		1-\EnextV{\TurningAngles_\Polygon}/\sin(\EnextV{\TurningAngles_\Polygon})
	}{
		\EnextV{\DualEdgeLengths_\Polygon}\,\EdgeLengths_\Polygon 
	}
	\EnextE{\Tangent_\Polygon}
	+
	\tfrac{
		1-\EprevV{\TurningAngles_\Polygon}/\sin(\EprevV{\TurningAngles_\Polygon})
	}{
		\EprevV{\DualEdgeLengths_\Polygon}\,\EdgeLengths_\Polygon 
	}
	\EprevE{\Tangent_\Polygon}
	+
	\Bigparen{
		\tfrac{
		\EnextV{\TurningAngles_\Polygon}/\tan(\EnextV{\TurningAngles_\Polygon}) - 1
	}{
		\EnextV{\DualEdgeLengths_\Polygon}\,\EdgeLengths_\Polygon 	
	}
	+
	\tfrac{
		\EprevV{\TurningAngles_\Polygon}/\tan(\EprevV{\TurningAngles_\Polygon}) - 1
	}{
		\EprevV{\DualEdgeLengths_\Polygon}\,\EdgeLengths_\Polygon 	
	}
	}
	\Tangent_\Polygon
	.
	\label{eq:DiscreteIndicatrixBootstrappingEquation}
\end{align}
Since
$
	\TurningAngles_\Polygon 
	= 	\DualEdgeLengths_\Polygon^{1/2} 
	\bigparen{\TurningAngles_\Polygon^2 / \DualEdgeLengths_\Polygon}^{1/2}
	\leq \DualEdgeLengths_\Polygon^{1/2} \, \sqrt{\DiscEulerBernoulli(\Polygon)},
$
we may assume that $\TurningAngles_\Polygon \leq \sqrt{\EnergyBound} \, \DualEdgeLengths_\Polygon^{1/2} < \frac{\uppi}{2}$.
One readily checks that  one has
\begin{align*}
	\nabs{1- \alpha/\sin(\alpha)} \leq \tfrac{1}{2} \alpha^2
	\qand
	\nabs{\alpha/\tan(\alpha)-1} \leq \tfrac{1}{2} \alpha^2
	\quad
	\text{for $\alpha \in \nintervalcc{0,\tfrac{\uppi}{2}}$.}
\end{align*}
Thus, we obtain
\begin{align}\label{eq:DiscreteRegularity1}
	\nabs{\DiscLaplacian \EdgeVectors_\Polygon}
	\leq
	\nabs{\mu_{10}}
	+
	\tfrac{
		(\EnextV{\TurningAngles_\Polygon})^2
	}{
		\EnextV{\DualEdgeLengths_\Polygon}\,\EdgeLengths_\Polygon 
	}
	+
	\tfrac{
		(\EprevV{\TurningAngles_\Polygon})^2
	}{
		\EprevV{\DualEdgeLengths_\Polygon}\,\EdgeLengths_\Polygon 
	} ,
\end{align}
which provides us with the following inequality:
\begin{align*}
	\textstyle
	\nseminorm{\Polygon}_{\tv[3][\Polygon]}= \nseminorm{\EdgeVectors_\Polygon}_{\tv[2][\Polygon]}
	=
	\sum_{\InteriorEdge \in \IEdges(\Triangulation)} 
		\nabs{(\DiscLaplacian \EdgeVectors_\Polygon)(\InteriorEdge)} \, \EdgeLengths_\Polygon(\InteriorEdge) 
	\leq
	\nabs{\mu_{10}} \, L + 
	\sum_{\InteriorEdge \in \IEdges(\Triangulation)} 
	\Bigparen{
		\tfrac{
		(\EnextV{\TurningAngles_\Polygon})^2
	}{
		\EnextV{\DualEdgeLengths_\Polygon}
	}
	+
	\tfrac{
		(\EprevV{\TurningAngles_\Polygon})^2
	}{
		\EprevV{\DualEdgeLengths_\Polygon}
	}	
	}
	\leq \nabs{\mu_{10}} \, L + 4\, \EnergyBound
	.
\end{align*}
Notice how this corresponds to $\TangentC\in \SoboC[2,1]$ from \autoref{lem:SmoothRegularity}.
In order to find a $\sobo[1,\infty]$-bound for $\EdgeVectors_\Polygon$, let $\Vertex_0 \in \IVertices(\Triangulation)$ be a vertex such that $\frac{\TurningAngles_\Polygon(\Vertex_0)}{\DualEdgeLengths_\Polygon(\Vertex_0)}$ is minimal.
Observe that
\begin{align*}
	\textstyle
	\paren{
		\frac{\TurningAngles_\Polygon(\Vertex_0)}{\DualEdgeLengths_\Polygon(\Vertex_0)} 
	}^2
	\leq \frac{1}{\Length(\Polygon)}\sum_{\Vertices \in \IVertices(\Triangulation)}
			\paren{
		\frac{\TurningAngles_\Polygon(\Vertex_0)}{\DualEdgeLengths_\Polygon(\Vertex_0)} 
	}^2 \, \DualEdgeLengths_\Polygon(\Vertex)
	\leq \frac{2}{\Length(\Polygon)} \DiscEulerBernoulli(\Polygon) \leq \frac{2 \, \EnergyBound}{L}.
\end{align*}
Hence, for each $\Vertex \in \IVertices(\Triangulation)$, we may deduce
\begin{align*}
	\textstyle
	\bigabs{
		\tfrac{\VnextE{\EdgeVectors_\Polygon}(\Vertex) - \VprevE{\EdgeVectors_\Polygon}(\Vertex)}{\DualEdgeLengths_\Polygon (\Vertex)}
	}
	&\textstyle
	\leq
	\bigabs{\tfrac{\VnextE{\EdgeVectors_\Polygon}(\Vertex_0) - \VprevE{\EdgeVectors_\Polygon}(\Vertex_0)}{\DualEdgeLengths_\Polygon(\Vertex_0)}}
	+ \sum_{
	\substack{
		\InteriorEdge\in\IEdges(\Triangulation)\\
		\text{$\InteriorEdge$ between $\Vertex$ and $\Vertex_0$}
	}}
	\nabs{(\DiscLaplacian \EdgeVectors_\Polygon)(\InteriorEdge)}\, \EdgeLengths_\Polygon(\InteriorEdge)
	\\
	&\textstyle
	\leq
	\bigabs{
		\tfrac{\TurningAngles_\Polygon(\Vertex_0)}{\DualEdgeLengths_\Polygon (\Vertex_0)}
	}
	+ 
	\sum_{\InteriorEdge \in\IEdges(\Triangulation)}
	\nabs{(\DiscLaplacian \EdgeVectors_\Polygon)(\InteriorEdge)}\, \EdgeLengths_\Polygon(\InteriorEdge)
	\leq \sqrt{\tfrac{2 \, \EnergyBound}{L}} + \nabs{\mu_{10}} \, L + 4\,\EnergyBound.
\end{align*}
Using the uniform bound on $\abs{\mu_{10}}$ from above, we obtain that there exists a constant $C=C(L, \EnergyBound, \eta)>0$, such that $\seminorm{\EdgeVectors_\Polygon}_{\sobo[1,\infty][\Polygon]} \leq C$. Compare this to \autoref{lem:SmoothRegularity}; there we found that $\TangentC\in \SoboC[1,\infty]$. 

Assuming that $\Triangulation$ is almost uniform, we may perform an additional step of ``elliptic bootstrapping'':
The inequality $\seminorm{\EdgeVectors_\Polygon}_{\sobo[1,\infty][\Polygon]} \leq C$ together with $\TurningAngles_\Polygon\leq\uppi/2$
implies that
$\EprevV{\TurningAngles_\Polygon} \leq C \, \EprevV{\DualEdgeLengths_\Polygon}$ 
and
$\EnextV{\TurningAngles_\Polygon} \leq C \, \EnextV{\DualEdgeLengths_\Polygon}$. 
Because of $\EdgeLengths_\Polygon = \ReferenceEdgeLengths$ and since $\Triangulation$ is almost uniform, we have $\shiftll{\EdgeLengths_\Polygon}$, $\shiftrr{\EdgeLengths_\Polygon} \leq \tilde C \, \EdgeLengths_\Polygon$ (for sufficiently small $\MaxRadius(\Triangulation)$), and we may deduce that
\begin{align*}
	\tfrac{(\shiftr{\TurningAngles_\Polygon})^2}{\shiftr{\DualEdgeLengths_\Polygon}  \EdgeLengths_\Polygon}
	+
	\tfrac{(\shiftl{\TurningAngles_\Polygon})^2}{\shiftl{\DualEdgeLengths_\Polygon} \EdgeLengths_\Polygon}
	\leq 
	C^2 \, \Bigparen{
		\tfrac{(\shiftr{\DualEdgeLengths_\Polygon})^2}{\shiftr{\DualEdgeLengths_\Polygon}  \EdgeLengths_\Polygon}
		+
		\tfrac{(\shiftl{\DualEdgeLengths_\Polygon})^2}{\shiftl{\DualEdgeLengths_\Polygon}  \EdgeLengths_\Polygon}
		}
	=
	\tfrac{1}{2} \,C^2 \, \Bigparen{
		\tfrac{\EdgeLengths_\Polygon + \shiftrr{\EdgeLengths_\Polygon}}{\EdgeLengths_\Polygon}	
		+
		\tfrac{\shiftll{\EdgeLengths_\Polygon} + \EdgeLengths_\Polygon}{\EdgeLengths_\Polygon}		
		}
	\leq
	C^2 \, \nparen{ 1+ \tilde C	}.	
\end{align*}
Substituting these inequalities into \eqref{eq:DiscreteRegularity1}
shows that
\begin{align*}
		\seminorm{\Polygon}_{\sobo[3,\infty][\Polygon]} = \nnorm{\DiscLaplacian \EdgeVectors_\Polygon}_{\ell^\infty} \leq C,
\end{align*}
which is in perfect correspondence with $\Curve \in \SoboC[3,\infty]$ from \autoref{lem:SmoothRegularity}
\end{proof}

%% file: Pictures/Notation_Discrete_pdf.tex
\begingroup%
  \makeatletter%
  \providecommand\color[2][]{%
    \errmessage{(Inkscape) Color is used for the text in Inkscape, but the package 'color.sty' is not loaded}%
    \renewcommand\color[2][]{}%
  }%
  \providecommand\transparent[1]{%
    \errmessage{(Inkscape) Transparency is used (non-zero) for the text in Inkscape, but the package 'transparent.sty' is not loaded}%
    \renewcommand\transparent[1]{}%
  }%
  \providecommand\rotatebox[2]{#2}%
  \ifx\svgwidth\undefined%
    \setlength{\unitlength}{360bp}%
    \ifx\svgscale\undefined%
      \relax%
    \else%
      \setlength{\unitlength}{\unitlength * \real{\svgscale}}%
    \fi%
  \else%
    \setlength{\unitlength}{\svgwidth}%
  \fi%
  \global\let\svgwidth\undefined%
  \global\let\svgscale\undefined%
  \makeatother%
  \begin{picture}(1,0.67878788)%
    \put(0,0){\includegraphics[width=\unitlength]{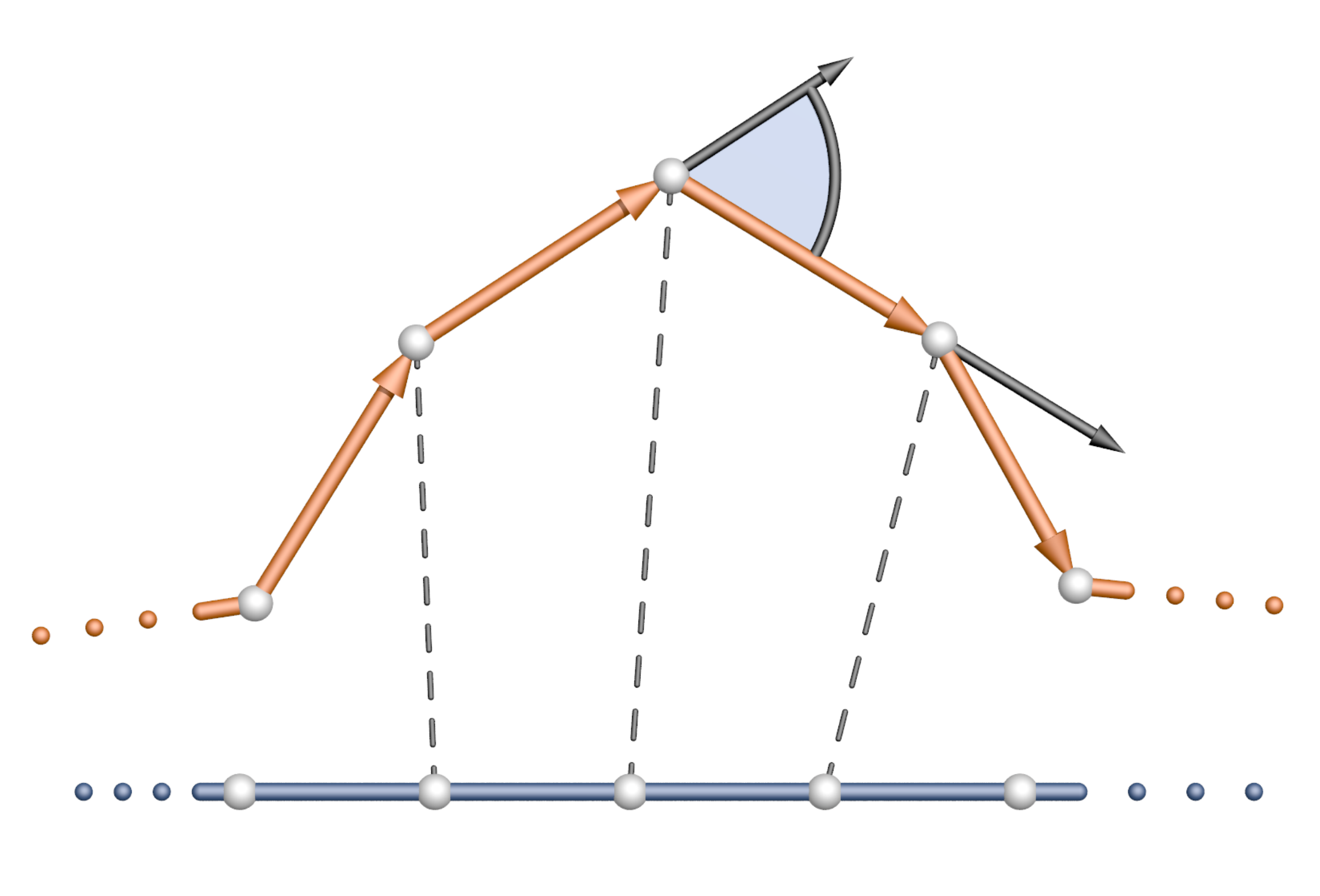}}%
    \put(0.47878788,0.024242424){\makebox(0,0)[cb]{\smash{$\Vertex$}}}%
    \put(0.33939394,0.024242424){\makebox(0,0)[cb]{\smash{$\VprevV{\Vertex}$}}}%
    \put(0.63636364,0.024242424){\makebox(0,0)[cb]{\smash{$\VnextV{\Vertex}$}}}%
    \put(0.41212121,0.11515152){\makebox(0,0)[cb]{\smash{$\VprevE{\Vertex}$}}}%
    \put(0.56363636,0.11515152){\makebox(0,0)[cb]{\smash{$\VnextE{\Vertex}$}}}%
    \put(0.4969697,0.58181818){\makebox(0,0)[cb]{\smash{$\Polygon(\Vertex)$}}}%
    \put(0.24242424,0.40606061){\makebox(0,0)[cb]{\smash{$\VprevV{\Polygon}(\Vertex)$}}}%
    \put(0.65454545,0.38181818){\makebox(0,0)[cb]{\smash{$\VnextV{\Polygon}(\Vertex)$}}}%
    \put(0.69090909,0.63030303){\makebox(0,0)[cb]{\smash{$\VprevE{\EdgeVectors_\Polygon}(\Vertex)$}}}%
    \put(0.8969697,0.32727273){\makebox(0,0)[cb]{\smash{$\VnextE{\EdgeVectors_\Polygon}(\Vertex)$}}}%
    \put(0.58424242,0.53939394){\makebox(0,0)[cb]{\smash{$\TurningAngles_\Polygon(\Vertex)$}}}%
  \end{picture}%
\endgroup%

%% file: Reconstruction.tex

\section{Reconstruction}
\label{sec:Reconstruction}

As outlined in the proof of \autoref{theo:ConvergenceTheorem}, we require a reconstruction operator $\Reconstruction$ that maps feasible polygons to feasible smooth curves with good approximation properties (in particular with respect to the elastic energy). To this end, we first construct an approximate reconstruction operator $\ReconstructionApprox$ that maps feasible polygons to \emph{almost} feasible curves (i.e., smooth curves with small constraint violation). Afterwards, we construct a restoration operator $\Adjust$ that repairs the constraint violation of such almost feasible smooth curves. The final reconstruction operator is defined by the composition $\Reconstruction= \Adjust \circ \ReconstructionApprox$.

\paragraph*{Notation} As before, let $\Interval = \intervalcc{0,L}$ be a compact interval and  $\Triangulation$ a partition of $\Interval$.
Throughout this section, we fix 
boundary conditions $\DCond \colon \partial \Interval \to \AmbSpace$ and $\NCond \colon  \partial \Interval \to \Sphere$. The constraints are encoded into the mapping $\ConstraintMap$ from \eqref{eq:DefTargetSpace}.

\subsection{Approximate Reconstruction Operator}\label{sec:ApproximateReconstruction}

We begin by constructing approximate reconstruction operators on the set $\DiscPriors$ of discrete \apriori information (see \eqref{eq:DiscretePriors}).

\begin{bproposition}[Approximate Reconstruction Operator]\label{prop:ApproximateReconstructionTheorem}
There exist constants $\MaxRadius_0>0$ and $C \geq 0$ such that for each partition $\Triangulation$ of $\Interval$ with $\MaxRadius(\Triangulation) \leq \MaxRadius_0$, there is an \emph{approximate reconstruction operator} $\ReconstructionApprox \colon \DiscPriors \to \ConfSpace$ with the following properties
for each $\Polygon \in \DiscPriors$
and $\Curve = \ReconstructionApprox(\Polygon)$:
\begin{enumerate}
	\item \label{item:ARConsistency} Energy consistency: 
	$\EulerBernoulli(\Curve) \leq  \DiscEulerBernoulli (\Polygon) + C \, \MaxRadius(\Triangulation)$.
	\item \label{item:ARProximity} $\Sobo[1,\infty]$-proximity: 
$\nnorm{\Curve -  \DiscTestMap (\Polygon) }_{\Sobo[1,\infty]} 
\leq C \, \MaxRadius(\Triangulation)$,
	where $\DiscTestMap$ is the piecewise affine interpolation operator from \eqref{eq:piecewiselinear}.
	\item \label{item:ARStrain} 
	Strain consistency: 
	$\LogStrain_\Curve = 0$.
	\item \label{item:ARFeasibility}
	Approximate feasibility: 
	$
	\nnorm{\ConstraintMap(\Curve)}_{\TV[2]}
	\leq 
	C \, \MaxRadius(\Triangulation)$.
	\item \label{item:ARCurvatureConsistency} 
	Curvature consistency:\footnote{Recall that $\DualEdge(\Vertex)$ denotes the dual edge of an interior vertex $\Vertex \in \IVertices(\Triangulation)$ and that $\shiftr{0}$ and $\shiftl{L}$ denote the two boundary edges. See \autoref{sec:DiscreteSetting} for details.}
	\begin{enumerate}
		\item 
		\label{item:ARCurvatureConsistencyInterior} 
		$\sup_{r \in \DualEdge(\Vertex)} \nabs{\Curvature_{\Curve}(r) - \Curvature_\Polygon(\Vertex)}
	\leq C \, \MaxRadius(\Triangulation)$
	\quad for each $\Vertex\in \IVertices(\Triangulation)$.
		\item 
		\label{item:ARCurvatureConsistencyBoundary} 
		$\sup_{r \in \shiftr{0}} \nabs{\Curvature_{\Curve}(r) - \Curvature_\Polygon(\shiftrr{0})}
	\leq C \, \MaxRadius(\Triangulation)
		\qand
		\sup_{r \in \shiftl{L}} \nabs{\Curvature_{\Curve}(r) - \Curvature_\Polygon(\shiftll{L})}
	\leq C \, \MaxRadius(\Triangulation)$.	
	\end{enumerate}
	\item \label{item:ARFTCR} 
	Finite total curvature rate: 
	$\nseminorm{\Curve}_{\TV[3]} = \nseminorm{\Curve}_{\TV[3][\Curve]} \leq C$.
\end{enumerate}
\end{bproposition}
\begin{proof}
For $\Polygon \in \DiscPriors$,
we construct $\Curve = \ReconstructionApprox(\Polygon)$
as a piecewise circular curve with $C^1$-continuity.
The basic idea is to interpolate the discrete indicatrix $\Tangent_\Polygon$ by a piecewise geodesic curve on the sphere in order to obtain an indicatrix $\Tangent_\Curve$ of class $C^0$.
This way, we obtain a curve $\Curve \in \Imm[2,\infty][][\Interval][\AmbSpace] \cap \BV[3][][\Interval][\AmbSpace]$. More concretely, we first define the unit tangents of $\Curve$ at edge midpoints by putting
$
	\Tangent_{\Curve} (\EdgeMidpoints(\Edge)) \ceq \EdgeVectors_\Polygon(\Edge) 
$
for each $\Edge \in \Edges(\Triangulation)$.
We then extent $\Tangent_{\Curve}$ to a continuous, piecewise geodesic curve $\Tangent_{\Curve} \colon \Interval \to \Sphere$ on the unit sphere $\Sphere \subset \AmbSpace$ by requiring that the restrictions
$\Tangent_{\Curve}|_{\shiftr{0}}$,
$\Tangent_{\Curve}|_{\shiftl{L}}$,
and
$\Tangent_{\Curve}|_{\DualEdge{\Vertex}}$ for each interior vertex $\Vertex \in \IVertices(\Triangulation)$ 
are geodesics (see \autoref{fig:ApproximateReconstructionOperator}). Third, we define  $\Curve(t) \ceq q(0) + \int_0^t \Tangent_\Curve(r) \, \dd r$.

It remains to verify the claims of the theorem.
By construction, we have $\Curve(0) = q(0)$ and $\LogStrain_\Curve=0$; in particular, this implies Statement~\ref{item:ARStrain}.
\begin{figure}[t]
\renewcommand{\myincludegraphics}[2]{\begin{tikzpicture}
    \node[inner sep=0pt] (fig) at (0,0) {\includegraphics{#1}};
	\node[below right= 1ex] at (fig.north west) {\footnotesize\textbf{(#2)}};    
\end{tikzpicture}
}
\capstart
    \setkeys{Gin}{%
        trim = 130 240 60 250, 
        clip=true, 
        width=0.4\textwidth
    }
\presetkeys{Gin}{clip}{}
\begin{center}
\myincludegraphics{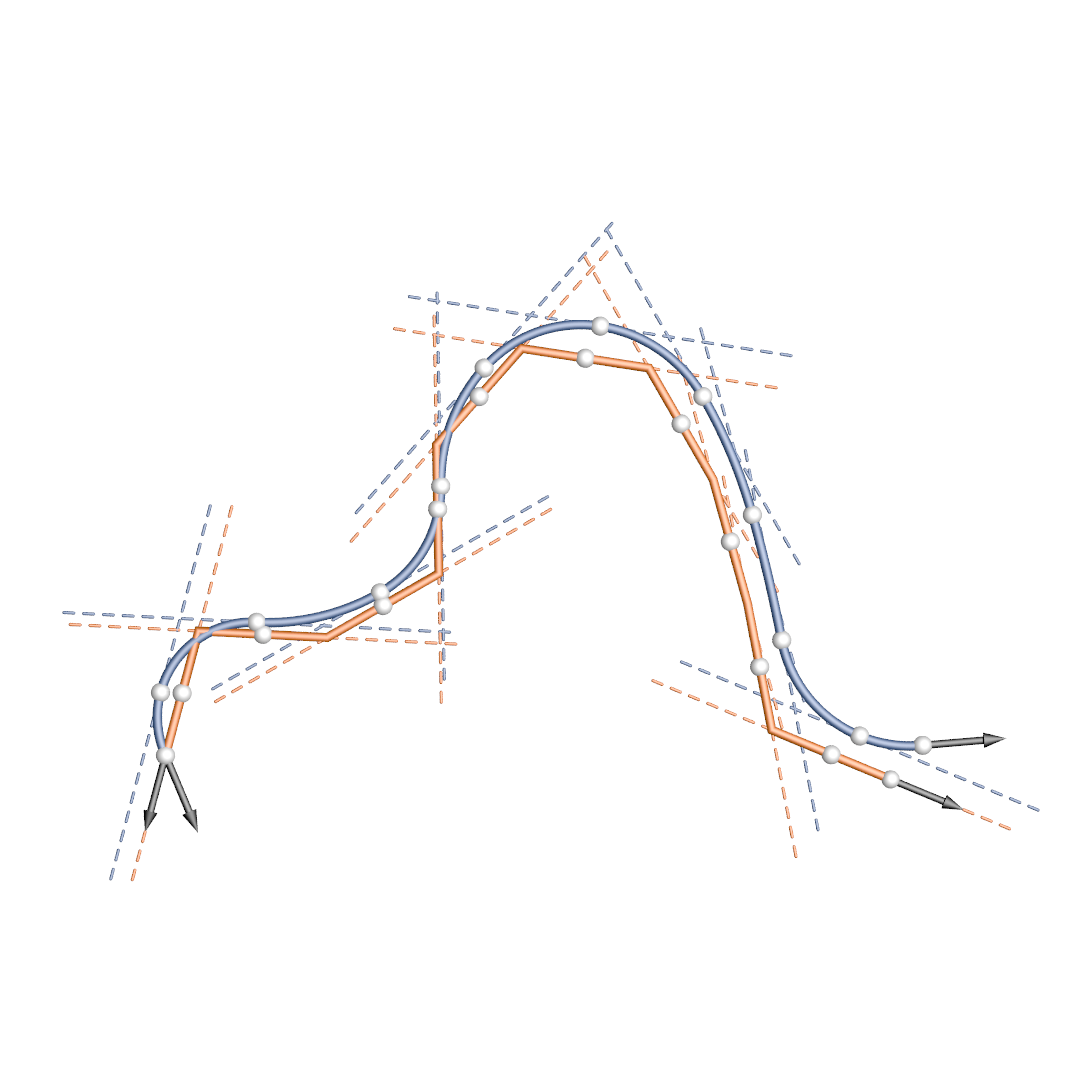}{a}
\myincludegraphics{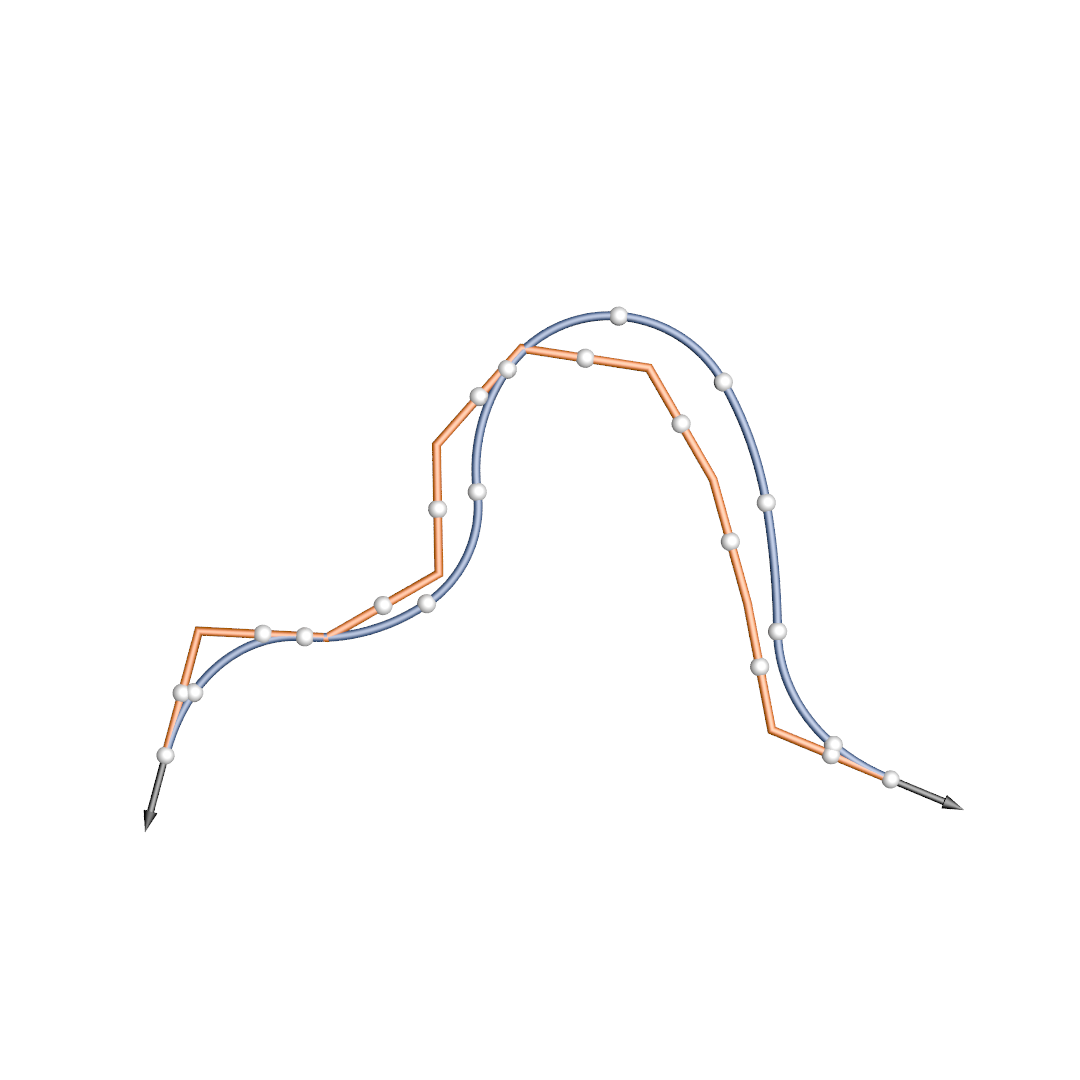}{b}
\end{center}
\caption{
(a) A discrete curve $\Polygon$ (orange) is smoothened by a piecewise circular curve $\ReconstructionApprox(\Polygon)$ (blue) such that tangents of  $\Polygon$ at edge midpoints agree with tangents of $\ReconstructionApprox(\Polygon)$ where circular arcs meet (white points). These curves have the same length, leading to differences between their end points and boundary tangents (gray), which can be controlled.
(b) These differences can be repaired by applying the restoration operator~$\Adjust$, leading to the curve $\Reconstruction(\Polygon) = \Adjust \circ \ReconstructionApprox(\Polygon)$ (blue).
}
\label{fig:ApproximateReconstructionOperator}
\end{figure}
Each circular arc that belongs to an interior dual edge $\DualEdge(\Vertex)$ has constant absolute curvature given by  $\frac{\TurningAngles_\Polygon(\Vertex)}{\DualEdgeLengths_\Polygon(\Vertex)}$. 
So the net bending energy contributed by the arcs of $\Curve$ belonging to interior dual edges is exactly equal to $\DiscEulerBernoulli(\Polygon)$. We have only small excess energy on the dual edges of the two boundary points, i.e.:
\begin{align*}
	\textstyle
	\EulerBernoulli(\Curve)
	=  \DiscEulerBernoulli(\Polygon) 
	+ \frac{1}{4} \paren{\frac{\TurningAngles_\Polygon(\shiftrr{0})}{\DualEdgeLengths_\Polygon(\shiftrr{0})}}^2 \, \EdgeLengths_\Polygon(\shiftr{0})
	+ \frac{1}{4} \paren{\frac{\TurningAngles_\Polygon(\shiftll{L})}{\DualEdgeLengths_\Polygon(\shiftll{L})}}^2 \,\EdgeLengths_\Polygon(\shiftr{L}).
\end{align*}
From
$\Curvature_\Polygon = \frac{1}{\DualEdgeLengths_\Polygon} \bigparen{\shiftr{\Tangent_{\Polygon}} - \shiftl{\Tangent_{\Polygon}}}$,
we deduce that $\nabs{\Curvature_\Polygon} = \tfrac{2 \, \sin(\TurningAngles_\Polygon/2)}{\DualEdgeLengths_\Polygon}$. 
Thus we have
$\frac{\TurningAngles_\Polygon}{\DualEdgeLengths_\Polygon} \leq 2 \frac{2 \sin(\TurningAngles_\Polygon/2)}{\DualEdgeLengths_\Polygon} = 2\, \nabs{\Curvature_\Polygon}  \leq 2 \, \EnergyBound_1$, so that we may deduce Statement \ref{item:ARConsistency}.

Let $\Vertex \in \IVertices(\Triangulation)$ be an interior vertex and let $r \in \DualEdge(\Vertex)^\circ$ be an interior point of its dual edge.
Then we have
$\Dtot \Curvature_\Curve(r) = - \Tangent_\Curve(r) \, \bigparen{\tfrac{\TurningAngles_\Polygon(\Vertex)}{\DualEdgeLengths_\Polygon(\Vertex)}}^2$.
Similar expressions for $\Dtot \Curvature_\Curve(r)$ can be derived for points $r$ in edges $\shiftr{0}$ and $\shiftl{L}$ of the boundary points.
This way, we obtain $\nnorm{\Dtot \Curvature_\Curve}_{L^\infty} \leq \EnergyBound_1^2$.
Since we have $\Tangent_{\DiscTestMap(\Polygon)}(r) = \Tangent_{\Polygon}(\Edge)$ for each edge $\Edge \in \Edges(\Triangulation)$ and $r \in \Edge^\circ$,
this leads to
\begin{align}
	\Angle{
			\Tangent_{\Curve}(r)
		}{
			\Tangent_{\DiscTestMap(\Polygon)}(r)
		}
	=
	\Angle{
			\Tangent_{\Curve}(r)
		}{
			\Tangent_{\Polygon}(\Edge)
		}
	=
	\Angle{
			\Tangent_{\Curve}(r)
		}{
			\Tangent_{\Curve}(\EdgeMidpoints(\Edge))
		}
	\leq \tfrac{1}{2} \nnorm{\Dtot \Tangent_{\Curve}}_{L^\infty} \, \EdgeLengths_\Polygon(\Edge)
	\leq C \, \MaxRadius(\Triangulation).
	\label{eq:ARAngleEstimate}
\end{align}
Because of $\LogStrain_{\DiscTestMap(\Polygon)} = \LogStrain_{\Curve} = 0$, we obtain
$\nseminorm{\Curve - \DiscTestMap(\Polygon)}_{\Sobo[1,\infty]} \leq C \, \MaxRadius(\Triangulation)$. Together with $\Curve(0) = \DiscTestMap(\Polygon)(0)$, this shows
Statement \ref{item:ARProximity}.

As for Statement~\ref{item:ARFeasibility},
we have 
$\nnorm{\Curve|_{\partial \Interval}-\DCond} \leq \nnorm{\Curve - \DiscTestMap(\Polygon)}_{L^\infty} + \nnorm{\DiscTestMap(\Polygon)|_{\partial \Interval}-\DCond}_{L^\infty} \leq C \, \MaxRadius(\Triangulation)$
and
\begin{align*}
	\nabs{\NCond(0) - \Tangent_\Curve(0)} \leq \nabs{\NCond(0) - \Tangent_\Polygon(\shiftr{0})} +  \nabs{\Tangent_\Polygon(\shiftr{0}) - \Tangent_\Curve(0)}
\leq 0 + \tfrac{1}{2} \nnorm{\Dtot \Tangent_\Curve}_{L^\infty} \, \EdgeLengths_\Polygon(\shiftr{0}) \leq C \, \MaxRadius(\Triangulation).
\end{align*}
Analogously, one shows that $\nabs{\NCond(L) - \Tangent_\Curve(L)} \leq C \, \MaxRadius(\Triangulation)$. Since $\LogStrain_\Curve = 0$, this shows Statement~\ref{item:ARFeasibility}.

Let $\Vertex \in \IVertices(\Triangulation)$ be an interior vertex and let $r \in \DualEdge(\Vertex)^\circ$ be an interior point of its dual edge.
Because $\Curvature_\Polygon(\Vertex)$ is contained in the two-dimensional span of $\set{\Curvature_\Curve(r) | r \in \DualEdge(\Vertex)^\circ }$, we obtain $\bigAngle{\Curvature_\Curve(r)}{\Curvature_\Polygon(\Vertex)} 
\leq \TurningAngles_\Polygon(\Vertex)$.
Now the triangle inequality implies
\begin{align*}
	\nabs{\Curvature_\Curve(r) - \Curvature_\Polygon(\Vertex)}
	&\leq 
	\nabs{\nabs{\Curvature_\Curve(r)} - \nabs{\Curvature_\Polygon(\Vertex)}}
	+
	\nabs{\Curvature_\Curve(r)} \, 
	\bigAngle{
		\Curvature_\Curve(r)}{\Curvature_\Polygon(\Vertex)}
	\\
	&\leq \nabs{
		\tfrac{\TurningAngles_\Polygon(\Vertex)}{\DualEdgeLengths_\Polygon(\Vertex)}
		- 
		\tfrac{2 \, \sin(\TurningAngles_\Polygon(\Vertex)/2)}{\DualEdgeLengths_\Polygon(\Vertex)}
	}
	+
	\tfrac{\TurningAngles_\Polygon(\Vertex)}{\DualEdgeLengths_\Polygon(\Vertex)}
	\,
	\TurningAngles_\Polygon(\Vertex)
	\leq 	
	\tfrac{\TurningAngles_\Polygon(\Vertex)}{\DualEdgeLengths_\Polygon(\Vertex)}
	\bigparen{
		\TurningAngles_\Polygon(\Vertex) + \tfrac{1}{24} \, \TurningAngles_\Polygon(\Vertex)^2
	}
	\leq C \, \tfrac{\TurningAngles_\Polygon(\Vertex)^2}{\DualEdgeLengths_\Polygon(\Vertex)}
	\leq C \, \MaxRadius(\Triangulation).
\end{align*}
This proves Statement~\ref{item:ARCurvatureConsistencyInterior}, and
Statement~\ref{item:ARCurvatureConsistencyBoundary} can be shown analogously.
Finally, we derive the following estimate for the jumps of $\Curvature_\Curve$:
\begin{align*}
	\nabs{
		[\![\Curvature_\Curve]\!](\EdgeMidpoints(\Edge))
		-
		(\EnextV{\Curvature_\Polygon}(\Edge) - \EprevV{\Curvature_\Polygon}(\Edge))	
	}
	&=
	\bigabs{
		\lim_{t \searrow \EdgeMidpoints(\Edge)} \Curvature_\Curve(t)
		-
		\lim_{t \nearrow \EdgeMidpoints(\Edge)} \Curvature_\Curve(t)		
		-
		(\EnextV{\Curvature_\Polygon}(\Edge) - \EprevV{\Curvature_\Polygon}(\Edge))
	}
	\leq 
	C \, \Bigparen{
		\tfrac{\EprevV{\TurningAngles_\Polygon}(\Edge)^2}{\EprevV{\DualEdgeLengths_\Polygon}(\Edge)}
		+
		\tfrac{\EnextV{\TurningAngles_\Polygon}(\Edge)^2}{\EnextV{\DualEdgeLengths_\Polygon}(\Edge)}
	}.
\end{align*}
Thus, we obtain Statement~\ref{item:ARFTCR} from
\begin{align*}
	\nseminorm{\Curve}_{\TV[3][\Curve]}
	=
	\nseminorm{\Curvature_\Curve}_{\TV[1][\Curve]}
	&\leq \textstyle
	\sum_{\Vertex\in\IVertices(\Triangulation)} \int_{\DualEdge(\Vertex)} \nabs{ \Dtot \Curvature_\Curve(t)} \, \LineElementC(t)	
	+ \sum_{\Edge\in\Edges(\Triangulation)}
	\nabs{[\![\Curvature_\Curve]\!](\EdgeMidpoints(\Edge))}
	\leq 
	\nseminorm{\Polygon}_{\tv[3][\Polygon]} 
	+ 2 \, C \, \DiscEulerBernoulli(\Polygon).	
\end{align*}
\end{proof}

\subsection{Restoration Operator}\label{sec:RestorationOperator}

Our aim in this section is to prove that for sufficiently tame immersions, small constraint violations can be repaired by perturbations of comparable size.
Recall the definition of $\Tame^{k,p}(\StrainBound, \EnergyBound,\eta)$ from \eqref{eq:DiscreteThetaConditions}.
We now define the  significantly smaller set
\begin{align}
	\RestorationPriors \ceq 
		\set{\Curve \in \Tame^{2,\infty}(\StrainBound,\EnergyBound,\eta)
		| 
			\nseminorm{\LogStrain_\Curve}_{\TV[2][\Curve]}			
			\leq \StrainBound, \, 
			\nseminorm{\Tangent_\Curve}_{\TV[2][\Curve]}			
			\leq \EnergyBound 
	}.
	\label{eq:RestorationPriors}
\end{align}

\begin{bproposition}[Restoration Operator]\label{prop:RestorationOperator}
There exist $C  > 0$,
$\varepsilon > 0$,
and a \emph{restoration operator}
\begin{align*}
	\Adjust \colon \set{ 
		\Curve \in \RestorationPriors
		| 
		\nnorm{\ConstraintMap(\Curve)}_{\TV[2]} < \varepsilon
	} \to \Feasible
\end{align*}
with the following properties:
\begin{enumerate}
	\item Proximity:\label{item:RestorationProximity} 
	$\nnorm{\Adjust(\Curve) - \Curve}_{\TV[3]}
	\leq C \, \nnorm{\ConstraintMap(\Curve)}_{\TV[2]}$.
	\item Energy consistency:\label{item:RestorationConsistency}
	$\nabs{(\EulerBernoulli \circ \Adjust)(\Curve) - \EulerBernoulli(\Curve)} \leq C\, \nnorm{\ConstraintMap(\Curve)}_{\TV[2]}$.	
\end{enumerate}
\end{bproposition}
\begin{proof}
Fix $\Curve \in U  \ceq \set{ 
		\Curve \in \RestorationPriors
		| 
		\nnorm{\ConstraintMap(\Curve)}_{\TV[2]} < \varepsilon
	}$.
We are going to apply the Newton--Kantorovich theorem (see \autoref{lem:KantorovichVarian} in \autoref{sec:NewtonKantotovich}) to a suitable mapping $F$ between Banach spaces to obtain a curve
$\Adjust(\Curve)$ close to $\Curve$ that satisfies $\ConstraintMap(\Adjust(\Curve)) =0$.
Since the Neumann conditions $\NCond$ map into spheres (and thus their differentials are surjective only onto the tangent spaces of these spheres and not onto $\AmbSpace$), we introduce auxiliary variables and define the mapping
\begin{gather*}
	F \colon (\ConfSpace\cap \BVC[3]) \times \R^2 \to Y \ceq \AmbSpace \times \AmbSpace \times \AmbSpace \times \AmbSpace \times \BV[2][][\Interval],
	\\
	F(\Curve,z_0,z_1)
	\ceq \bigparen{
		\Curve(0) \!-\! \DCond(0),
		\Curve(L) \!-\! \DCond(L),
		\ee^{z_0} \, \Tangent_\Curve(0) \!-\! \NCond(0),
		\ee^{z_1} \, \Tangent_\Curve(L) \!-\! \NCond(L),
		\LogStrain_\Curve
	}.	
\end{gather*}
Notice 
that $F(\Curve,z_0,z_1) = 0$ is equivalent to 
$(\ConstraintMap(\Curve) ,z_0,z_1) = (0,0,0)$.
The derivative of $F$ in direction $(u,w_0,w_1) \in X$ is easily computed:
\begin{align*}
	\MoveEqLeft
		DF(\Curve,z_0,z_1)\,(u,w_0,w_1)\\
	&= \bigparen{
		u(0),
		u(L),
		\ee^{z_0} \,\Dnor u(0) + \ee^{z_0} \, w_0 \, \Tangent_\Curve (0),
		\ee^{z_1} \,\Dnor u(L) + \ee^{z_1} \, w_1 \, \Tangent_\Curve (L),
		\ninnerprod{\TangentC,\Dtot u}
	}.
\end{align*}
Thus, a bounded 
right inverse $R_{(\Curve,z_0,z_1)}$ of $DF(\Curve,z_0,z_1)$ can be readily constructed from the right inverse $B_{\Curve}$ of $D\ConstraintMap(\Curve)$ from \autoref{lem:SmoothRightinverseDPhi}. 
Indeed, uniform boundedness of $\nnorm{B_{\Curve}}_{\TV[2] \to \TV[3]}$ is checked straight-forwardly by applying product rules for functions in spaces $\Sobo[k,\infty]$ and
$\BV[k]$.

By \autoref{lem:SmoothDDPhiBound} below, $DF$ is Lipschitz continuous
on each set $\tilde W$ of the form $\tilde W \ceq \tilde\Priors(\tilde\StrainBound,\tilde\EnergyBound,\tilde\eta) \times \nintervaloo{-R,R}^2$ 
with 
$\StrainBound \leq \tilde \StrainBound < \infty$ 
$\EnergyBound \leq \tilde \EnergyBound < \infty$,
$0 < \tilde \eta \leq \eta$,
and
$R>0$.
Put 
$\tilde U \ceq U \times \nintervaloo{-R/2,R/2}^2$.
\autoref{lem:NormEquivalences} guarantees 
that we can pick an $r>0$ such that 
the \hbox{$\nnorm{\cdot}_{\TV[3]}$-ball}
$\ClosedBall{\tilde U}{2\,r}$ is still contained in $\tilde W$.
By further decreasing $r$ and $\varepsilon$ if necessary,
we can fulfill the conditions of \autoref{lem:KantorovichVarian}. The lemma provides us with a triple $(\Adjust(\Curve),z_0,z_1)$ such that $F(\Adjust(\Curve),z_0,z_1)=0$.
The considerations from above imply that $z_0 = z_1 =0$ and that $\ConstraintMap(\Adjust(\Curve)) = 0$.
Moreover, \autoref{lem:KantorovichVarian} provides us with the estimate
\begin{align*}
	\nnorm{\Adjust(\Curve) - \Curve}_{\TV[3]}
	= \nnorm{(\Adjust(\Curve),0,0) - (\Curve,0,0)}
	\leq C \, \nnorm{F(\Curve,0,0)} 
	= C \, \nnorm{\ConstraintMap(\Curve)}_{\TV[2]}.
\end{align*}
Hence, feasibility $\Adjust(\Curve) \in \Feasible$ and
Claim~\ref{item:RestorationProximity}
are already established. We are left to show consistency of the bending energy;
this follows from the fact that $\EulerBernoulli$ is Lipschitz-continuous on $\Tame^{2,\infty}(\tilde\StrainBound,\tilde\EnergyBound,\tilde\eta)$ (see \autoref{lem:EulerBernoulliisLipschitz} in \autoref{sec:EulerBernoulliisLipschitz}).
\end{proof}

\newpage
\subsection{Reconstruction Operator}
Combining approximate reconstruction with restoration yields our final reconstruction operator  $\Reconstruction= \Adjust \circ \ReconstructionApprox$. Its properties, summarized below, follow immediately from \autoref{prop:ApproximateReconstructionTheorem} and \autoref{prop:RestorationOperator}.
We would like to point out that \autoref{prop:ApproximateReconstructionTheorem}, Statement~\ref{item:ARFeasibility} guarantees that $\ReconstructionApprox(\DiscPriors) \subset \dom(\Adjust)$ for each partition $\Triangulation$ with sufficiently small $\MaxRadius(\Triangulation)$.

\begin{btheorem}[Reconstruction Operator]\label{theo:ReconstructionTheorem}
There are constants $\MaxRadius_0>0$ and $C \geq 0$ such that for each partition $\Triangulation$ of $\Interval$ with $\MaxRadius(\Triangulation) \leq \MaxRadius_0$, there is a reconstruction operator $\Reconstruction \colon \DiscPriors \to \Feasible$ 
with the following properties:
\begin{enumerate}
\item Energy consistency: \label{item:RConsistency}
$\nabs{(\EulerBernoulli \circ \Reconstruction)(\Polygon) - \DiscEulerBernoulli (\Polygon)}  \leq C\, \MaxRadius(\Triangulation)$.
\item $\Sobo[1,\infty]$-proximity: \label{item:RProximity}
$\nnorm{ \Reconstruction(\Polygon) - \DiscTestMap (\Polygon) }_{\Sobo[1,\infty]} \leq C\, \MaxRadius(\Triangulation)$.
	\item \label{item:RCurvatureConsistency} 
	Curvature consistency:
	\begin{enumerate}
		\item 
		\label{item:RCurvatureConsistencyInterior} 
		$\sup_{r \in \DualEdge(\Vertex)} \nabs{\Curvature_{\Reconstruction(\Polygon)}(r) - \Curvature_\Polygon(\Vertex)}
	\leq C \, \MaxRadius(\Triangulation)$
	\quad for each $\Vertex\in \IVertices(\Triangulation)$.
		\item 
		\label{item:RCurvatureConsistencyBoundary} 
		$\sup_{r \in \shiftr{0}} \nabs{\Curvature_{\Reconstruction(\Polygon)}(r) - \Curvature_\Polygon(\shiftrr{0})}
	\leq C \, \MaxRadius(\Triangulation)
		\qand
		\sup_{r \in \shiftl{L}} \nabs{\Curvature_{\Reconstruction(\Polygon)}(r) - \Curvature_\Polygon(\shiftll{L})}
	\leq C \, \MaxRadius(\Triangulation)$.	
	\end{enumerate}
	\item \label{item:RFTCR} 
	Finite total curvature rate: 
	$\nseminorm{\Reconstruction(\Polygon)}_{\TV[3]} \leq  C$.	
\end{enumerate}
\end{btheorem}

%% file: Sampling.tex

\section{Sampling}\label{sec:Sampling}

We closely follow the outline of the previous section. This time, we construct a sampling operator $\Sampling$ that maps smooth, feasible curves to feasible polygons with good approximation properties (in particular with respect to the elastic energy). To this end, we first construct an approximate sampling operator $\SamplingApprox$ that maps smooth, feasible curves to \emph{almost} feasible polygons (i.e., polygons with small constraint violation, see \autoref{fig:SamplingOperator}~(b)). Afterwards, we construct a discrete restoration operator $\DiscAdjust$ that repairs the constraint violation of these almost feasible polygons. The final sampling operator is defined by the composition $\Sampling= \DiscAdjust\circ \SamplingApprox$ (see \autoref{fig:SamplingOperator}~(c)).

\paragraph*{Notation} As before, let $\Interval = \intervalcc{0,L}$ be a compact interval and  $\Triangulation$ a partition of $\Interval$.
Throughout this section, we fix 
boundary conditions $\DCond \colon \partial \Interval \to \AmbSpace$ and $\NCond \colon  \partial \Interval \to \Sphere$. Recall also the definitions of 
the constraint map $\DiscConstraintMap$ from~\eqref{eq:DefDiscTargetSpace}
and of the set 
$\Priors$ from~\eqref{eq:SmoothPriors}.

\begin{figure}[t]
\capstart
    \setkeys{Gin}{%
        trim = 48 115 5 150, 
        clip=true, 
        width=0.32\textwidth
    }
\presetkeys{Gin}{clip}{}
\begin{center}
\myincludegraphics{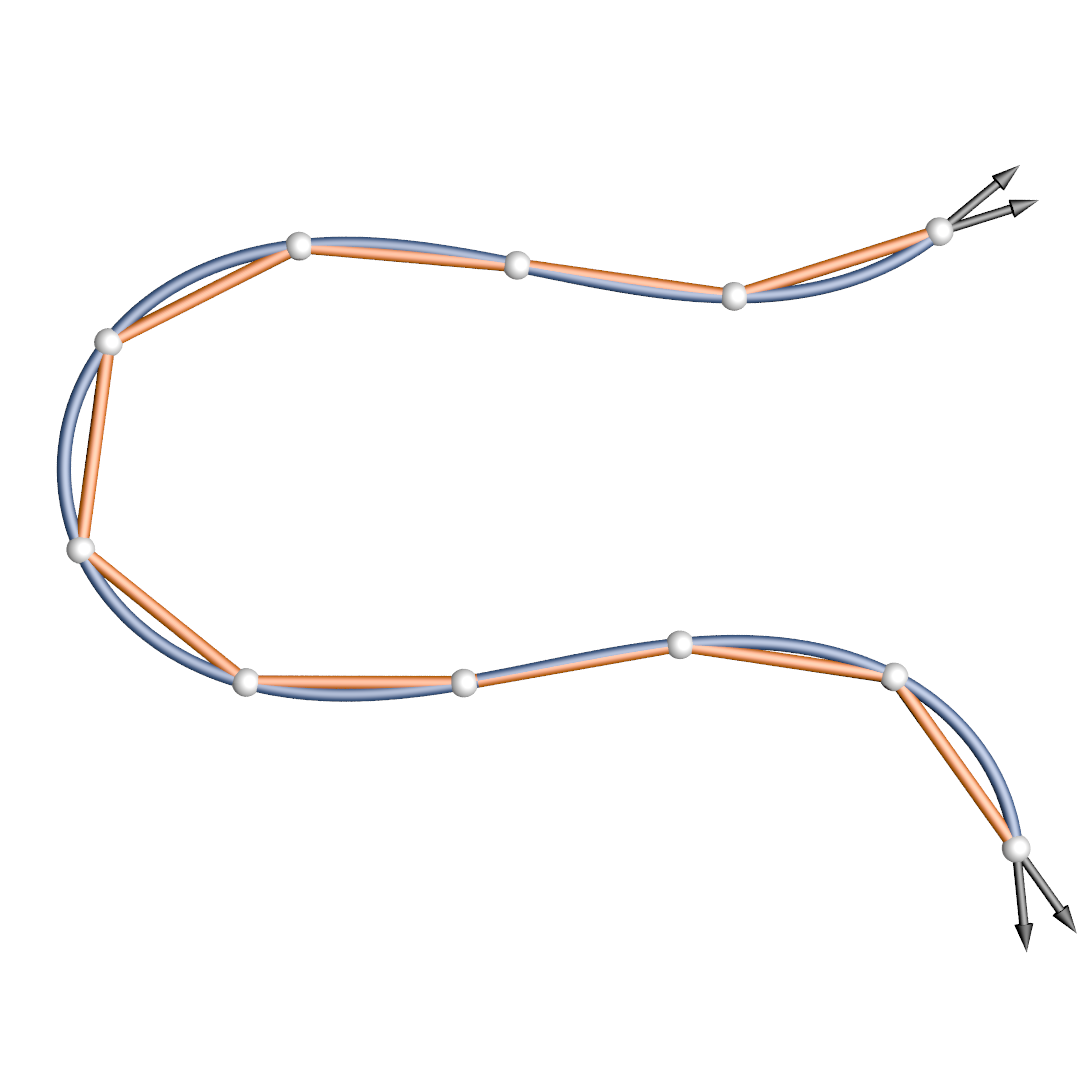}{a}
\myincludegraphics{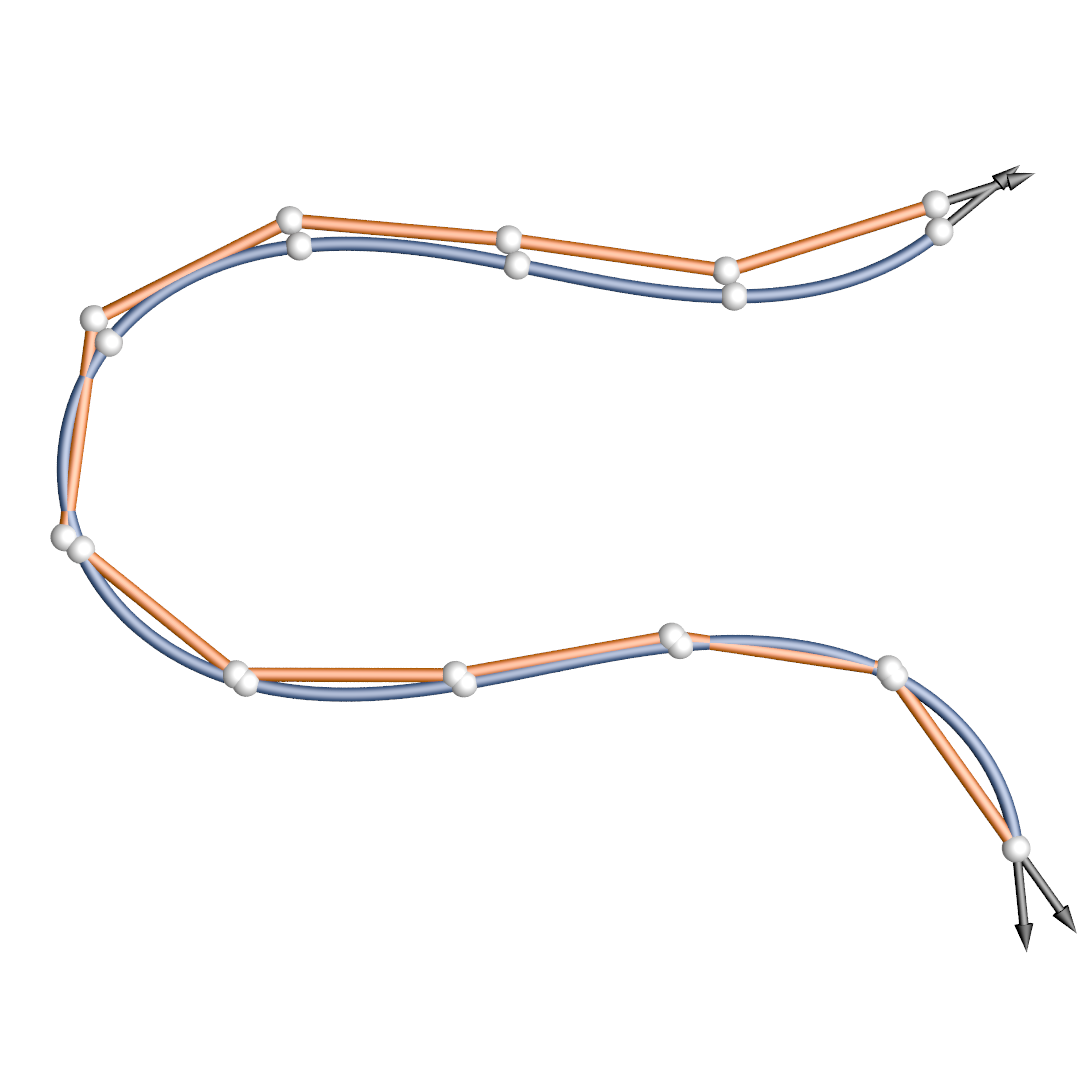}{b}
\myincludegraphics{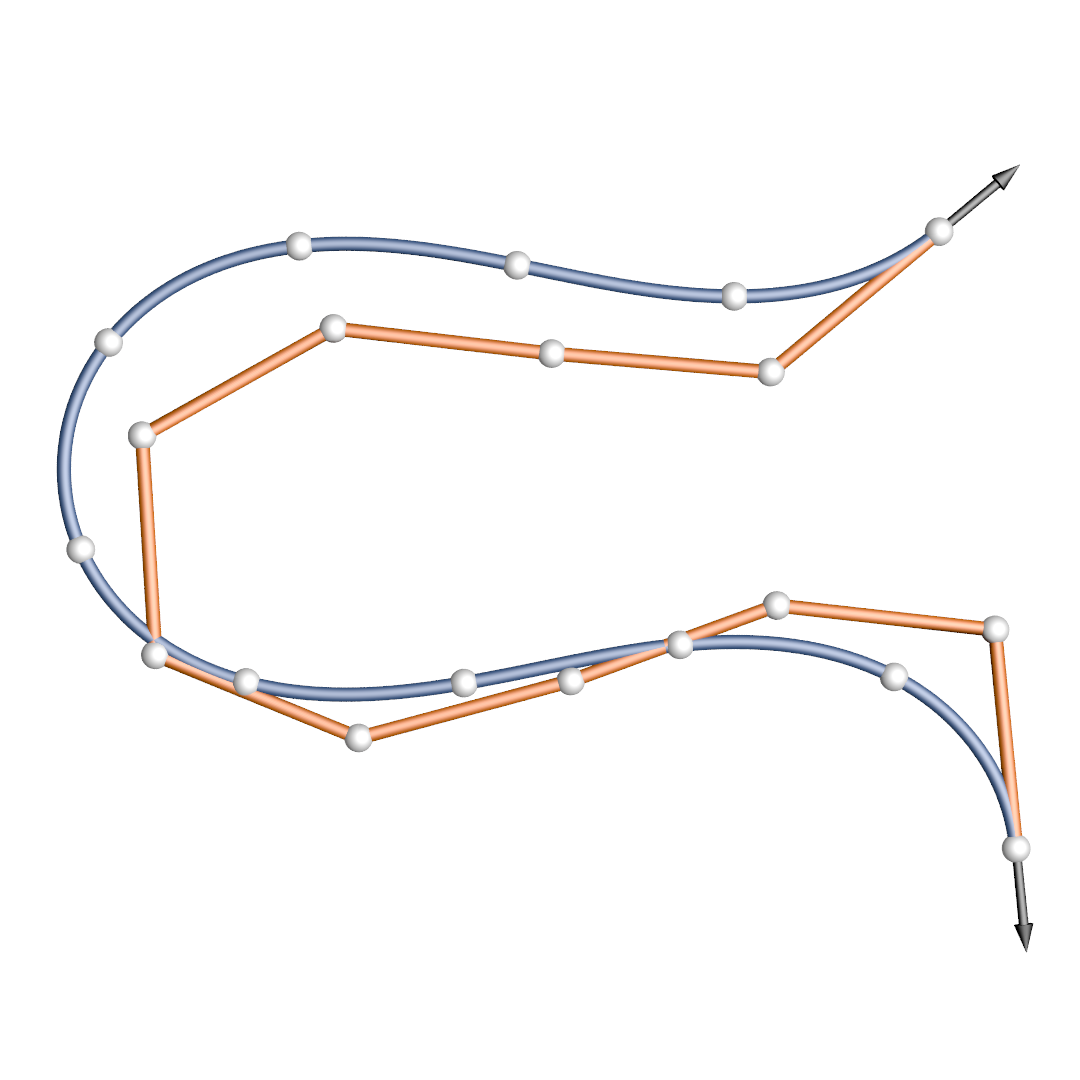}{c}
\end{center}
\caption{
Smooth curve $\Curve$ (blue) together with (a) sampling points (white) and  inscribed polygon $\OtherPolygon$,
(b) polygon $\Polygon = \SamplingApprox(\Curve)$, obtained from $\OtherPolygon$ by stretching the edges to their desired length, and
(c) the final polygon $\Sampling(\Curve) = \DiscAdjust(\Polygon)$. 
}
\label{fig:SamplingOperator}
\end{figure}

\subsection{Approximate Sampling Operator}\label{sec:ApproximateSampling}

\begin{bproposition}[Approximate Sampling Operator]\label{prop:ApproximateSamplingTheorem}
There are constants $\MaxRadius_0>0$ and $C \geq 0$ such that for each partition $\Triangulation$ of $\Interval$ with $\MaxRadius(\Triangulation) \leq \MaxRadius_0$, 
there is an \emph{approximate sampling operator} $\SamplingApprox \colon \Priors \to \DiscConfSpace$ with the following properties
for each $\Curve \in \Priors$ and $\Polygon = \SamplingApprox(\Curve)$:
\begin{enumerate}
\item \label{item:ASConsistency} Energy consistency: 
	$\nabs{\DiscEulerBernoulli(\Polygon) - \EulerBernoulli(\Curve)} \leq C\, \MaxRadius(\Triangulation)$.	
\item \label{item:ASProximity} $\Sobo[1,\infty]$-proximity: 
	$\nnorm{ \DiscTestMap (\Polygon) - \Curve }_{\Sobo[1,\infty]} \leq C\, \MaxRadius(\Triangulation)$, where $\DiscTestMap$ is the piecewise affine interpolation operator from \eqref{eq:piecewiselinear}.
\item \label{item:ASStrain} 
	Strain consistency: $\LogStrain_\Polygon = 0$.
	\item \label{item:ASFeasibility} 
	Approximate feasibility: 
	$\nnorm{\DiscConstraintMap(\Polygon)}_{\tv[2]} \leq 
	 C \, \MaxRadius(\Triangulation)$.	
	\item \label{item:ASCurvatureConsistency} 
	Curvature consistency:\footnote{Recall that $\DualEdge(\Vertex)$ denotes the dual edge of an interior vertex $\Vertex \in \IVertices(\Triangulation)$ and that $\shiftr{0}$ and $\shiftl{L}$ denote the two boundary edges. See \autoref{sec:DiscreteSetting} for details.}
	\begin{enumerate}
	\item \label{item:ASCurvatureConsistencyInterior} 
	$\sup_{r \in \DualEdge(\Vertex)} \nabs{\Curvature_\Polygon(\Vertex) - \Curvature_{\Curve}(r)}
	\leq C \, \MaxRadius(\Triangulation)$
	\quad for each $\Vertex\in \IVertices(\Triangulation)$.
		\item 
		\label{item:ASCurvatureConsistencyBoundary} 
		$\sup_{r \in \shiftr{0}} \nabs{\Curvature_\Polygon(\shiftrr{0}) - \Curvature_{\Curve}(r)}
	\leq C \, \MaxRadius(\Triangulation)
		\qand
		\sup_{r \in \shiftl{L}} \nabs{\Curvature_\Polygon(\shiftll{L}) - \Curvature_{\Curve}(r)}
	\leq C \, \MaxRadius(\Triangulation)$.	
	\end{enumerate}	 
\item \label{item:ASCompactness} Finite total curvature rate: 
	$\nseminorm{\Polygon}_{\tv[3]} = \nseminorm{\Polygon}_{\tv[3][\Polygon]} \leq C$.	
\end{enumerate}
\end{bproposition}
\begin{proof}
Fix $\Curve \in \Priors$. As an intermediate step towards $\SamplingApprox$, we define the polygon $\OtherPolygon$ by pointwise sampling, i.e., by
$\OtherPolygon(\Vertex) \ceq \Curve(\Vertex)$ for each vertex $\Vertex \in \Vertices(\Triangulation)$ (see \autoref{fig:SamplingOperator} (a)).
Observe that the logarithmic strain $\LogStrain_{\OtherPolygon}$ will  not vanish in general since the length of a secant inscribed into $\Curve$ is shorter than the length of the respective arc of $\Curve$.
\autoref{lem:LengthDistortion} in \autoref{sec:SamplingEstimates} below shows that $\nnorm{\LogStrain_{\OtherPolygon}}_{\ell^\infty}$ is of order $\MaxRadius(\Triangulation)^2$.
However, $\nnorm{\LogStrain_{\OtherPolygon}}_{\tv[2]}$ is of order $\MaxRadius(\Triangulation)^0$, so one cannot expect that the constraint violation $\nnorm{\DiscConstraintMap(\OtherPolygon)}_{\tv[2]}$ is bounded by $C \, \MaxRadius(\Triangulation)$. Notice that the latter will become crucial when we attempt to restore feasibility by a
\emph{small perturbation in the norm $\nnorm{\cdot}_{\tv[3]}$}.
We therefore modify $\OtherPolygon$ in order to obtain the desired polygon $\Polygon$.
But for the moment, we observe that $\OtherPolygon$ satisfies
\begin{align*}
	\nnorm{\OtherPolygon |_{\partial \Interval}-\DCond} = 0,
	\:
	\nnorm{\Tangent_{\OtherPolygon}(\shiftr{0}) - \NCond(0) },
	\:
	\nnorm{\Tangent_{\OtherPolygon}(\shiftl{L}) - \NCond(L) }
	\leq C\, \MaxRadius(\Triangulation),	
	\: \text{and} \:
	\nnorm{ \DiscTestMap (\OtherPolygon) - \Curve }_{\Sobo[1,\infty]} \leq C\, \MaxRadius(\Triangulation).
\end{align*}
Moreover, \autoref{lem:SecondDerivativeEstimate} in \autoref{sec:SamplingEstimates} applied to the functions $f = \Curve$ and $F =\OtherPolygon$ shows
\begin{align}
	\nabs{\Curvature_{\OtherPolygon}(\Vertex) -  \Curvature_\Curve(\Vertex)}
	=
	\nabs{\DiscD^2 \OtherPolygon(\Vertex) -  \Dtot^2 \Curve(\Vertex)}
	\leq C\, \DualReferenceEdgeLengths(\Vertex).
	\label{eq:QCurvatureConsistency}
\end{align}

We now construct a further polygon $\SamplingApprox(\Curve) = \Polygon$ by stretching each edge vector of $\OtherPolygon$ such that $\EdgeLengths_\Polygon = \ReferenceEdgeLengths$ (see \autoref{fig:SamplingOperator} (b)). 
More precisely, we define $\Polygon$ recursively by
\begin{align*}
	\Polygon(0) = \OtherPolygon(0),
	\qand
	\Polygon(\shiftrr{\Vertex}) 
	\ceq 
	\Polygon(\Vertex) + \tfrac{\ReferenceEdgeLengths(\shiftr{\Vertex})}{\EdgeLengths_\OtherPolygon(\shiftr{\Vertex})} \nparen{\OtherPolygon(\shiftrr{\Vertex}) - \OtherPolygon(\Vertex)}.
\end{align*}
By construction, we have $\LogStrain_\Polygon = 0$, hence Statement~\ref{item:ASStrain}. Since $\Curve$ is parameterized by arc length (since $\LogStrain_\Curve =0$), the length of a secant inscribed into $\Curve$ differs from the length of the respective arc of $\Curve$ by $C\, \MaxRadius(\Triangulation)^3$; more precisely we have 
\begin{align}\label{eq:RelocationError}
	\nabs{\ReferenceEdgeLengths(\shiftr{\Vertex})-{\EdgeLengths_\OtherPolygon(\shiftr{\Vertex})}}	
	 	\leq C \, \ReferenceEdgeLengths(\shiftr{\Vertex})^3 ,
\end{align}
which implies that
\begin{align*}
	\nnorm{\tfrac{\ReferenceEdgeLengths(\shiftr{\Vertex})}{\EdgeLengths_\OtherPolygon(\shiftr{\Vertex})} \nparen{\OtherPolygon(\shiftrr{\Vertex}) - \OtherPolygon(\Vertex)}
	-
	\nparen{\OtherPolygon(\shiftrr{\Vertex}) - \OtherPolygon(\Vertex)}
	}
	=
	\nabs{\tfrac{\ReferenceEdgeLengths(\shiftr{\Vertex})}{\EdgeLengths_\OtherPolygon(\shiftr{\Vertex})} -1}
	\,
	\nnorm{\OtherPolygon(\shiftrr{\Vertex}) - \OtherPolygon(\Vertex)}
 	\leq C \, \ReferenceEdgeLengths(\shiftr{\Vertex})^3 .
\end{align*}
Therefore, since the number of vertices in the partition is of order $1/\MaxRadius(\Triangulation)$, each point $\OtherPolygon(\Vertex)$ is relocated to $\Polygon(\Vertex)$ by a shift of magnitude $C\, \MaxRadius(\Triangulation)^2$.
This implies $\nnorm{\DiscTestMap(\Polygon) - \Curve}_{L^\infty} \leq C \, \MaxRadius(\Triangulation)^2$ and
$\nnorm{\Polygon|_{\partial \Interval}-\DCond} \leq C \, \MaxRadius(\Triangulation)^2$.
For an edge $\Edge \in \Edges(\Triangulation)$ and a point $r \in \Edge$, we have
$
	\DiscTestMap(\Polygon)(r)' = \tfrac{\ReferenceEdgeLengths(\shiftr{\Vertex})}{\EdgeLengths_\OtherPolygon(\shiftr{\Vertex})} \, \DiscTestMap(\OtherPolygon)(r)'
$, 
which yields 
\begin{align*}
	\nnorm{\DiscTestMap(\Polygon)' - \Curve'}_{L^\infty}
	\leq
	\nnorm{\DiscTestMap(\Polygon)' - \DiscTestMap(\OtherPolygon)'}_{L^\infty}
	+
	\nnorm{\DiscTestMap(\OtherPolygon)' - \Curve'}_{L^\infty}
	\leq C \, \MaxRadius(\Triangulation) ,
\end{align*}
and thus Statement~\ref{item:ASProximity}.
Moreover, we have $\Tangent_\Polygon = \Tangent_\OtherPolygon$,
thus $\nnorm{\Tangent_\Polygon|_{\partial \Interval} - \NCond} \leq C \, \MaxRadius(\Triangulation)$.
Along with $\LogStrain_\Polygon=0$, we obtain Statement~\ref{item:ASFeasibility}.

A further consequence of $\Tangent_\Polygon = \Tangent_\OtherPolygon$ is the identity
$
	\Curvature_\Polygon = \frac{\DualEdgeLengths_\OtherPolygon}{\DualReferenceEdgeLengths} \Curvature_\OtherPolygon
$.
From \eqref{eq:RelocationError} we obtain 
\begin{align*}
	\abs{
		1 - \tfrac{\DualEdgeLengths_\OtherPolygon(\Vertex)}{\DualReferenceEdgeLengths(\Vertex)}
	}
	&= 
	\abs{
	\tfrac{
	\ReferenceEdgeLengths(\shiftr{\Vertex}) - \EdgeLengths_\OtherPolygon(\shiftr{\Vertex})
	}{
		\ReferenceEdgeLengths(\shiftr{\Vertex})
		+
		\ReferenceEdgeLengths(\shiftl{\Vertex})
	}
	+
	\tfrac{
	\ReferenceEdgeLengths(\shiftl{\Vertex}) - \EdgeLengths_\OtherPolygon(\shiftl{\Vertex})
	}{
		\ReferenceEdgeLengths(\shiftr{\Vertex})
		+
		\ReferenceEdgeLengths(\shiftl{\Vertex})
	}
	}	
	\leq
\tfrac{
	C \, \ReferenceEdgeLengths(\shiftr{\Vertex})^3
	}{
		\ReferenceEdgeLengths(\shiftr{\Vertex})
		+
		\ReferenceEdgeLengths(\shiftl{\Vertex})
	}
	+
	\tfrac{
	C \, \ReferenceEdgeLengths(\shiftl{\Vertex})^3
	}{
		\ReferenceEdgeLengths(\shiftr{\Vertex})
		+
		\ReferenceEdgeLengths(\shiftl{\Vertex})
	}	
	\leq C \, \DualReferenceEdgeLengths(\Vertex)^2,
\end{align*}
which together with \eqref{eq:QCurvatureConsistency} leads to
\begin{align}
	\nabs{\Curvature_\Polygon(\Vertex) - \Curvature_\Curve(\Vertex)} \leq C \, \DualReferenceEdgeLengths(\Vertex).
	\label{eq:PCurvatureConsistency}
\end{align}
Now, Lipschitz continuity of $\Curvature_\Curve$ implies Statement~\ref{item:ASCurvatureConsistencyInterior}; Statement~\ref{item:ASCurvatureConsistencyBoundary} can be shown analogously.

Lipschitz continuity of $\nabs{\Curvature_\Curve}^2$ implies that the integral
$\EulerBernoulli(\Curve) = \tfrac{1}{2} \int_\Interval \nabs{\Curvature_\Curve}^2 \, \LineElementC$ can be approximated up to an error of order $\MaxRadius(\Triangulation)$ by piecewise constant interpolation of $\nabs{\Curvature_\Curve}^2$ on dual edges, even if we neglect the dual edges of the boundary points.
Along with \eqref{eq:PCurvatureConsistency}, we obtain
$
	\nabs{
		\textstyle
		\EulerBernoulli(\Curve) 
		- 
		\frac{1}{2}\sum_{\Vertex \in \IVertices(\Triangulation)} \nabs{\Curvature_\Polygon(\Vertex)}^2 \,
		\DualEdgeLengths_\Polygon(\Vertex)
	}
	\leq C \, \MaxRadius(\Triangulation).
$
In this sum, the square of the discrete curvature $\nabs{\Curvature_\Polygon(\Vertex)}^2$ appears where we would like to have squared (rescaled) turning angles $(\TurningAngles_\Polygon(\Vertex)/\DualEdgeLengths_\Polygon(\Vertex))^2$.
But we have $\TurningAngles_\Polygon = \TurningAngles_\OtherPolygon$ and since $\OtherPolygon$ is inscribed into $\Curve$, we have $\TurningAngles_\Polygon = \TurningAngles_\OtherPolygon \leq C \, \EdgeLengths_\OtherPolygon \leq \tilde C \, \EdgeLengths_\Polygon$.
For small angles, we may estimate
\begin{align*}
	\nabs{
		\nabs{ \Curvature_\Polygon}^2 - (\tfrac{\TurningAngles_\Polygon}{\DualEdgeLengths_\Polygon})^2
	}
	&=
	\nabs{
		\nabs{ \DiscD \Tangent_\Polygon}^2 - (\tfrac{\TurningAngles_\Polygon}{\DualEdgeLengths_\Polygon})^2
	}
	= \DualEdgeLengths_\Polygon^{-2} \, \nabs{
		\nabs{ \shiftr{\Tangent_\Polygon}- \shiftl{\Tangent_\Polygon} }^2 - \TurningAngles_\Polygon^2
	}
	\\
	&= \DualEdgeLengths_\Polygon^{-2} \, \nabs{
		(2 \sin(\TurningAngles_\Polygon/2))^2
		- 
		\TurningAngles_\Polygon^2
	}
	\leq C\, \DualEdgeLengths_\Polygon^{-2} \, \TurningAngles_\Polygon^4
	\leq C\, \MaxRadius(\Triangulation)^2,
\end{align*}
which allows us to deduce Statement~\ref{item:ASConsistency}.
Finally,
Lipschitz continuity of $\Curvature_\Curve$ and \eqref{eq:PCurvatureConsistency} imply that
\begin{align*}
	\seminorm{\Polygon}_{\tv[3][\Polygon]}
	&= \textstyle
	\sum_{\Edge \in \IEdges(\Triangulation)}
	\nabs{\DiscD \Curvature_\Polygon(\Edge)} \, \EdgeLengths_\Polygon(\Edge)
	= \sum_{\Edge \in \IEdges(\Triangulation)}
	\nabs{\Curvature_\Polygon(\shiftr{\Edge}) - \Curvature_\Polygon(\shiftl{\Edge})}
	\\
	&\leq  \textstyle
	\sum_{\Edge \in \IEdges(\Triangulation)}
	\nabs{\Curvature_\Curve(\shiftr{\Edge}) - \Curvature_\Curve(\shiftl{\Edge})}
	+
	\sum_{\Edge \in \IEdges(\Triangulation)}
	\bigparen{
		\nabs{\Curvature_\Polygon(\shiftr{\Edge}) -\Curvature_\Curve(\shiftr{\Edge})}
		+
		\nabs{\Curvature_\Polygon(\shiftl{\Edge}) -\Curvature_\Curve(\shiftl{\Edge})}		
	}
	\\
	&\leq  \textstyle
	\sum_{\Edge \in \IEdges(\Triangulation)}
	C \, \ReferenceEdgeLengths(\Edge)
	+
	\sum_{\Edge \in \IEdges(\Triangulation)}
	\bigparen{
		C\, \DualReferenceEdgeLengths(\shiftr{\Edge})
		+
		C\, \DualReferenceEdgeLengths(\shiftl{\Edge})
	}	
	\leq 3\, C\, L ,
\end{align*}
which shows Statement~\ref{item:ASCompactness}.
\end{proof}

\subsection{Discrete Restoration Operator}\label{sec:DiscreteRestorationOperator}

In this section, we prove a discrete version of \autoref{prop:RestorationOperator}, stating that
small constraint violations of sufficiently tame polygons can be repaired by perturbations of comparable size. In particular, we have to show that the bending energy is increased only insignificantly.
With $\varSigma = \intervalcc{0,L}$, we rewrite the constraint mapping $\DiscConstraintMap$ as
\begin{gather*}
	\DiscConstraintMap \colon \DiscConfSpace \to 
	\DiscTargetSpace \ceq 
	\AmbSpace \times \AmbSpace \times \Sphere \times \Sphere \times \Map(\Edges(\Triangulation);\R),
	\\
	\DiscConstraintMap(\Polygon)
	= \bigparen{
		\Polygon(0) - \DCond(0),
		\Polygon(L) - \DCond(L),
		\Tangent_\Polygon(\shiftr{0}) - \NCond(0),
		\Tangent_\Polygon(\shiftl{L}) - \DCond(L),
		\LogStrain_\Polygon
	}
.
\end{gather*}

Recall the definition of $\DiscTame^\infty(\StrainBound, \EnergyBound,\eta)$ from \eqref{eq:DiscreteThetaConditions}.
We now define the set
\begin{align}
	\DiscRestorationPriors \ceq 
		\set{\Polygon \in \DiscTame^\infty(\StrainBound,\EnergyBound,\eta)
		| 
			\nseminorm{\LogStrain_\Polygon}_{\tv[2][\Polygon]}			
			\leq \StrainBound, \, 
			\seminorm{\Polygon}_{\tv[3][\Polygon]}			
			\leq \EnergyBound 
	}.
	\label{eq:DiscRestorationPriors}
\end{align}

\begin{bproposition}[Discrete Restoration Operator]\label{prop:DiscreteRestorationOperator}
There exist $C > 0$, $\varepsilon > 0$, 
and a \emph{restoration operator}
\begin{align*}
		\DiscAdjust \colon 
		\set{ 
		\Polygon \in \DiscRestorationPriors | \nnorm{\DiscConstraintMap(\Polygon)}_{\tv[2]} \leq \varepsilon
	} \to \DiscFeasible
\end{align*}
with the following properties:
\begin{enumerate}
	\item Proximity:\label{item:DiscreteRestorationProximity} 
	$\nnorm{\DiscAdjust(\Polygon)- \Polygon}_{\tv[3]}
	\leq C \, \nnorm{\DiscConstraintMap(\Polygon)}_{\tv[2]}$.
	\item Energy consistency:\label{item:DiscreteRestorationConsistency} 
	$\nabs{(\DiscEulerBernoulli \circ \DiscAdjust)(\Polygon) - \DiscEulerBernoulli(\Polygon)} \leq C \, \nnorm{\DiscConstraintMap(\Polygon)}_{\tv[2]}$.
\end{enumerate}
\end{bproposition}
\begin{proof}
We augment the system $\DiscConstraintMap(\Polygon) = 0$ to a system $F_\SubscriptTriangulation(\Polygon,z_0,z_1) = 0$ precisely as in \autoref{prop:RestorationOperator}.
By successive application of the discrete product rule
\begin{align*}
	\cD_\Polygon (\varphi \, \psi)
	&=
	(\cD_\Polygon \varphi) \, \shiftl{\psi} + \shiftr{\varphi} \, (\cD_\Polygon \psi)
	=
	(\cD_\Polygon \varphi) \, \shiftr{\psi} + \shiftl{\varphi} \, (\cD_\Polygon \psi),	
\end{align*}
one can show that $D \DiscConstraintMap$ and its right inverse $B_\Polygon$ constructed in 
\autoref{lem:DiscreteRightinverseDPhi}
satisfy
\begin{gather*}
	\nnorm{D \DiscConstraintMap(\Polygon) \, u}_{\tv[2]} \leq C\, \nnorm{u}_{\tv[3]},
	\quad
	\nnorm{D^2 \DiscConstraintMap(\Polygon) \, (u ,v)}_{\tv[2]} \leq C\, \nnorm{u}_{\tv[3]}\,\nnorm{v}_{\tv[3]},
	\qand
	\nnorm{B_\Polygon \, w}_{\bv[3]} \leq C \, \nnorm{w}_{\bv[2]}
\end{gather*}
for all $\Polygon \in \DiscRestorationPriors$ (compare also to \autoref{lem:SmoothDDPhiBound}).
This
guarantees that the Newton--Kantorovich theorem (see \autoref{lem:KantorovichVarian} in \autoref{sec:NewtonKantotovich}) can be applied to the starting value $(\Polygon,0,0)$
with
$\Polygon \in \set{ 
		\Polygon \in \DiscRestorationPriors | \nnorm{\DiscConstraintMap(\Polygon)}_{\tv[2]} \leq \varepsilon
	}$ 
to obtain a polygon $\DiscAdjust(\Polygon)$ that satisfies $F_\Triangulation(\DiscAdjust(\Polygon),0,0) = (0,0,0)$ and hence $\DiscConstraintMap(\DiscAdjust(\Polygon)) = 0$.
Also due to \autoref{lem:KantorovichVarian}, $\DiscAdjust(\Polygon)$ satisfies
$
	\nnorm{\DiscAdjust(\Polygon)-\Polygon}_{\tv[3]}
	\leq 
	C \, \nnorm{\DiscConstraintMap(\Polygon)}_{\tv[2]}
	.
$
This shows feasibility $\Adjust(\Polygon) \in \DiscFeasible$ and Statement~\ref{item:DiscreteRestorationProximity}.
Finally, Statement~\ref{item:DiscreteRestorationConsistency} follows from Lipschitz continuity of $\DiscEulerBernoulli$ on $\DiscTame^\infty(\StrainBound,\EnergyBound,\eta)$, a fact that can be shown in a very similar way as in the smooth case (compare to \autoref{lem:EulerBernoulliisLipschitz} in \autoref{sec:EulerBernoulliisLipschitz}).
\end{proof}

\subsection{Sampling Operator}

Combining approximate sampling with discrete restoration yields our final sampling operator  $\Sampling= \DiscAdjust\circ \SamplingApprox$. Its properties, summarized below, follow immediately from \autoref{prop:ApproximateSamplingTheorem} and \autoref{prop:DiscreteRestorationOperator} after realizing that the norm $\nnorm{\cdot}_{\tv[3]}$ dominates the norm $\nnorm{\cdot}_{\sobo[2,\infty]}$.

\begin{btheorem}[Sampling Operator]\label{theo:SamplingTheorem}
There are constants $\MaxRadius_0>0$ and $C \geq 0$ such that for each partition $\Triangulation$ of $\Interval$ with $\MaxRadius(\Triangulation) \leq \MaxRadius_0$, 
there exists a sampling operator $\Sampling \colon \Priors \to \DiscFeasible$ 
with the following properties:
\begin{enumerate}
\item Energy consistency: \label{item:SConsistency}
$\nabs{(\DiscEulerBernoulli \circ \Sampling)(\Curve) - \EulerBernoulli (\Curve)} \leq C\, \MaxRadius(\Triangulation) $.
\item $\Sobo[1,\infty]$-Proximity: \label{item:SProximity}
$\nnorm{ (\DiscTestMap \circ \Sampling)(\Curve) - \Curve }_{\Sobo[1,\infty]} \leq C\, \MaxRadius(\Triangulation)$.
	\item 
	Curvature consistency: \label{item:SCurvatureConsistency} 
	\begin{enumerate}
	\item \label{item:SCurvatureConsistencyInterior} 
	$\sup_{r \in \DualEdge(\Vertex)} \nabs{\Curvature_{\Sampling(\Curve)}(\Vertex) - \Curvature_{\Curve}(r)}
	\leq C \, \MaxRadius(\Triangulation)$
	\quad for each $\Vertex\in \IVertices(\Triangulation)$.
	\item \label{item:SCurvatureConsistencyBoundary} 
	$\sup_{r \in \shiftr{0}} \nabs{\Curvature_{\Sampling(\Curve)}(\shiftrr{0}) - \Curvature_{\Curve}(r)}
	\leq C \, \MaxRadius(\Triangulation)
		\qand
		\sup_{r \in \shiftl{L}} \nabs{\Curvature_{\Sampling(\Curve)}(\shiftll{L}) - \Curvature_{\Curve}(r)}
	\leq C \, \MaxRadius(\Triangulation)$.	
	\end{enumerate}	 
\item $\Sampling(\Priors) \subset \DiscPriors$ for appropriately chosen constant $\EnergyBound_2 \geq 0$ in the definition \eqref{eq:DiscretePriors} of $\DiscPriors$.
\end{enumerate}
\end{btheorem}

Combining the curvature consistency statements in \autoref{theo:ReconstructionTheorem} and \autoref{theo:SamplingTheorem} leads to the following final consequence:

\begin{theorem}[$W^{2,\infty}$-Proximity]\label{cor:Loop}
There exists a $\MaxRadius_0 >0$ and a $C \geq 0$ such that for each partition~$\Triangulation$ with $\MaxRadius(\Triangulation) \leq \MaxRadius_0$ and for each curve $\Curve \in \Priors$,
we have
\begin{align*}
	\nnorm{
		\Curve
		-
		(\Reconstruction \circ \Sampling)(\Curve)
	}_{\Sobo[2,\infty]} \leq C \, \MaxRadius(\Triangulation).
\end{align*}
\end{theorem}

%% file: NormEstimates.tex
\section{Norm Equivalences}\label{sec:NormEquivalences}

\begin{proof}[of \autoref{lem:NormEquivalences}]
For simplicity, put $\sigma \ceq \nnorm{\LogStrain_\Curve}_{L^\infty}$
and suppose that $u \colon \Interval \to \AmbSpace$ is a sufficiently smooth function. 
Let us start with the estimates for the Sobolev norms.
We discuss only the case $1 \leq p < \infty$, but the case $p = \infty$ can be shown analogously.

\textbf{Case $k=0.$} We have by definition that $\LineElementC = \nabs{\Curve'} \, \dd t$, hence
$\dd t  \leq \ee^{\sigma} \, \LineElementC$ and $\LineElementC  \leq \ee^\sigma \, \dd t$, showing that $\nnorm{u}_{L^p} \leq \ee^{\sigma/p} \, \nnorm{u}_{L^p_\Curve}$ and $\nnorm{u}_{L^p_\Curve} \leq \ee^{\sigma/p} \, \nnorm{u}_{L^p}$ hold true.

\textbf{Case $k=1.$}
By the very definitions of $\Dtot$ and $\LineElementC$, we have
$\nabs{\Dtot \,u}^p \, \LineElementC =\ee^{(1-p) \, \LogStrain_\Curve} \, \nabs{u'}^p \, \dd t$
and
$\nabs{u'}^p \, \dd t  = \ee^{(p-1) \, \LogStrain_\Curve} \, \nabs{\Dtot \,u}^p \, \LineElementC$.
Thus, $\ee^{-\sigma} \leq \nabs{\Curve'} \leq \ee^{\sigma}$ implies
\begin{align*}
	\nnorm{\Dtot \,u}_{L^p_\Curve} \leq \ee^{(1-1/p)\sigma} \nnorm{u'}_{L^p} 
	\qand 
	\nnorm{u'}_{L^p} \leq \ee^{(1-1/p)\sigma} \nnorm{\Dtot \,u}_{L^p_\Curve}.
\end{align*}

\textbf{Case $k=2.$}
Because of $	\Dtot \, u =  \ee^{-\LogStrain_\Curve} \, u'$ and $\LineElementC = \ee^{\LogStrain_\Curve} \, \dd t$, 
the product rule implies
\begin{align*}
	\Dtot \Dtot \, u = \ee^{-2 \, \LogStrain_\Curve} \, \bigparen{u'' - {\LogStrain_\Curve}' \, u'}
	\qand
	u''  = \ee^{2 \, \LogStrain_\Curve} \bigparen{ \Dtot \Dtot \, u + \Dtot \LogStrain_\Curve \, \Dtot u}.
\end{align*}
Thus, we obtain
\begin{align*}
	\nnorm{\Dtot \Dtot \, u}_{L^p_\Curve}
	&\leq \ee^{(2 - 1/p)\, \sigma} \, \bigparen{ \nnorm{u''}_{L^p} + \nnorm{{\LogStrain_\Curve}'}_{L^p} \, \nnorm{u'}_{L^\infty}},
	\qand
	\\
	\nnorm{u''}_{L^p}
	&\leq \ee^{(2- 1/p) \, \sigma} \bigparen{ 
		\nnorm{\Dtot \Dtot u}_{L^p_\Curve} 
		+ 
		\nnorm{\Dtot \LogStrain_\Curve}_{L^p_\Curve} \, \nnorm{\Dtot u}_{L^\infty_\Curve}
	}.
\end{align*}
The analogous statements for general $k >2$ can be obtained by applying the product rule successively.
Finally, the analogous statements for the total variation norms can be derived in a similar way from the product rule $\nnorm{\varphi \, \psi}_{\TV[1]} \leq C \, \nnorm{\varphi}_{\TV[1]}\, \nnorm{\psi}_{\TV[1]}$ and from the Lipschitz-continuity of the exponential function on the compact set $\intervalcc{-\StrainBound,\StrainBound}$.
\end{proof}

%% file: EulerBernoulliLipschitz.tex
\section{Local Lipschitz Continuity of $\EulerBernoulli$}\label{sec:EulerBernoulliisLipschitz}

We need the Lipschitz continuity of the Euler-Bernoulli energy on controlled sets at various places, e.g., in \autoref{sec:SmoothRegularity}, \autoref{prop:RestorationOperator} in \autoref{sec:RestorationOperator}, and, in a discretized version, also in \autoref{prop:DiscreteRestorationOperator} in \autoref{sec:DiscreteRestorationOperator}.

\begin{lemma}\label{lem:EulerBernoulliisLipschitz}
For $p \geq 2$, the bending energy is (uniformly) Lipschitz continuous on $\Tame^{2,p}(\StrainBound,\EnergyBound,\eta)$ with respect to $\nnorm{\cdot}_{\Sobo[2,p]}$.\footnote{See \eqref{eq:ThetaConditions} for the definition of $\Tame^{2,p}$.}
\end{lemma}
\begin{proof}
Let $\Curve_1$, $\Curve_2 \in \Tame^{2,p}(\StrainBound,\EnergyBound,\eta)$. We start with
\begin{align*}
	2 \, \nabs{\EulerBernoulli(\Curve_1) - \EulerBernoulli(\Curve_2)}
	&\leq \textstyle
	\int_\Interval 
		\bigparen{
			\nparen{\nabs{\Curvature_{\Curve_1}}+\nabs{\Curvature_{\Curve_2}}}
			\,
			\nparen{\nabs{\Curvature_{\Curve_1}}-\nabs{\Curvature_{\Curve_2}}}
		}
		\,
		\LineElement_{\Curve_1}
	+
	\ee^{\StrainBound}
	\,
	\nnorm{\nabs{\Curve_1'} - \nabs{\Curve_2'}}_{L^\infty}
	\int_\Interval 
		\nabs{\Curvature_{\Curve_2}}^2
		\,
		\LineElement_{\Curve_2}
	\\
	&\leq
	\bigparen{
		\sqrt{\EulerBernoulli(\Curve_1)}
		+
		\ee^{\StrainBound}
		\sqrt{\EulerBernoulli(\Curve_2)}
	}
	\,
	\nnorm{\Curvature_{\Curve_1}-\Curvature_{\Curve_2}}_{L^2_{\Curve_1}}	
	+
	\ee^{\StrainBound}
	\,
	\EulerBernoulli(\Curve_2)
	\,
	\nnorm{
		\Curve_2' - \Curve_1'
	}_{L^\infty}	.
\end{align*}
From
$\Curvature_{\Curve} = \tfrac{1}{\nabs{\Curve'}^2} \, \prnor \, \Curve''$,
we deduce that
\begin{align*}
	\nnorm{\Curvature_{\Curve_1} - \Curvature_{\Curve_2}}_{L^p_{\Curve_1}}
	&\leq
	\bignorm{
		\tfrac{1}{\nabs{\Curve_1'}^2} 
		-
		\tfrac{1}{\nabs{\Curve_2'}^2} 
	}_{L^\infty}
	\nnorm{
		\Curve_1''
	}_{L^p_{\Curve_1}}
	+
	\nnorm{\tfrac{1}{\nabs{\Curve_2'}^2}}_{L^\infty}
	\,
	\nnorm{ 
		\pr_{\Curve_1}^\perp
		-
		\pr_{\Curve_2}^\perp
	}_{L^\infty} \, 
	\nnorm{
		\Curve_1''
	}_{L^p_{\Curve_1}}
	+
	\nnorm{\tfrac{1}{\nabs{\Curve_2'}^2}}_{L^\infty}
	\nnorm{
		\Curve_1'' - \Curve_2''
	}_{L^p_{\Curve_1}}.
\end{align*}
Moreover, we have
\begin{align*}
	\bigabs{
		\tfrac{1}{\nabs{\Curve_1'}^2} 
		-
		\tfrac{1}{\nabs{\Curve_2'}^2} 
	}
	&\leq 
	\bigparen{\tfrac{1}{\nabs{\Curve_1'}^3}+ \tfrac{1}{\nabs{\Curve_2'}^3}}
	\,
	\bigabs{\nabs{\Curve_1'}-\nabs{\Curve_2'}}
	\leq 2 \, \ee^{3\StrainBound}
	\,
	\nnorm{\Curve_1' - \Curve_2'}_{L^\infty}
	\qand
	\\
	\nabs{\pr_{\Curve_1}^\perp - \pr_{\Curve_2}^\perp} 
	&\leq 2 \, \bigabs{\tfrac{\Curve_1'}{\nabs{\Curve_1'}}-\tfrac{\Curve_1'}{\nabs{\Curve_1'}}}
	\leq 2 \, \bigabs{\tfrac{1}{\nabs{\Curve_1'}}-\tfrac{1}{\nabs{\Curve_2'}}} \, \nabs{\Curve_1'}
	+
	2 \, \tfrac{1}{\nabs{\Curve_2'}}
	\, \nabs{\Curve_1'-\Curve_2'}
	\\
	& \leq 2 \, 
	\bigparen{\tfrac{1}{\nabs{\Curve_1'}^2} + \tfrac{1}{\nabs{\Curve_1'}^2}}
	\,
	\nabs{\Curve_1'-\Curve_2'}
	+
	2 \, \tfrac{1}{\nabs{\Curve_2'}}
	\, \nabs{\Curve_1'-\Curve_2'}
	\leq 6 \, \ee^{\StrainBound}\,\nnorm{\Curve_1'-\Curve_2'}_{L^\infty}.
\end{align*}
An upper bound for the Lipschitz constant can now be derived
from \autoref{lem:NormEquivalences}, from the fact that the $L^p$-norm controls the $L^2$-norm, and from the Sobolev embedding $\Sobo[1,p] \hookrightarrow L^\infty$.
\end{proof}

%% file: ThetaLemma.tex
\section{Uniform Invertibility of $\varTheta_\Polygon$}

This is an auxiliary result required in the proof of \autoref{lem:DiscreteRightinverseDPhi} in \autoref{sec:DiscreteRegularity}; it guarantees that we have control over the right inverse of the differential of the discrete constraint map $\DiscConstraintMap$.

\begin{lemma}\label{lem:DiscreteBoundonTheta}
Let $p \in \intervalcc{2,\infty}$.
There are $\MaxRadius_0 > 0$ and $C > 0$ such that the matrix $\varTheta_{\Polygon}$ from \eqref{eq:DefvarThetaPolygon} is invertible and satisfies $\nnorm{\varTheta_{\Polygon}^{-1}} \leq C$ for 
each partition $\Triangulation$ with $\MaxRadius(\Triangulation)<\MaxRadius_0$ and 
each $\Polygon \in \DiscTame^p(\StrainBound,\EnergyBound,\eta)$.
\end{lemma}
\begin{proof}
Let $\Polygon \in \DiscTame^p(\StrainBound,\EnergyBound,\eta)$ and put $\beta \ceq \arccos(\tfrac{2}{2+\eta})< \frac{\uppi}{2}$.
For $p<\infty$, we may bound
$\TurningAngles_\Polygon$ as follows:
\begin{align*}
	\tfrac{2}{\uppi} \, \TurningAngles_\Polygon(\Vertex) 
	\leq 2 \sin(\TurningAngles_\Polygon(\Vertex)/2)
	= \DualEdgeLengths_\Polygon(\Vertex) \, \nabs{\Curvature_\Polygon(\Vertex)}	
	= \DualEdgeLengths_\Polygon^{1-1/p}(\Vertex) 
	\, \bigparen{
		\nabs{\Curvature_\Polygon(\Vertex)}^p \, \DualEdgeLengths_\Polygon
	}^{\frac{1}{p}}
	\leq \ee^{(1-1/p) \, \StrainBound} \, \nseminorm{\Polygon}_{\sobo[2,p][\Polygon]}\, \MaxRadius^{1-1/p}(\Triangulation).
\end{align*}
For $p = \infty$, $\TurningAngles_\Polygon(\Vertex) \leq C \, \nseminorm{\Polygon}_{\sobo[2,\infty][\Polygon]}\, \MaxRadius(\Triangulation)$ is also evident.

Hence, we may choose $\MaxRadius_0>0$ so small that $\TurningAngles_\Polygon(\Vertex) < \uppi - 2\,\beta$ and $\TurningAngles_\Polygon(\Vertex) < \frac{\uppi}{2}$ hold for all $\Vertex \in \IVertices(\Triangulation)$.
For $\varepsilon \in \nintervaloo{0,\tfrac{1}{2}}$, 
we define the sets
\begin{align*}
	\textstyle
	\Edges_\varepsilon
	\ceq
	\set{
		\Edge \in \Edges(\Triangulation) | 
		\forall t \in \Edge\colon \varepsilon \leq \varphi_\Polygon(t) \leq (1-\varepsilon)
	}
	\qand
	\Interval_\varepsilon \ceq \bigcup_{\Edge \in \Edges_\varepsilon} \Edge.
\end{align*}
Denote the restriction of $\Polygon$ onto $\Interval_\varepsilon \cap \Vertices(\Triangulation)$ by $\Polygon_\varepsilon$.
With $\lambda \ceq \Length(\Polygon)^{-1} \sup_{\Edge \in \Edges(\Triangulation)} \EdgeLengths_\Polygon(\Edge) \leq \Length(\Polygon)^{-1} \, \ee^\StrainBound \, \MaxRadius_0$, we have the inequalities
\begin{align*}
	\textstyle
	\Length(\Polygon_\varepsilon) \geq \Length(\Polygon) \, (1 - 2 \, \varepsilon - 3\, \lambda)
	\qand
	\nabs{\int_{\partial \Interval_\varepsilon} \!\Polygon_\varepsilon}
	\leq \nabs{\int_{\partial \Interval} \Polygon}
	+ (2 \, \varepsilon + 3 \, \lambda)\, \Length(\Polygon).
\end{align*}
Combining these, we obtain $\nabs{\int_{\partial \Interval_\varepsilon} \!\Polygon_\varepsilon} \leq c\, \Length(\Polygon_\varepsilon)$,
where $c \ceq (\tfrac{1}{1+\eta} + 2\, \varepsilon + 3 \, \lambda) \,  (1 - 2 \, \varepsilon - 3\, \lambda)^{-1}$.
Choosing $\varepsilon < \frac{\eta }{5 \, (\eta^2+5\, \eta +4)}$ and supposing that $\lambda<\varepsilon$, we obtain $c \leq \frac{2}{2+\eta}$.
\newline
\textbf{Claim 1:} \emph{For each $v \in \Sphere$ there is an $\Edge_v \in \Edges_\varepsilon$ with
$
	\cos( \Angle{\EdgeVectors_\Polygon(\Edge_v)}{v})
	= \nabs{\ninnerprod{\EdgeVectors_\Polygon(\Edge_v),v}} 
	\leq \tfrac{2}{2+\eta}
$.}
\newline
\emph{Assume} that we have
$\cos( \Angle{\EdgeVectors_\Polygon(\Edge)}{v}) >  \tfrac{2}{2+\eta}$ for all $\Edge \in \Edges_\varepsilon$.
For each $\Edge \in \Edges_\varepsilon$, this implies that either $\Angle{\EdgeVectors_\Polygon(\Edge)}{v} < \beta$ or  $\Angle{\EdgeVectors_\Polygon(\Edge)}{v}> \uppi - \beta$ is fulfilled.
Notice the inequality
\begin{align*}
	\nabs{ \Angle{\VprevE{\EdgeVectors_\Polygon}(\Vertex)}{v} - \Angle{\VnextE{\EdgeVectors_\Polygon}(\Vertex)}{v}}
	\leq \Angle{\VprevE{\EdgeVectors_\Polygon}(\Vertex)}{\VnextE{\EdgeVectors_\Polygon}(\Vertex)}
	= \TurningAngles_\Polygon(\Vertex)
	< \uppi - 2 \, \beta
	\quad \text{for $\Vertex \in \IVertices(\Triangulation)$}.
\end{align*}
Thus, we have either
$\Angle{\EdgeVectors_\Polygon(\Edge)}{v} < \beta$ or  $\Angle{\EdgeVectors_\Polygon(\Edge)}{v}> \uppi - \beta$ for all $\Edge \in \Edges_\varepsilon$ \emph{simultaneously}.
By reversing the direction of $v$ if necessary, we may assume $\Angle{\EdgeVectors_\Polygon(\Edge)}{v} < \beta$ and obtain the \emph{contradiction}
\begin{align*}
	\nabs{\textstyle\int_{\bnd \Interval_\varepsilon} \Polygon_\varepsilon}
	\geq \nabs{\ninnerprod{\textstyle\int_{\bnd \Interval_\varepsilon} \Polygon_\varepsilon,v}}
	= 	
	\textstyle
	\sum_{\Edge \in \Edges_\varepsilon} \EdgeLengths_\Polygon(\Edge) \, \ninnerprod{\EdgeVectors_\Polygon(\Edge), v}
	> \tfrac{2}{2+\eta} \sum_{\Edge \in \Edges_\varepsilon} \EdgeLengths_\Polygon(\Edge)
	= \tfrac{2}{2+\eta} \Length(\Polygon_\varepsilon)
	\geq \nabs{\textstyle\int_{\bnd \Interval_\varepsilon} \Polygon_\varepsilon}.
\end{align*}
\textbf{Claim 2:} \emph{There is a constant $C \in \intervaloo{0,\infty}
$ such that 
$
	\ninnerprod{V,\varTheta_{\Polygon}V} \geq C^{-1} \, \nabs{V}^2
	$ holds for each $V \in \AmbSpace$.}
\newline
For $V = 0$, this is trivially true, so we suppose $V \neq 0$. With $v=V/\nabs{V}$, we have
\begin{align}
	\ninnerprod{V,\varTheta_{\Polygon} V} 
	&= \textstyle
	\sum_{\Edge \in \Edges(\Triangulation)}
		(1-\varphi_\Polygon(\Edge)) \, \varphi_\Polygon(\Edge)
		\, 
		\EdgeLengths_\Polygon(\Edge) \, \ninnerprod{V,\Discprnor(\Edge) \, V}
	\notag\\
	&
	\geq \textstyle \nabs{V}^2 \sum_{\Edge \in \Edges_\varepsilon}
		(1-\varphi_\Polygon(\Edge)) \, \varphi_\Polygon(\Edge)
		\, 
		\EdgeLengths_\Polygon(\Edge) \, (1 - \ninnerprod{\EdgeVectors_\Polygon(\Edge) , v}^2).
	\label{eq:ThetaSum}
\end{align}
Notice that $\set{(1-\varphi_\Polygon(\Edge)) \, \varphi_\Polygon(\Edge) | \Edge \in \Edges_\varepsilon}$ 
is uniformly bounded from below by a positive constant.
Moreover, the function $\Edge \mapsto 1 - \ninnerprod{\EdgeVectors_\Polygon(\Edge) , v}^2$ is nonnegative and enjoys a ``discrete Hölder-$(1-\frac{1}{p})$-continuity'' 
\begin{align*}
	\nabs{(1 - \ninnerprod{\VnextE{\EdgeVectors_\Polygon}(\Vertex) , v}^2)-(1 - \ninnerprod{\VprevE{\EdgeVectors_\Polygon}(\Vertex) , v}^2)} 
\leq 2\, \TurningAngles_\Polygon(\Vertex)
	\leq C \, \DualReferenceEdgeLengths(\Vertex)^{1-\frac{1}{p}}
	\quad
	\text{for $\Vertex \in \IVertices(\Triangulation)$}.
\end{align*}
By Claim 1, its maximum value is greater than $1 - \frac{2}{(2+\eta)} > 0$.
Thus it follows that the sum in \eqref{eq:ThetaSum} is bounded from below by a positive constant.
Finally, the statement of the lemma follows from the fact that $\varTheta_{\Polygon}$ is self-adjoint and from the Rayleigh-Ritz principle.
\end{proof}

%% file: DDPhi.tex
\section{Second Derivative of $\ConstraintMap$}

Existence and Lipschitz-continuity of the constraint map $\ConstraintMap$ is essential for the application of the Newton-Kantorovich theorem in the proof of \autoref{prop:RestorationOperator} in \autoref{sec:RestorationOperator}.
The analogous discrete result (which we do not prove here) is utilized for the same purpose in
\autoref{prop:DiscreteRestorationOperator} in \autoref{sec:DiscreteRestorationOperator}

\begin{lemma}\label{lem:SmoothDDPhiBound}	
Let $k \in \N$, $k \geq 2$, $p \in \intervalcc{1,\infty}$, $\StrainBound \geq 0$ and $\EnergyBound \geq 0$.
\begin{enumerate}
	\item \label{item:SmoothDDPhiBoundSobolev}
	There is a constant $C_{k,p} \geq 0$ such that
	the following holds true for each $\Curve \in \ImmC[k,p]$ satisfying 
$\nnorm{\LogStrain_\Curve}_{\Sobo[k-1,p][\Curve]}\leq \StrainBound$
and
$\nnorm{\Tangent_\Curve}_{\Sobo[k-1,p][\Curve]}\leq \EnergyBound$:
\begin{align*}
	\nnorm{D^2\ConstraintMap(\Curve)\,(u,v)}_{\Sobo[k-1,p]} \leq C_{k,p} \, \nnorm{u}_{\Sobo[k,p]} \, \nnorm{v}_{\Sobo[k,p]}
	\quad
	\text{for all $u$, $v \in \SoboC[k,p]$.}
\end{align*}
	\item \label{item:SmoothDDPhiBoundBV}
	There is a constant $C_{k} \geq 0$ such that
	the following holds true for each $\Curve \in \ImmC[1,\infty]$ satisfying 
$\nnorm{\LogStrain_\Curve}_{\TV[k-1][\Curve]}\leq \StrainBound$
and
$\nnorm{\Tangent_\Curve}_{\TV[k-1][\Curve]}\leq \EnergyBound$:
\begin{align*}
	\nnorm{D^2\ConstraintMap(\Curve)\,(u,v)}_{\TV[k-1]} \leq C_{k} \, \nnorm{u}_{\TV[k]} \, \nnorm{v}_{\TV[k]}
	\quad
	\text{for all $u$, $v \in \BVC[k]$.}
\end{align*}
\end{enumerate}
In particular, $\Curve \mapsto D\ConstraintMap (\Curve)$ is locally Lipschitz continuous with respect to these norms.
\end{lemma}
\begin{proof}
We restrict our attention the cases
$k = 2$ in Statement~\ref{item:SmoothDDPhiBoundSobolev} and
$k = 3$ in Statement~\ref{item:SmoothDDPhiBoundBV};
they are of primary interest to us.
The general case can be shown analogously by successively applying chain and product rules.

Suppose that 
$\nnorm{\LogStrain_\Curve}_{\Sobo[1,p][\Curve]}\leq \StrainBound$
and
$\nnorm{\Curve}_{\Sobo[2,p][\Curve]}\leq \EnergyBound$.
We define $w \ceq D(\Curve \mapsto \Dnor u)(\Curve) \, v$ and $y \ceq D(\Curve \mapsto \ninnerprod{\TangentC,\Dtot u})(\Curve) \, v$ so that we have
$
		D^2\ConstraintMap(\Curve)\,(u,v)
	= \nparen{
		0,
		0,
		w(0),
		w(L),
		y
	}
$.
Straight-forward calculations show that
\begin{align*}
	D(\Curve \mapsto \TangentC)(\Curve) \,v
	&= \Dnor v,
	\\
	D(\Curve \mapsto \Dtot u)(\Curve) \,v
	&= - \ninnerprod{\TangentC, \Dtot v} \, \Dtot u,
	\\
	y = D(\Curve \mapsto \ninnerprod{\TangentC,\Dtot u})(\Curve) \, v
	&=
	\ninnerprod{\Dnor u,\Dnor v}
	- \ninnerprod{\Dtan u, \Dtan v}, \qand
	\\
	w = D(\Curve \mapsto \Dnor u)(\Curve) \, v
	&= 
	- \ninnerprod{\TangentC, \Dtot v} \, \Dnor u
	- \Dnor v \, \ninnerprod{\TangentC,\Dtot u}
	- \TangentC \, \ninnerprod{\Dnor v,\Dnor u}.	
\end{align*}
Here, $\Dnor$ and $\Dtan$ denote the normal and the tangential derivative with respect to unit speed (see \autoref{sec:SmoothSetting}).
The Morrey embedding theorem yields
\begin{align*}
	\nnorm{y}_{L^\infty} \leq C \, \nnorm{u}_{\Sobo[2,p][\Curve]} \, \nnorm{v}_{\Sobo[2,p][\Curve]}
	\qand
	\nabs{w(0)}, \, \nabs{w(L)}
	\leq C \, \nnorm{u}_{\Sobo[2,p][\Curve]} \, \nnorm{v}_{\Sobo[2,p][\Curve]}.
\end{align*}
Further calculations lead to
\begin{align*}
	\Dtot y 
	&= 
	\ninnerprod{\Dtot\Dtot u,\Dtot v} + \ninnerprod{\Dtot u,\Dtot\Dtot v}
	- 2\, \ninnerprod{\Dtan\Dtan u, \Dtan v}	
	- 2\, \ninnerprod{\Dtan u, \Dtan\Dtan v}.
\end{align*}
Because of $\Dtan \TangentC = 0$, we obtain
$
	\Dtan \Dtan u
	= \TangentC \, \ninnerprod{\Dtot \TangentC, \Dtot u} + \TangentC \, \ninnerprod{\TangentC, \Dtot\Dtot u}	
$
and hence
\begin{align*}
	\seminorm{y}_{\Sobo[1,p][\Curve]}
	&\leq 
	3 \, \seminorm{u}_{\Sobo[2,p][\Curve]} \seminorm{v}_{\Sobo[1,\infty][\Curve]}
	+ 3 \, \seminorm{u}_{\Sobo[1,\infty][\Curve]} \seminorm{v}_{\Sobo[2,p][\Curve]}
	+ 4 \, \seminorm{\Curve}_{\Sobo[2,p][\Curve]} \seminorm{u}_{\Sobo[1,\infty][\Curve]}  \seminorm{v}_{\Sobo[1,\infty][\Curve]}
	\leq C \, \norm{u}_{\Sobo[2,p][\Curve]}  \norm{v}_{\Sobo[2,p][\Curve]}
	.
\end{align*}
Now, suppose additionally that 
$\nnorm{\LogStrain_\Curve}_{\TV[2][\Curve]}\leq \StrainBound$
and
$\nnorm{\Curve}_{\TV[3][\Curve]}\leq \EnergyBound$.
Then $\Dtot y$ is a sum of terms of the form $\varphi \, \psi$ with $\varphi \in \BVC[1]$ and $\psi \in \BVC[2]$.
Since multiplication in $\BV[1][][\Interval][\R]$ is continuous, we obtain
$\nseminorm{\varphi \, \psi}_{\TV[1][\Curve]} \leq C \, \nnorm{\varphi}_{\TV[1][\Curve]} \,  \nnorm{\psi}_{\TV[1][\Curve]}\leq C \, \nnorm{\varphi}_{\TV[1][\Curve]} \,  \nnorm{\psi}_{\TV[2][\Curve]}$
and thus
$\nseminorm{y}_{\TV[2][\Curve]} \leq C \, \nnorm{u}_{\TV[3][\Curve]} \, \nnorm{v}_{\TV[3][\Curve]}$.
Finally, the equivalence of norms (see \autoref{lem:NormEquivalences}) implies
Statement~\ref{item:SmoothDDPhiBoundSobolev} for $k=2$ and
Statement~\ref{item:SmoothDDPhiBoundBV} for $k=3$.
\end{proof}

%% file: Kantorovich.tex
\section{Newton--Kantorovich Theorem}\label{sec:NewtonKantotovich}

In \autoref{prop:RestorationOperator} and \autoref{prop:DiscreteRestorationOperator},
we repair the (small) constraint violations that arise by 
approximate reconstruction (\autoref{prop:ApproximateReconstructionTheorem} in \autoref{sec:ApproximateReconstruction}) 
and
approximate sampling (\autoref{prop:ApproximateSamplingTheorem} in \autoref{sec:ApproximateSampling}) 
with the help of the Newton-Kantorovich theorem. We use the following formulation of the theorem, which, along with a detailed proof, can be found as Theorem 7.7-3 in \cite{MR3136903}.

\begin{theorem}[Newton--Kantorovich Theorem]\label{theo:Kantorovich}
Let $X$, $Y$ be Banach spaces, $U \subset X$ an open set, $x_0 \in U$, and $F \in C^{1}(U;Y)$ such that $DF(x_0) \in L(X;Y)$ is continuously invertible. Assume that there are constants $\lambda$, $\mu$, $\nu$ such that
\begin{enumerate}
	\item $0< \lambda \, \mu \, \nu < \frac{1}{2}$ and $\OpenBall{x_0}{r} \subset U$,
	\item $\nnorm{DF(x_0)^{-1}F(x_0)}_X \leq \lambda$,
	\item $\nnorm{DF(x_0)^{-1}}_{L(X;Y)} \leq \mu$,
	\item $\nnorm{DF(\tilde x) - DF(x)}_{L(X;Y)} \leq \nu \, \nnorm{\tilde x- x}_X$ for all $\tilde x$, $x \in \OpenBall{x_0}{r}$,
\end{enumerate}
where
$r \ceq \frac{1}{\mu\,\nu}$
and
$r_{-} \ceq r \, \bigparen{1 - \sqrt{1-2 \, \lambda \, \mu \, \nu}} \leq 2\, \lambda$.

Then there is a point $a \in \ClosedBall{x_0}{r_-}$ with $F(a)=0$ and $a$ is the unique solution of $F(x)=0$ in $\OpenBall{x_0}{r}$.
Moreover, the point $a$ can be obtained as limit of the the Newton iterates $x_{n} \ceq x_{n-1}- DF(x_{n-1})^{-1} F(x_{n-1})$, $n \in \N$ and one has the error estimate
$
	\textstyle
	\norm{x_n-a}_X 
	\leq \frac{r}{2^n} \left( \frac{r_-}{r}\right)^{2^{n}}
	.
$
\end{theorem}

We transform this theorem into a variant that is more suitable for our purposes.
In a nutshell, the lemma below states that under moderate conditions, we find a solution $a$ of the equation $F(a)=b$ whenever the right hand side $b$ is not too far away from the value of $F(x_0)$ at a given starting point $x_0$. Moreover, the distance of this solution $a$ to the point $x_0$ is controlled by the deviation of $b$ from $F(x_0)$.

\begin{lemma}\label{lem:KantorovichVarian}
Let $X$, $Y$ be Banach spaces, $U \subset X$ an open, convex set, $x_0 \in U$, and $F \in C^{1}(U;Y)$.
Suppose that there are constants $\mu \geq 0$, $\nu \geq 0$, and $0<r < \frac{1}{\mu \, \nu}$ with the following properties:
\begin{enumerate}
	\item There is a bounded, linear right inverse $R \colon Y \to X$ of $DF(x_0)$ with $\norm{R} \leq \mu$.
	\item The slice $S \ceq \ClosedBall{x_0}{r} \cap (x_0 + \Ima(R))$ is contained in $U$,
	where $\Ima(R)$ denotes the image of the operator $R$. 
	\item $\norm{DF(\tilde x)-DF(x)} \leq \nu \, \norm{\tilde x-x}$ for all $\tilde x$, $x \in S$.
\end{enumerate}
Then with $\varepsilon \ceq \frac{r}{2 \,\mu}$, the following statements hold true:
\begin{enumerate}
	\item The image of the ball $\ClosedBall{x_0}{r}$ under $F$ contains the ball $\ClosedBall{F(x_0)}{\varepsilon}$, i.e., for each 
	$b \in Y$ with $\nnorm{F(x_0) - b} < \varepsilon$ there is an $a \in X$ with $\nnorm{a-x_0} \leq r$ with $F(a) =b$.
	\item One has the estimate $\norm{a-x_0} \leq 2\, \mu \, \norm{F(x_0)-b}$.	
\end{enumerate}
\end{lemma}
\begin{proof}
Notice that $P \ceq R \, DF(x_0)$ is a continuous projector onto $Z \ceq \Ima(R)$. Thus, $Z \subset X$ is a closed subspace and hence a Banach space.
Define the mapping $G \colon S \to Y$ by $G(x) \ceq F(x)-b$.
Observe that $DG(x)\,w = DF(x)\,w$ holds true for all $x \in S$ and all $w \in Z$.
In particular, we have $DG(x_0) \, R = DF(x_0) \, R = \id_Y$ and
$R \, DG(x_0) \, w = R \,  DF(x_0) \, w = w$ for all $w \in Z$.
Thus, $DG(x_0) \colon Z \to Y$ is continuously invertible and we have $\nnorm{DG(x_0)^{-1}} = \nnorm{R} \leq \mu$.
Moreover, it follows that
$
	\nnorm{DG(\tilde x) - DG(x)}	\leq \nu \,\nnorm{\tilde x - x}
$,
for all $\tilde x$, $x \in S$.
For $b \in \ClosedBall{F(x_0)}{\varepsilon}$, we deduce the bound
$
	\lambda \ceq \nnorm{DG(x_0)^{-1}G(x_0)} = \nnorm{R\, (F(x_0)-b)} 
	\leq  \mu \, \nnorm{F(x_0)-b}
	< \tfrac{1}{2 \, \mu \, \nu}.
$
Thus, we may apply the Newton--Kantorovich theorem \autoref{theo:Kantorovich} to the mapping $G$:
We find a solution $a \in \ClosedBall{x_0}{r_-}$ 
where
$r_{-} \ceq r \, \bigparen{1 - \sqrt{1-2 \, \lambda \, \mu \, \nu}} \leq 2\, \lambda \leq 2 \, \mu \, \nnorm{F(x_0)-b}$.
\end{proof}

%% file: SamplingEstimates.tex
\section{Sampling Estimates}\label{sec:SamplingEstimates}

This is complementary material for \autoref{prop:ApproximateSamplingTheorem} in \autoref{sec:ApproximateSampling}.
In order to streamline the following exposition, we introduce the \emph{signed distance} $\SignedDistC{a}{b} \ceq \int_a^b \LineElementC$ and the \emph{unsigned distance} $\UnsignedDistC{a}{b} \ceq \nabs{\SignedDistC{a}{b}}$
of two points $a$, $b \in \Interval$ with respect to an immersed curve $\Curve \in \ImmC$.

The following lemma provides us with estimates on the deviation of secant lengths from arc lengths.

\begin{blemma}[Length Distortion Estimates]\label{lem:LengthDistortion}
Let $p \in \intervalcc{1,\infty}$, and $\EnergyBound \geq 0$.
There are constants $C_{p}>0 $ and $\delta_{p}>0$ such that the following inequalities
hold true for each curve $\Curve \in \ImmC[2,p]$ with 
$\nnorm{\Tangent_\Curve}_{\Sobo[1,p][\Curve]} \leq \EnergyBound$ 
and for all compact intervals $\intervalcc{a,b} \subset \Interval$ with 
$\UnsignedDistC{a}{b}\leq \delta_{p}$:
\begin{align*}
	\bigabs{ 
		1-
		\tfrac{\nabs{\Curve(b)-\Curve(a)}}{\UnsignedDistC{a}{b}}
	} \leq C_{p} \, \nnorm{\CurvatureC}_{L^p}^2 \, \UnsignedDistC{a}{b}^{2-\frac{1}{p}}
	\qand
	\bigabs{ 
		1
		-
		\tfrac{\UnsignedDistC{a}{b}}{\nabs{\Curve(b)-\Curve(a)}}
	} \leq C_{p} \, \nnorm{\CurvatureC}_{L^p}^2  \, \UnsignedDistC{a}{b}^{2-\frac{1}{p}}.
	\label{item:LengthDistortion}
\end{align*}
\end{blemma}
\begin{proof}
We apply the fundamental theorem of calculus $f(b) = f(a) + \int_a^b \Dtot f \, \LineElement_\Curve$ repeatedly and exploit that $\ninnerprod{\TangentC(r),\TangentC(r)} = 1$ and $\ninnerprod{\TangentC(r),\CurvatureC(r)}=0$ hold true. This way, we obtain
\begin{align*}
	\MoveEqLeft
	\ninnerprod{\TangentC(r),\TangentC(s)}
	= \textstyle
	\ninnerprod{\TangentC(r),\TangentC(r)} + 
	\int_r^s \ninnerprod{\TangentC(r),\CurvatureC(t)}\, \LineElementC[t]
	\notag\\
	&= \textstyle
	1 
	+ \int_r^s \cancel{\ninnerprod{\TangentC(t),\CurvatureC(t)}}\, \LineElementC[t]	
	- \int_r^s\!\!\int_r^t \ninnerprod{\CurvatureC(u),\CurvatureC(t)} \, \LineElementC[u]		 \, \LineElementC[t].
\end{align*}
This leads us to
\begin{align*}
	\nabs{\Curve(b)-\Curve(a)}^2
	&=\textstyle
	\int_a^b \!\! \int_a^b \ninnerprod{\TangentC(r),\TangentC(s)} \, \LineElement_\Curve(s)\, \LineElement_\Curve(t)
	\\
	&=\textstyle
	\UnsignedDistC{a}{b}^2
	-
	\int_a^b\!\!\int_a^b\!\!\int_r^s \!\!\int_r^t \ninnerprod{\CurvatureC(u),\CurvatureC(t)} \, \LineElement_\Curve(u)\, \LineElement_\Curve(t) \, \LineElement_\Curve(s) \, \LineElement_\Curve(r)
\end{align*}
and by Hölder's inequality, the latter integral is bounded by $\frac{2^{-1+1/p}\, p^2}{(2 p-1) \,(3 p-2)}
	\,
	\nnorm{\CurvatureC}_{L^p}^2
	\,
	d_\Curve(a,b)^{4-2/p}$
so that we obtain
\begin{align*}
	\textstyle
	\abs{
	1
	-
	\paren{
		\frac{\nabs{\Curve(b)-\Curve(a)}}{\UnsignedDistC{a}{b}}
	}^2
	}
	\leq 
	C_p  \,  \nnorm{\CurvatureC}_{L^p}^2 \, d_\Curve(a,b)^{2-2/p}.
\end{align*}
Now the statements of the lemma follow from applying the estimates
$\nabs{1-\sqrt{1-x}} \leq \frac{x}{2} + \frac{x^2}{4}$
and
$\nabs{1-1/\sqrt{1-x}} \leq \frac{x}{2} + x^2$ (both true for $x\in \intervalcc{0,1/2}$)
to
applied to $x = 1 - \paren{\tfrac{\nabs{\Curve(b)-\Curve(a)}}{\UnsignedDistC{a}{b}}}^2$.
\end{proof}

What follows is in the tradition of classical results on the consistency of second order finite differences. The only twist here is that we divide by secant lengths instead of arc lengths; it is basically the second order consistency of the secant lengths (which we have just shown) that allows us to perform this replacement.

\begin{lemma}[Consistency of Second  Derivatives]\label{lem:SecondDerivativeEstimate}
Let $p \in \intervalcc{1,\infty}$ and $\EnergyBound \geq 0$.
There are constants $C>0 $ and $\delta>0$ such that the following
holds true 
for each curve $\Curve \in \ImmC[3,p]$ with 
$\nnorm{\Curvature_\Curve}_{\Sobo[1,p]} \leq \EnergyBound$,
for each partition $\Triangulation$ of $\Interval$ satisfying $\MaxRadius(\Triangulation) \leq \delta$, and 
for each function $f \in \Sobo[3,p][][\Interval][\R]$:

Then one has
for $\Polygon(\Vertex) \ceq \Curve(\Vertex)$ and $F(\Vertex) \ceq f(\Vertex)$, $\Vertex \in \Vertices(\Triangulation)$ that
\begin{align*}
	\nabs{(\DiscD^2 F)(\Vertex) - (\Dtot^2 f)(\Vertex)}
	\leq 
	C \, \nnorm{\Dtot f}_{\Sobo[2,p][\Curve]} \, \UnsignedDistC{\VprevV{\Vertex}}{\VnextV{\Vertex}}^{1-1/p}
	\quad \text{for each $\Vertex \in \IVertices(\Triangulation)$.}
\end{align*}
\end{lemma}
\begin{proof}
Fix $\Vertex \in \IVertices(\Triangulation)$ and
abbreviate $a \ceq \VprevV{\Vertex}$,
$b \ceq \VthisV{\Vertex}$, and $c \ceq \VnextV{\Vertex}$.
This way, we have 
\begin{align*}
	\DiscD^2 F(\Vertex) 
	= \tfrac{2}{\nabs{\Curve(b)-\Curve(a)}+\nabs{\Curve(c)-\Curve(b)}} \paren{
		\tfrac{f(c)-f(b)}{\nabs{\Curve(c)-\Curve(b)}}
		-
		\tfrac{f(b)-f(a)}{\nabs{\Curve(b)-\Curve(a)}}
	}.
\end{align*}
Taylor's formula with remainder in integral form implies
\begin{align*}
	\tfrac{f(c)-f(b)}{\SignedDistC{b}{c}}
	&= 
	\textstyle	
	\Dtot f(b) 
	+ \tfrac{\SignedDistC{b}{c} }{2} \Dtot^2 f(b)
	+ \frac{1}{\SignedDistC{b}{c}} \int_b^c
		\tfrac{\SignedDistC{t}{c}^2}{2} \Dtot^3 f(t) \, \LineElementC(t)
	\qand
	\\
	\tfrac{f(b)-f(a)}{\SignedDistC{a}{b}}
	&= 
	\textstyle	
	\Dtot f(b) 
	- \tfrac{\SignedDistC{a}{b} }{2} \Dtot^2 f(b)
	+ \frac{1}{\SignedDistC{a}{b}} \int_a^b
		\tfrac{\SignedDistC{a}{t}^2}{2} \Dtot^3 f(t) \, \LineElementC(t).
\end{align*}
Taking differences leads to
\begin{align*}
	\abs{
	\tfrac{2}{\SignedDistC{a}{c}}
	\paren{
	\tfrac{f(c)-f(b)}{\SignedDistC{b}{c}}
	-
	\tfrac{f(b)-f(a)}{\SignedDistC{a}{b}}
	}
	-
	\Dtot^2 f(b)
	}
	&\leq
	\tfrac{1}{2}\, \seminorm{f}_{\Sobo[3,p][\Curve]}\, \UnsignedDistC{a}{c}^{1-1/p}.
\end{align*}
By \autoref{lem:LengthDistortion}, we have
\begin{align*}
	\MoveEqLeft
	\abs{
	\paren{
		\tfrac{f(c)-f(b)}{\nabs{\Curve(c)-\Curve(b)}}
		-
		\tfrac{f(b)-f(a)}{\nabs{\Curve(b)-\Curve(a)}}
	}
	-
	\paren{
		\tfrac{f(c)-f(b)}{\SignedDistC{b}{c}}
		-
		\tfrac{f(b)-f(a)}{\SignedDistC{a}{b}}
	}
	}
	\\
	&=
	\abs{
		\tfrac{f(b)-f(a)}{\SignedDistC{a}{b}}
		\,
		\Bigparen{
			\tfrac{\SignedDistC{a}{b}}{\nabs{\Curve(b)-\Curve(a)}}
			-
			1
		}
	}
	+
	\abs{	
		\tfrac{f(c)-f(b)}{\SignedDistC{b}{c}}
		\,
		\Bigparen{
			\tfrac{\SignedDistC{b}{c}}{\nabs{\Curve(c)-\Curve(b)}}
			-
			1
		}
	}
	\leq
	C\, \nseminorm{f}_{\Sobo[1,\infty][\Curve]}\, \UnsignedDistCp[2-1/p][a][c]
\end{align*}
Finally, by combining this and using \autoref{lem:LengthDistortion} once more, we obtain:
\begin{align*}
	\nabs{\DiscD^2 F(\Vertex) - \Dtot^2 f(\Vertex)}
	\leq  C \, \nnorm{\Dtot f}_{\Sobo[2,p][\Curve]} \, \UnsignedDistC{\shiftll{\Vertex}}{\shiftrr{\Vertex}}^{1-1/p}.
\end{align*}
\end{proof}